\documentclass[12pt]{article}
\usepackage{amsthm,amsfonts,amsmath,amscd,amssymb,epsfig}
\usepackage[left=2cm, right=2cm]{geometry}
\usepackage{epigraph} 
\numberwithin{equation}{section}
\theoremstyle{plain}
\newtheorem{theorem}{Theorem}[section]

\newtheorem{proposition}[theorem]{Proposition}
\newtheorem{corollary}[theorem]{Corollary}
\newtheorem{lemma}[theorem]{Lemma}
\newtheorem{lemma-definition}[theorem]{Lemma-Definition}

\newtheorem{conjecture}[theorem]{Conjecture}
\newtheorem{claim}[theorem]{Claim}

\theoremstyle{definition}
\newtheorem{definition}[theorem]{Definition}

\newtheorem{remark}[theorem]{Remark}

\newtheorem{example}[theorem]{Example}
\newtheorem{counterexample}[theorem]{Counterexample}

\newcommand{\Addresses}{{
  \bigskip
\footnotesize
\textsc{School of Mathematics, Institute for Advanced Study, Princeton, NJ 08540, USA}\par\nopagebreak
\textit{E-mail address:} \texttt{shahmath19@gmail.com, sfaisal@ias.edu}}}
\graphicspath{{Images/}}
\begin{document}
\title{Extremal Lagrangian tori in toric domains}
\author{Shah Faisal}
\date{}
\maketitle
\begin{abstract}
Let $L$ be a closed Lagrangian submanifold of a symplectic manifold $(X,\omega)$. Cieliebak and Mohnke define the symplectic area of $L$ as the minimal positive symplectic area of a smooth $2$-disk in $X$ with boundary on $L$. An extremal Lagrangian torus in $(X,\omega)$  is a Lagrangian torus that maximizes the symplectic area among the Lagrangian tori in  $(X,\omega)$. We prove that every extremal Lagrangian torus in the symplectic unit ball $(\bar{B}^{2n}(1),\omega_{\mathrm{std}})$ is contained entirely in the boundary $\partial B^{2n}(1)$. This answers a question attributed to Lazzarini and completely settles a conjecture of Cieliebak and Mohnke in the affirmative. In addition, we prove the conjecture for a class of toric domains in $(\mathbb{C}^n, \omega_{\mathrm{std}})$, which includes all compact strictly convex four-dimensional toric domains. We explain with counterexamples that the general conjecture does not hold for non-convex domains. \end{abstract}	
\tableofcontents
\parskip=4pt
\section{Introduction and statements of results}
\subsection{Introduction}
A symplectic manifold is a pair $(X,\omega)$, where $X$ is a smooth $2n$-dimensional manifold and $\omega$ is a closed differential $2$-form which is non-degenerate in the sense that the top-degree form $\omega^n$ is nowhere vanishing. Important examples of symplectic manifolds are convex toric domains in the standard symplectic Euclidean space $(\mathbb{C}^{n}, \omega_{\mathrm{std}}:=\frac{i}{2}\sum_{j=1}^ndz_j\wedge d\bar{z}_j)$. These are defined as follows. Consider the standard moment map $\mu:\mathbb{C}^{n}\to \mathbb{R}^{n}_{\geq 0}$ given by
\[\mu(z_1,z_2,\dots, z_n):= \pi(|z_1|^2,\dots, |z_n|^2).\]
A toric domain in $\mathbb{C}^{n}$ is a subset $X^{2n}_{\Omega}$ of $\mathbb{C}^{n}$ that can be written as $X^{2n}_{\Omega}=\mu^{-1}(\Omega)$, for some domain $\Omega \subset \mathbb{R}^{n}_{\geq 0}$. A convex toric domain is a toric domain $X^{2n}_{\Omega}$ for which the set
\[\widehat{\Omega}:=\{(x_1,x_2,\dots,x_n)\in \mathbb{R}^n: (|x_1|,|x_2|,\dots,|x_n|)\in \Omega\}\] 
is a convex subset of $\mathbb{R}^{n}$. We define the diagonal of a toric domain $X^{2n}_{\Omega}$ to be the number defined by
\[\operatorname{diagonal}(X^{2n}_{\Omega}):=\sup\{a>0:(a,a,\dots,a)\in \Omega\}.\]

Convex toric domains inherit the standard symplectic structure from  $\mathbb{C}^{n}$; this way, they form an essential class of symplectic manifolds with boundary. 

For example, let $\Omega$ be the $n$-simplex in $\mathbb{R}^{n}_{\geq 0}$ with vertices $(0,\dots,0), (a_1,\dots,0),$ $(0,a_2, \dots,0)$ and $(0, \dots,a_n)$, where $0<a_1\leq a_2\leq a_3\leq \dots\leq a_n<\infty$. Then $X^{2n}_{\Omega}$ is the standard $2n$-dimensional ellipsoid defined by
\[E^{2n}(a_1,a_2,a_3,\dots,a_n):=\bigg\{(z_1,\dots, z_{n})\in \mathbb{C}^{n}: \sum_{i=1}^{n}\frac{  \pi|z_i|^2}{a_i}\leq 1\bigg \}.\] 
In particular,
\[\bar{B}^{2n}(a):=E^{2n}(a,a,a,\dots,a)\] 
is the standard ball of capacity $a>0$ and radius $\sqrt{a/\pi}$. Moreover, $S^{2n-1}(a):=\partial \bar{B}^{2n}(a)$ is the sphere of radius $\sqrt{a/\pi}$ centered at the origin. In particular, $S^{1}(a)$ is the circle centered at the origin that bounds a symplectic area equal to $a$. It is easy to see that 
\[\operatorname{diagonal}(E^{2n}(a_1,a_2,a_3,\dots,a_n))= \left( \frac{1}{a_1}+\dots+\frac{1}{a_n}\right)^{-1},\]
and 
\[\operatorname{diagonal}(\bar{B}^{2n}(a))= \frac{a}{n}.\]

 A submanifold $L\subset (X,\omega)$ is called Lagrangian if $2\operatorname{dim}(L)=\operatorname{dim}(X)$ and $i^*\omega=0$, where $i: L\to X$ is the inclusion. In this article, our focus is on Lagrangian tori; a Lagrangian torus is a Lagrangian submanifold that is diffeomorphic to the standard torus. A Lagrangian torus $L\subset (X,\omega)$ is monotone if there exists $K>0$ such that for all $A\in H_2(X,L)$ we have
 \[\int_{A}\omega=K\mu_L(A),\]
 where $\mu_L:H_2(X,L)\to \mathbb{Z}$ denotes the Maslov class of $L$.
\begin{definition}[Cieliebak--Mohnke \cite{Cieliebak2018}]
Let $L\subset (X,\omega)$ be a Lagrangian torus. The quantity 
\[\inf_{A\in \pi_2(X, L)\atop \int_{A}\omega>0} \int_{A}\omega\in [0, \infty]\]
is called the symplectic area  of $L$ in $ (X,\omega)$.
\end{definition} 
\begin{definition}[Cieliebak--Mohnke \cite{Cieliebak2018}]
The Lagrangian capacity of the symplectic manifold $(X,\omega)$, denoted by $\mathrm{c}_{L}(X,\omega)$, is the maximal symplectic area of a Lagrangian torus sitting in $ (X,\omega)$, i.e.,
\[\mathrm{c}_{L}(X,\omega):=\sup_{
	L\subset (X,\omega)\atop
	\text{Lag. torus}} \inf_{A\in \pi_2(X, L)\atop \int_{A}\omega>0} \int_{A}\omega \in  [0, \infty].\]
\end{definition}
\begin{definition}[Cieliebak--Mohnke \cite{Cieliebak2018}]
A Lagrangian torus $L$ in  $(X,\omega)$ is called extremal if  
\[\inf_{A\in \pi_2(X, L)\atop \int_{A}\omega>0} \int_{A}\omega=\mathrm{c}_{L}(X,\omega).\]
\end{definition}	
\begin{theorem}[Cieliebak--Mohnke \cite{Cieliebak2018}]\label{compute-lag}
 Let $\omega_{\mathrm{FS}}$ be the Fubini--Study form on  $\mathbb{CP}^n$ scaled so that $\int_{\mathbb{CP}^1}\omega_{\mathrm{FS}}=1$. Then
\[\mathrm{c}_{L}(\mathbb{CP}^n,\omega_{\mathrm{FS}})=\frac{1}{n+1}.\]
 For the standard closed ball of capacity $r>0$ (and radius $\sqrt{r/\pi}$), we have
	\[\mathrm{c}_{L}(\bar{B}^{2n}(r),\omega_{\mathrm{std}})=\frac{r}{n}.\]
\end{theorem} 
Let $X^{2n}_{\Omega}$ be a toric domain and $\delta=\operatorname{diagonal}(X^{2n}_{\Omega})$. The Clifford torus $S^1(\delta)\times\dots\times S^1(\delta)$ in $(\mathbb{C}^n,\omega_{\mathrm{std}})$ is a Lagrangian torus of symplectic area $\delta$  that sits on the boundary $\partial X^{2n}_{\Omega}$. Therefore, every extremal Lagrangian torus in $(X^{2n}_{\Omega},\omega_{\mathrm{std}})$ has symplectic area at least $\operatorname{diagonal}(X^{2n}_{\Omega})$. In other words,
 \begin{equation}\label{lag-capacity-con}
 \mathrm{c}_{L}(X^{2n}_{\Omega},\omega_{\mathrm{std}})\geq \operatorname{diagonal}(X^{2n}_{\Omega}).
 \end{equation}
In particular, 
\begin{equation}\label{lag-capacity}
	\mathrm{c}_{L}(E^{2n}(a_1,a_2,a_3,\dots,a_n),\omega_{\mathrm{std}})\geq \left( \frac{1}{a_1}+\dots+\frac{1}{a_n}\right)^{-1}.
\end{equation}

Theorem \ref{compute-lag} implies that we have equality in (\ref{lag-capacity}) for the round balls  $E(r,r,\dots,r)=\bar{B}^{2n}(r)$. Pereira \cite{Pereira:2022ab, Pereira:2022aa} has shown that (\ref{lag-capacity}) is an equality for all four-dimensional ellipsoids, whereas for a higher dimensional ellipsoid equality holds provided that a suitable virtual perturbation scheme exists to define the curve counts for its linearized contact homology.

Cieliebak and Mohnke have stated two conjectures about extremal Lagrangian tori in their influential $2014$ paper \cite{Cieliebak2018}. The first conjecture claims that the symplectic area of a Lagrangian torus remembers information about the extrinsic geometry of the Lagrangian torus:
\begin{conjecture}[Cieliebak--Mohnke {\cite[Conjecture 1.9]{Cieliebak2018}}]\label{extremal-lang1}
	Every extremal Lagrangian torus in the standard symplectic unit ball $(\bar{B}^{2n}(1),\omega_{\mathrm{std}})$  lies entirely on the boundary $\partial B^{2n}(1)$. 
\end{conjecture}

Let  $L$ be a Lagrangian torus in  $(\bar{B}^{2n}(1), \omega_{\mathrm{std}})$ that lies entirely on the boundary $\partial B^{2n}(1)$. To test the validity of Conjecture \ref{extremal-lang1}, one can try pushing a part (or whole) of the Lagrangian $L$ into the interior of the ball via a Hamiltonian isotopy without making $L$ exit the ball $\bar{B}^{2n}(1)$. However, such a Hamiltonian isotopy does not exist. Hopf-fibers foliate the Lagrangian $L$ because the standard Reeb vector field is tangent to it. From the proof of {\cite[Proposition B.1]{Cieliebak2018}}, it follows that any Hamiltonian isotopy of $L$ that does not make $L$ exit $\bar{B}^{2n}(1)$ keeps all Hopf-fibers on $L$ in $\partial \bar{B}^{2n}(1)$---hence keeps $L$ on $\partial \bar{B}^{2n}(1)$.

 Based on the intersection theory of holomorphic curves in dimension four,
	Dimitroglou Rizell \cite{DimitroglouRizell1} has given a proof of Conjecture \ref{extremal-lang1} for the four-dimensional ball $(\bar{B}^{4}(1),\omega_{\mathrm{std}})$. In this article, we prove this conjecture in the affirmative in all dimensions (cf. Theorem \ref{extremal-lag-ball}).  Our proof does not rely on intersection theory.

The second conjecture claims that extremal Lagrangian tori offer a replacement for monotone tori, which exist only in monotone symplectic manifolds.

\begin{conjecture}[Cieliebak--Mohnke {\cite[Conjecture 1.8]{Cieliebak2018}}]\label{monotone-extremal-conjecture}
	Let  $\omega_{\mathrm{FS}}$ be the Fubini--Study form on $\mathbb{CP}^n$. A Lagrangian torus in $(\mathbb{CP}^n,\omega_{\mathrm{FS}})$ is monotone if and only if it is extremal.
\end{conjecture}
Conjecture \ref{monotone-extremal-conjecture} holds for $\mathbb{CP}^2$. Using the works of Cieliebak--Mohnke \cite{Cieliebak2018}, Dimitroglou Rizell--Goodman--Ivrii \cite{DimitroglouRizell2}, and Hind--Opshtein \cite{Hind},  we provide a proof in the next section (cf. Theorem \ref{monoton=extremal}). 

\subsection{Statements of results}
We prove the following, hence answering Conjecture \ref{extremal-lang1}.
\begin{theorem}\label{extremal-lag-ball}
	For all $n\in \mathbb{N}$, every extremal Lagrangian torus in $(\bar{B}^{2n}(1),\omega_{\mathrm{std}})$  lies entirely on the boundary $\partial B^{2n}(1)$.\end{theorem}

By Theorem \ref{compute-lag}, an extremal Lagrangian torus in the standard $2n$-dimensional ball of capacity $r>0$ has symplectic area equal to $r/n$. So, by Theorem \ref{extremal-lag-ball}, a Lagrangian torus of symplectic area $A_{\mathrm{min}}(L)>0$ in $(\mathbb{C}^n,\omega_{\mathrm{std}})$ cannot be placed in a $2n$-dimensional ball of capacity $nA_{\mathrm{min}}(L)$. Figure \ref{not-allowed} depicts how a Lagrangian torus $L$ of positive symplectic area $A_{\mathrm{min}}(L)$ can sit in $(\mathbb{C}^n,\omega_{\mathrm{std}})$ according to Theorem \ref{extremal-lag-ball}.
\begin{figure}[h]
	\centering
	\includegraphics[width=4cm]{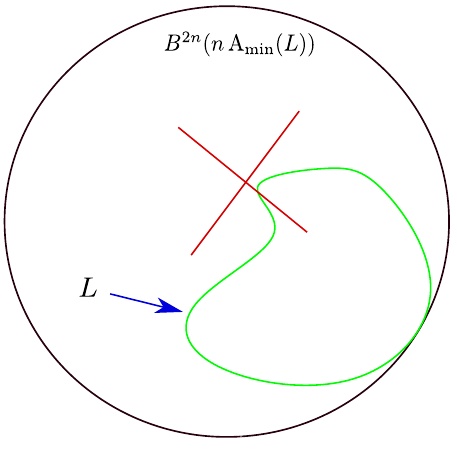}
	\includegraphics[width=4cm]{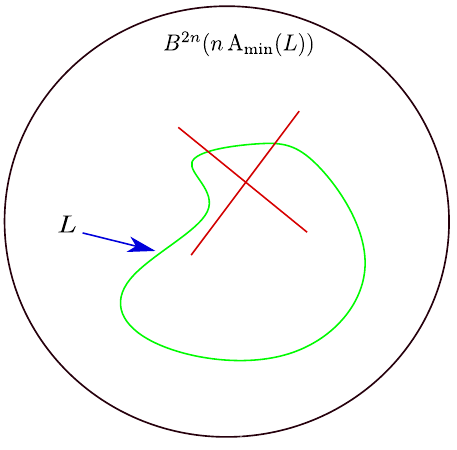}
	\includegraphics[width=4.62cm]{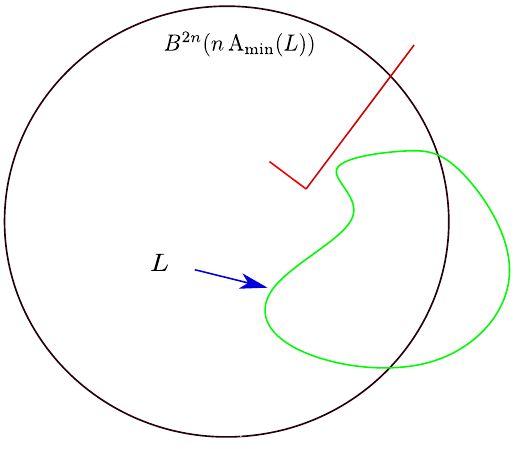}
	\includegraphics[width=4cm]{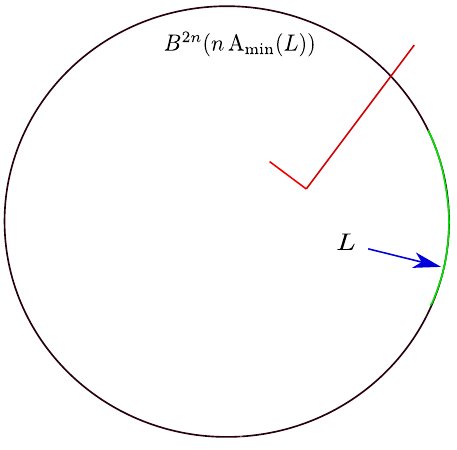}
	\caption{The crossed out configurations on the left are not allowed by Theorem \ref{extremal-lag-ball}}\label{not-allowed}
\end{figure}

The following theorem generalizes Theorem \ref{extremal-lag-ball}.
\begin{theorem}\label{extremal-lag-cylinder}
	For any $k,m \in \mathbb{Z}_{\geq 0}$, we have
	\[\mathrm{c}_{L}(B^{2k}(1)\times \mathbb{C}^m,\omega_{\mathrm{std}})=\mathrm{c}_{L}(B^{2k}(1),\omega_{\mathrm{std}}).\]
	Moreover, every extremal Lagrangian torus in the cylinder $(B^{2k}(1)\times \mathbb{C}^m,\omega_{\mathrm{std}})$  lies entirely on the boundary $\partial B^{2k}(1)\times \mathbb{C}^m$.
\end{theorem}
\begin{corollary}\label{cora-extemal}
	Let $m\in \mathbb{N}$ and $0<1\leq r_1\leq\dots \leq r_m\leq \infty$. Every extremal Lagrangian torus $L$ in the polydisk $\big(\bar{B}^{2}(1)\times \bar{B}^{2}(r_1)\times\dots\times \bar{B}^{2}(r_m),\omega_{\mathrm{std}}\big)$ lies on the boundary. More precisely, 
\[L\subset \partial \bar{B}^{2}(1)\times \bar{B}^{2}(r_1)\times\dots\times \bar{B}^{2}(r_m).\]
\end{corollary}
\begin{proof}[Proof of Corollary \ref{cora-extemal}]
We have the symplectic inclusion
\[(\bar{B}^{2}(1)\times \bar{B}^{2}(r_1)\times\dots\times \bar{B}^{2}(r_m),\omega_{\mathrm{std}})\to (\bar{B}^{2}(1)\times \mathbb{C}^m,\omega_{\mathrm{std}}).\]
The monotonicity property of the Lagrangian capacity {\cite[Section 1.2]{Cieliebak2018}}, Theorem \ref{extremal-lag-cylinder},  and Theorem \ref{compute-lag} imply
\[\mathrm{c}_{L}(\bar{B}^{2}(1)\times \bar{B}^{2}(r_1)\times\dots\times \bar{B}^{2}(r_m),\omega_{\mathrm{std}})\leq \mathrm{c}_{L}(\bar{B}^{2}(1)\times \mathbb{C}^m,\omega_{\mathrm{std}})=1.\]
On the other hand, the torus $\overbrace{S^1(1)\times \cdots \times S^1(1)}^\text{ $m+1$ times }$ has symplectic area equal to $1$ and is contained in  $\bar{B}^{2}(1)\times \bar{B}^{2}(r_1)\times\dots\times \bar{B}^{2}(r_m)$. Therefore, 
\[\mathrm{c}_{L}(\bar{B}^{2}(1)\times \bar{B}^{2}(r_1)\times\dots\times \bar{B}^{2}(r_m),\omega_{\mathrm{std}})=1.\]

If $L$ is an extremal Lagrangian torus in $\big(\bar{B}^{2}(1)\times \bar{B}^{2}(r_1)\times\dots\times \bar{B}^{2}(r_m),\omega_{\mathrm{std}}\big)$, then $L$ is also extremal in $(\bar{B}^{2}(1)\times \mathbb{C}^m,\omega_{\mathrm{std}})$ by Theorem \ref{extremal-lag-cylinder} and, moreover, we must have $L\subset \partial \bar{B}^{2}(1)\times \mathbb{C}^m$. We conclude that
\[L\subset \partial \bar{B}^{2}(1)\times \bar{B}^{2}(r_2)\times\cdots\times \bar{B}^{2}(r_m).\qedhere\]
\end{proof}
Let $\omega_{\mathrm{FS}}$ be the Fubini--Study form on $\mathbb{CP}^n$ scaled so that $\int_{\mathbb{CP}^1}\omega_{\mathrm{FS}}=1$. Recall that by Theorem \ref{compute-lag} an extremal Lagrangain torus in $(\mathbb{CP}^{n-1},\omega_{\mathrm{FS}})$ has symplectic area $1/n$.
\begin{corollary}\label{1-1coresp}
	\begin{itemize}
		\item [(1)]  There is a one-to-one correspondence between Hamiltonian isotopy classes of extremal Lagrangian tori in $(\bar{B}^{2n}(1),\omega_{\mathrm{std}})$ and the Hamiltonian isotopy classes of extremal (monotone for $n=3$ by Theorem \ref{monoton=extremal}) Lagrangian tori in $(\mathbb{CP}^{n-1},\omega_{\mathrm{FS}})$. 
		
		\item [(2)] For each $n\geq 3$, there exist infinitely many pairwise distinct Hamiltonian isotopy classes of extremal Lagrangian tori in $(\bar{B}^{2n}(1),\omega_{\mathrm{std}})$.
		\item [(3)] {\cite[Corollary 1.3]{DimitroglouRizell1}} Up to  Hamiltonian isotopy of $(\bar{B}^{4}(1),\omega_{\mathrm{std}})$ preserving the boundary set-wise, there is a unique extremal Lagrangian torus in  $(\bar{B}^{4}(1),\omega_{\mathrm{std}})$, namely the Clifford torus $S^1(\frac{1}{2})\times S^1(\frac{1}{2})$. 
	\end{itemize}
\end{corollary}
\begin{proof}[Proof of Corollary \ref{1-1coresp}]
	By Theorem \ref{extremal-lag-ball}, all extremal Lagrangian tori in $(\bar{B}^{2n}(1),\omega_{\mathrm{std}})$ are entirely contained in the boundary $S^{2n-1}(1)$. Since $L$ is Lagrangian, the standard Reeb vector field on $S^{2n-1}(1)$ is tangent to $L$. Therefore, $L$ is invariant under the $S^1$-action induced by the Reeb flow. Under the Hopf fibration $P:S^{2n-1}(1)\to \mathbb{CP}^{n-1}$, extremal Lagrangian tori in $(\bar{B}^{2n}(1),\omega_{\mathrm{std}})$ descend to extremal Lagrangian tori in $(\mathbb{CP}^{n-1},\omega_{\mathrm{FS}})$ and vice versa. Moreover, if Conjecture \ref{monotone-extremal-conjecture} holds, then there is a one-to-one correspondence between extremal Lagrangian tori in $(\bar{B}^{2n}(1),\omega_{\mathrm{std}})$ and  monotone Lagrangian tori in $(\mathbb{CP}^{n-1},\omega_{\mathrm{FS}})$.
	
	Let $L_0, L_1$ be two Lagrangian tori in $(\mathbb{CP}^{n-1},\omega_{\mathrm{FS}})$ that are Hamiltonian isotopic. Let 
	$H:\mathbb{R}\times \mathbb{CP}^{n-1}\to \mathbb{R}$ be a Hamiltonian generating the isotopy. The function $H\circ P:\mathbb{R}\times S^{2n-1}(1)\to \mathbb{R}$ can be extended to the symplectization $(\mathbb{R}\times S^{2n-1}(1),d(e^r\lambda_{\mathrm{std}}))=(\mathbb{C}^4\setminus \{0\}, \omega_{\mathrm{std}})$  by $r$-translations and then to a smooth function on $(\mathbb{C}^{2n}, \omega_{\mathrm{std}})$. The flow of the time-dependent Hamiltonian $H\circ P$ preserves the unit sphere $S^{2n-1}(1)$ and takes $P^{-1}(L_0)$ to 
	$P^{-1}(L_1)$. 
	
	On the other hand, let $L_0$ and $ L_1$ be extremal Lagrangian tori in $(\bar{B}^{2n}(1),\omega_{\mathrm{std}})$ and let $\{\phi^t\}$ be a Hamiltonian isotopy taking $L_0$ to $L_1$ while keeping it in the ball. By Theorem \ref{extremal-lag-ball}, the  extremal Lagrangian torus $\phi^t(L_0)$ is contained entirely in the boundary $S^{2n-1}(1)$ for each $t$. The projection under Hopf fibration $P(\phi^t(L_0))$ is a smooth isotopy between $P(L_0)$ and $P(L_1)$ that preserves the symplectic area class. By the isotopy extension theorem, we can extend this smooth isotopy to a global Hamiltonian isotopy of $(\mathbb{CP}^{n-1},\omega_{\mathrm{FS}})$. This completes the proof of $(1)$.
	
	For $n\geq 3$, the Vianna tori \cite{Vianna:2014aa,Chanda:2023aa} are extremal Lagrangian tori in $(\mathbb{CP}^{n-1},\omega_{\mathrm{FS}})$ that are pairwise non Hamiltonian isotopic. Lifting these tori to $S^{2n-1}(1)$ yields infinitely many Hamiltonian isotopy classes of extremal Lagrangian tori in $(\bar{B}^{2n}(1),\omega_{\mathrm{std}})$.
	
	For $(3)$, it is well known that any embedded loop on  $\mathbb{CP}^1$ that bounds an area equal to half of the total area can be moved to the equator via a Hamiltonian isotopy of $\mathbb{CP}^1$. Under the Hopf fibration the equator lifts to the torus  $S^1(\frac{1}{2})\times S^1(\frac{1}{2})$ in $(\bar{B}^{4}(1),\omega_{\mathrm{std}})$. The result follows from $(1)$ because any extremal Lagrangian torus in $(\bar{B}^{4}(1),\omega_{\mathrm{std}})$ projects to an embedded loop on  $\mathbb{CP}^1$ bounding an area of $\frac{1}{2}$. 
\end{proof}
The following theorem proves that Conjecture \ref{extremal-lang1} holds for a broader class of symplectic manifolds, which includes four-dimensional strictly convex toric domains. This pushes the work of Georgios Dimitroglou Rizell
 \cite{DimitroglouRizell1} in a different direction.
\begin{theorem}\label{extremal-lag-toric}
Let $(X^4_{\Omega},\omega_{\mathrm{std}})$ be a four-dimensional toric domain such that there exists an ellipsoid $E^{4}(a,b)$ satisfying  $X^4_{\Omega} \subseteqq E^{4}(a,b)$ and $\operatorname{diagonal }(X^4_{\Omega})=\operatorname{diagonal }(E^{4}(a,b))$; see Figure \ref{toric-inter} for an illustration. Then every Lagrangian torus in $(X^4_{\Omega},\omega_{\mathrm{std}})$ which has  symplectic area at least $\operatorname{diagonal}(X^{4}_{\Omega})$ lies entirely in $\partial X^4_{\Omega}\cap \partial E^{4}(a,b)$. In particular, it follows from (\ref{lag-capacity-con}) that
	Conjecture \ref{extremal-lang1} holds for $(X^4_{\Omega},\omega_{\mathrm{std}})$ (cf. Conjecture \ref{generalizedconjecture}).
\end{theorem}
\begin{corollary}\label{extremal-cylinder}
Let $X^4_{\Omega}$ be a four-dimensional toric domain such that there exists an ellipsoid $E^{4}(a,b)$ satisfying  $X^4_{\Omega} \subseteqq E^{4}(a,b)$ and $\operatorname{diagonal }(X^4_{\Omega})=\operatorname{diagonal }(E^{4}(a,b))$; see Figure \ref{toric-inter} for an illustration. Then every Lagrangian torus in $(X^4_{\Omega},\omega_{\mathrm{std}})$ which has  symplectic area at least $\operatorname{diagonal}(X^{4}_{\Omega})$ intersects the  Clifford torus $\mu^{-1}(\frac{ab}{a+b},\frac{ab}{a+b})=S^1(\frac{ab}{a+b})\times S^1(\frac{ab}{a+b})\subset \partial X^4_{\Omega}\cap \partial E^{4}(a,b)$. In particular:
\begin{itemize}
	\item for any $0<a\leq b<\infty$, a Hamiltonian isotopy $L_t\subset \mathbb{C}^2$ of the torus $S^1(\frac{ab}{a+b})\times S^1(\frac{ab}{a+b})$ cannot displace $S^1(\frac{ab}{a+b})\times S^1(\frac{ab}{a+b})$ from itself provided that $L_1\subset E^{4}(a,b)$;
	\item every extremal Lagrangain torus in $(X^4_{\Omega},\omega_{\mathrm{std}})$ lies in the connected component of $\partial X^4_{\Omega}\cap \partial E^{4}(a,b)$ that contains $S^1(\frac{ab}{a+b})\times S^1(\frac{ab}{a+b})$. Moreover, if $(\frac{ab}{a+b},\frac{ab}{a+b})\in \partial \Omega \cap  \mu(\partial E^{4}(a,b))$ is an isolated intersection, then $(X^4_{\Omega},\omega_{\mathrm{std}})$ contains a unique extremal Lagrangain torus, namely, $S^1(\frac{ab}{a+b})\times S^1(\frac{ab}{a+b})$. 
\end{itemize}  
\end{corollary}
\begin{figure}[h]
	\centering
	\includegraphics[width=6cm]{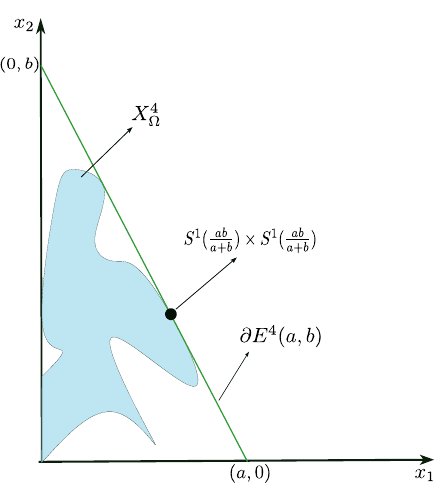}
\caption{The moment map images of $X^4_{\Omega}$ and $E^{4}(a,b)$.}\label{toric-inter}
\end{figure}

\begin{proof}[Proof of Corollary \ref{extremal-cylinder}]
The moment map image of $E^4(a,b)$ is the  triangle with vertices $(0,0), (a,0)$ and $(0,b)$ as shown in Figure \ref{toric-inter}. The green side corresponds to the boundary $\partial E^4(a,b)$, where the point $(\frac{ab}{a+b}, \frac{ab}{a+b})$, denoted by the black dot, is the moment map image of  $S^1(\frac{ab}{a+b})\times S^1(\frac{ab}{a+b})$.
	
Suppose there is an extremal Lagrangain torus $L$ in 	$(E^4(a,b),\omega_{\mathrm{std}})$ that does not intersect the Clifford torus $S^1(\frac{ab}{a+b})\times S^1(\frac{ab}{a+b})$. By Theorem \ref{extremal-lag-toric}, $L$ entirely lies on $\partial E^4(a,b)$. So either $L$ lies above or below the point $(\frac{ab}{a+b}, \frac{ab}{a+b})$ on the green side of the triangle. If it lies above $(\frac{ab}{a+b}, \frac{ab}{a+b})$, then  $L$ is an extremal Lagrangian torus in $(B^{2}(\frac{ab}{a+b})\times \mathbb{C},\omega_{\mathrm{std}})$  which does not lie entirely on the boundary of $(B^{2}(\frac{ab}{a+b})\times\mathbb{C},\omega_{\mathrm{std}})$. This contradicts Theorem \ref{extremal-lag-cylinder}.  If it lies below the point $(\frac{ab}{a+b}, \frac{ab}{a+b})$, then  $L$ is an extremal Lagrangian torus in $(\mathbb{C}\times B^{2}(\frac{ab}{a+b}),\omega_{\mathrm{std}})$  which does not lies entirely on its boundary---a contradiction again. 
	
Hamiltonian isotopies preserve the symplectic area of Lagrangian tori. If $L_t\subset E^4(a,b)$ is a Hamiltonian isotopy of the Clifford torus $S^1(\frac{ab}{a+b})\times S^1(\frac{ab}{a+b})$, then $L_t$ is extremal for each $t$. Hence $L_t\subset \partial E^4(a,b)$ for all $t$ by Theorem \ref{extremal-lag-toric}. We now apply the above argument for each $t$.
\end{proof}
\begin{remark}
It is an elementary fact that any embedded loop on $\mathbb{CP}^1$ that bounds an area of $\frac{1}{2}$ (half of the total area) does not lie in an open hemisphere. This can be deduced from Corollary \ref{extremal-cylinder}: note that it is enough to prove that any such loop intersects the equator. Arguing by contradiction, suppose that it is not the case. The lift of this loop to $S^3(1)$ under the Hopf fibration is a Lagrangian torus in the ball $\bar{B}^4(1)$. This lifted torus has symplectic area $\frac{1}{2}$ and hence is extremal by Theorem \ref{compute-lag}. Moreover, our assumption implies that this lift does not intersect the Clifford torus $S^1(\frac{1}{2})\times S^1(\frac{1}{2})$. This is a contradiction to Corollary \ref{extremal-cylinder}.
	
It would be interesting to see if the higher dimensional version of this holds, i.e., whether all extremal (hence monotone by Theorem \ref{monoton=extremal}) Lagrangian tori in $(\mathbb{CP}^2, \omega_{\mathrm{FS}})$ intersect the monotone Clifford torus in $(\mathbb{CP}^2, \omega_{\mathrm{FS}})$.  By the discussion above, it boils down to whether or not Corollary \ref{extremal-cylinder} is valid for  $(\bar{B}^{6}(1),\omega_{\mathrm{std}})$.
\end{remark}
\begin{remark}[Non-squeezing of monotone tori]
Consider the moment map image of $\mathbb{CP}^2$ as shown in Figure \ref{momentofball1}. As pointed out in {\cite[Figure 1]{MR4251089}}, the monotone Chekanov torus $L_{\mathrm{ch}}$ in $\mathbb{CP}^2$ can be placed on the diagonal in the region defined by
\[x_1<\frac{1}{3}+\epsilon,\]
for any $\epsilon>0$. It follows from Theorem  \ref{extremal-lag-cylinder}, with $k=m=1$, that $L_{\mathrm{ch}}$ can neither be squeezed into the region
\begin{equation}\label{non-squeezing01}
x_1\leq \frac{1}{3},
\end{equation}
nor into the region 
\begin{equation}\label{non-squeezing02}
x_2\leq \frac{1}{3}.
\end{equation}
In fact, by Theorem \ref{extremal-lag-cylinder}, no monotone Lagrangian torus in $\mathbb{CP}^2$, except the Clifford torus, can be squeezed into regions given by  (\ref{non-squeezing01}) and (\ref{non-squeezing02}). Moreover, by Theorem \ref{extremal-lag-toric}, no monotone Lagrangian torus in $\mathbb{CP}^2$ can be squeezed into regions given by the open ellipsoids $\operatorname{int}(E(a,b))$ of diagonal $\frac{1}{3}$.    
	\begin{figure}[h]
	\centering
	\includegraphics[width=6cm]{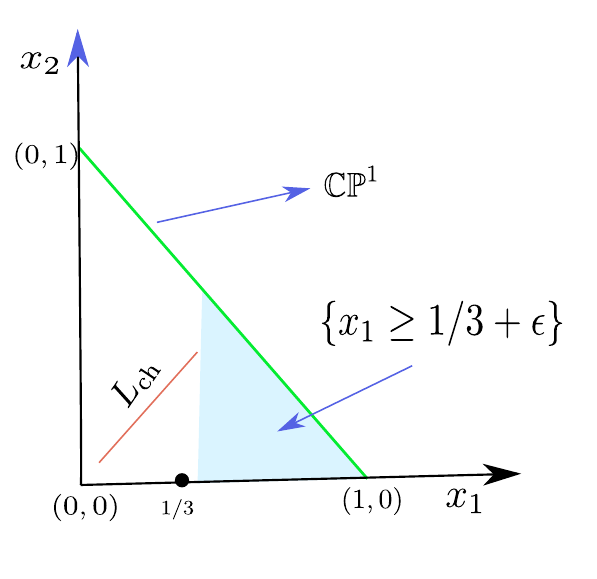}
\caption{Moment map image of $\mathbb{CP}^2$. The red line represents $L_{\mathrm{ch}}$.}\label{momentofball1}
\end{figure}
\end{remark}
The following class of examples explains that Conjecture \ref{extremal-lang1} does not generally hold for non-convex toric domains in all dimensions.
\begin{counterexample}
Consider the standard moment map $\mu:\mathbb{C}^{2n}\to \mathbb{R}^{n}_{\geq 0}$  defined by
\[\mu(z_1,z_2,\dots, z_n):= \pi(|z_1|^2,|z_2|^2,\dots, |z_n|^2).\]

The non-disjoint union of cylinders of diagonal $r>0$ in $\mathbb{C}^{2n}$ is the subset given by
	\[N^{2n}(r):=\mu^{-1}(\Omega),\]
	where
\[\Omega:=\big\{(x_1,x_2,\dots, x_n)\in \mathbb{R}^{n}_{\geq 0}: x_i\leq r \text{ for some } i\in \{1,2,\dots,n\}\big\}.\]
Figure \ref{nondistoric-inter1} dispicts $\Omega$ for $n=2$.

 The diagonal of $N^{2n}(r)$ as a toric domain is $r$. By {\cite[Theorem 4.37]{Pereira:2022aa}}, an extremal Lagrangian torus in  $(N^{2n}(r), \omega_{\mathrm{std}})$ has symplectic area equal to $r$.
	
	There are extremal Lagrangian tori in the non-disjoint union of cylinders $N^4(\frac{1}{2})$ that intersect the interior of $N^4(\frac{1}{2})$. Take a smoothly embedded loop on $\mathbb{CP}^1$ other than the equator such that it bounds a disk of area equal to $\ \frac{1}{2}$ (half of the total area). Lifting this loop to $S^3(1)$ via the Hopf fibration gives a Lagrangian torus of symplectic area $\frac{1}{2}$ in $S^3(1)\subset N^4(\frac{1}{2})$. This torus is extremal in $(N^4(\frac{1}{2}), \omega_{\mathrm{std}})$ and intersects the interior of $N(\frac{1}{2})$.
	
For $n\in \mathbb{Z}_{\geq 3}$, the lifts of the Vianna tori \cite{Vianna:2014aa, Chanda:2023aa} to $S^{2n-1}(1)$ produce infinitely many extremal Lagrangian tori in $(N^{2n}(\frac{1}{n}), \omega_{\mathrm{std}})$ that intersect the interior of $N^{2n}(\frac{1}{n})$.\qed
\end{counterexample}
\begin{figure}[h]
	\centering
	\includegraphics[width=6cm]{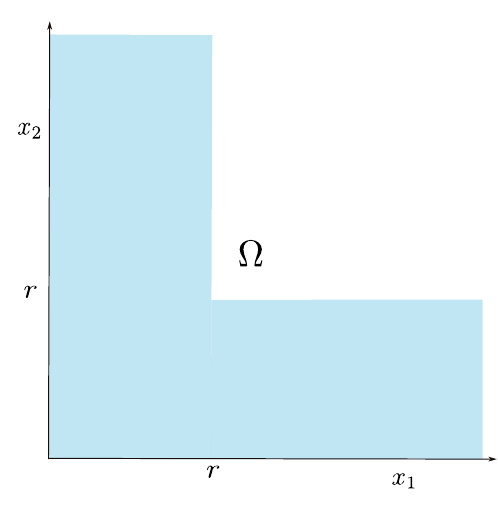}
\caption{The domain $\Omega$ for $n=2$}\label{nondistoric-inter1}
\end{figure}

{\cite[Theorem 1.16]{Cieliebak2018}} implies that every monotone Lagrangian torus in $(\mathbb{CP}^{n},\omega_{\mathrm{FS}})$ intersects every symplectically embedded  closed ball of capacity $\frac{n}{n+1} $. In the case of $\mathbb{CP}^{2}$, if the intersection occurs only at the boundary of the ball, then the Lagrangian is contained entirely in the boundary of the embedded ball due to  {\cite[Proposition 3.1]{MR4251089}}. The following theorem proves this is true for all $\mathbb{CP}^{n}$.
\begin{theorem}\label{intersection01}
	Let $L$ be a monotone Lagrangian torus in $(\mathbb{CP}^{n},\omega_{\mathrm{FS}})$. If there exists a symplectic embedding $\phi:B^{2n}(\frac{n}{n+1})\mapsto \mathbb{CP}^{n}\setminus L$, then $L$ lies entirely on the boundary $\partial \phi( B^{2n}(\frac{n}{n+1}))$. In particular, if a monotone Lagrangian torus does not intersect the standard open ball $B^{2n}(\frac{n}{n+1})$ given by the inclusion, then it lies entirely on $S^{2n-1}(\frac{n}{n+1})=\partial B^{2n}(\frac{n}{n+1})$.
\end{theorem}
\begin{corollary}\label{intersection}
	All monotone Lagrangian tori in $(\mathbb{CP}^{2},\omega_{\mathrm{FS}})$ that are disjoint from the standard open ball $B^{4}(\frac{2}{3})$(given by the inclusion) intersect the monotone Clifford torus. More generally, every monotone Lagrangian torus that is disjoint from the interior of a symplectically embedded closed ball  $\phi:\bar{B}^{4}(\frac{2}{3})\mapsto \mathbb{CP}^{2}$ must intersect the monotone Lagrangian torus $\phi(S^1(\frac{1}{3})\times S^1(\frac{1}{3}))$.  
\end{corollary}
\begin{proof}[Proof of Corollary \ref{intersection}]
	By Theorem \ref{intersection01}, a monotone Lagrangian torus $L$ in $(\mathbb{CP}^{2},\omega_{\mathrm{FS}})$ that is disjoint from the standard open ball $B^{4}(\frac{2}{3})$ must lie entirely in $S^3(\frac{2}{3})=\partial B^{4}(\frac{2}{3})$. All monotone Lagrangian tori in $(\mathbb{CP}^{2},\omega_{\mathrm{FS}})$ have symplectic areas equal to $\frac{1}{3}$. So $L$ is extremal in $(B^{4}(\frac{2}{3}),\omega_{\mathrm{std}})$. The result follows from Corollary \ref{extremal-cylinder}.
\end{proof}
\begin{theorem}\label{monoton=extremal}
	Conjecture \ref{monotone-extremal-conjecture} holds for $(\mathbb{CP}^{2},\omega_{\mathrm{FS}})$, i.e.,  a Lagrangian torus in $(\mathbb{CP}^2,\omega_{\mathrm{FS}})$ is extremal if and only if  it is monotone.
\end{theorem}
\begin{proof}[Proof of Theorem \ref{monoton=extremal}]
	We have a symplectic inclusion
	\[ \bigl(B^{4}(3), \omega_{\mathrm{std}}\bigr)=\bigl(\mathbb{CP}^2\setminus \mathbb{CP}^{1},3\omega_{\mathrm{FS}}\bigr)\subset \bigl(\mathbb{CP}^2,3\omega_{\mathrm{FS}}\bigr)\]
	where $\mathbb{CP}^{1}$ is the hypersurface at infinity.	
	
	Let $L $ be an extremal Lagrangian torus in $(\mathbb{CP}^2,3\omega_{\mathrm{FS}}\bigr)$. By the conformality of the Lagrangian capacity {\cite[Section 1.2]{Cieliebak2018}} and Theorem \ref{compute-lag}, $L$ has symplectic area equal to $1$. Moreover, every smooth disk $u:(D^2,\partial D^2)
	\to (B^{4}(3), L)$ satisfies
	\[\int_{D^2}u^*\omega_{\mathrm{std}}\in \mathbb{Z}.\]
	
Recall that Hamiltonian isotopies preserve both the extremality and monotonicity of Lagrangian tori. By {\cite[Theorem C]{DimitroglouRizell2}}, after applying a Hamiltonian isotopy to $L$, we can assume that 
	\[ L\subset \bigl(B^{4}(3), \omega_{\mathrm{std}}\bigr)=\bigl(\mathbb{CP}^2\setminus \mathbb{CP}^{1},3\omega_{\mathrm{FS}}\bigr)\subset \bigl(\mathbb{CP}^2,3\omega_{\mathrm{FS}}\bigr).\]
	
Arguing by contradiction, suppose  that $L$ is not monotone in $(\mathbb{CP}^2,3\omega_{\mathrm{FS}}\bigr)$. Then $ H_1(L,\mathbb{Z})$ cannot admit integral basis $\{b_1, b_2\}$ of Maslov indices $2$ and both bounding disks of symplectic areas $1$.

By \cite{Cieliebak2018}, there is a smooth disk $u:(D^2,\partial D^2)
\to (B^{4}(3), L)$ such that $b_1:=[u(\partial D^2)]\in H_1(L,\mathbb{Z})$ has Maslov index $2$ and symplectic area $1$.
Choose  $b_2\in H_1(L,\mathbb{Z})$ so that the pair $\{b_1, b_2\}$ forms a basis of $H_1(L,\mathbb{Z})$. Since $L$ is orientable, $b_2$ has even Maslov index. After adding multiples of $b_1$ to $b_2$, we can assume that $b_2$ also has Maslov index $2$. Since $L$ is not monotone, $b_2$ does not bound a disk of symplectic area $1$. If $b_2$ bounds a disk of symplectic area greater than $1$, then $\{b_1, b_2\}$ forms a basis of $H_1(L,\mathbb{Z})$ that satisfies the hypothesis of {\cite[Theorem 2]{Hind}}. This is a contradiction to {\cite[Theorem 2]{Hind}} because $L\subset B^{4}(3)$. If $b_2$ bounds a disk of symplectic area strictly less than $1$, then we replace $b_2$ with $\bar{b}_2:=2b_1-b_2$. The pair $\{b_1, \bar{b}_2\}$ forms an integral basis of $H_1(L,\mathbb{Z})$ that satisfies the hypothesis of {\cite[Theorem 2]{Hind}} and hence a contraction to $L\subset B^{4}(3)$.

The converse follows from {\cite[Corollary 1.7]{Cieliebak2018}}.\qedhere

%
\end{proof}
This paper is organized as follows. In Section \ref{Chapter02}, we review the symplectic field theory framework and state the Gromov--Hofer compactness theorem for moduli spaces of punctured pseudo-holomorphic curves, adapted to the context of this document. Sections \ref{Chapter03} through \ref{puncturecountlio} present two variants of Gromov--Witten invariants for a symplectic manifold $(X,\omega)$. The first variant, due to Cieliebak--Mohnke \cite{Cieliebak_2007}, counts holomorphic spheres in $X$ satisfying the holomorphic curve equation $(du)^{0,1}=0$, subject to a generic local tangency constraint with respect to a local symplectic divisor. The second variant, introduced by McDuff--Siegel \cite{McDuff:2021aa}, counts punctured holomorphic spheres in a completed symplectic cobordism, also subject to a generic local tangency constraint with respect to a local symplectic divisor.

We prove Theorem \ref{extremal-lag-ball} in Section \ref{proof for ball}, Theorem \ref{intersection01} in Section \ref{proofof1.16}, Theorem \ref{extremal-lag-cylinder} in Section \ref{proofof1.8}, and Theorem \ref{extremal-lag-toric} in Section \ref{sectiontoricproof}. The proof of Theorem \ref{extremal-lag-cylinder} closely follows that of Theorem \ref{extremal-lag-ball}, so we provide only a sketch, as the details can be recovered from the earlier argument.

\section*{Acknowledgement}
This article is based on my PhD thesis, which was carried out at Humboldt-Universit\"at zu Berlin. I have received financial support from the Deutsche Forschungsgemeinschaft (DFG, German Research Foundation) under Germany's Excellence Strategy-The Berlin Mathematics Research Center MATH+ (EXC-2046/1, project ID: 390685689).

This article owes a lot to Kai Cieliebak, Georgios Dimitroglou Rizell, and Urs Frauenfelder. Most of this work was carried out during my eight-month visit to the research group of Kai Cieliebak and Urs Frauenfelder in Augsburg and a one-week visit to the research group of Georgios Dimitroglou Rizell in Uppsala. I have learned a lot from them, and I wish to thank them for their helpful feedback, guidance, encouragement, support, and hospitality. 

 I have benefited from discussions with several other mathematicians. Particularly, I am very thankful to Janko Latschev, Kei Irie, Chris Wendl, Jean Gutt, Dusa McDuff,  Thomas Walpuski, Felix Schlenk,  Alexandru Oancea, Fr\a'ed\a'eric Bourgeois, Claude Viterbo, Milica Dukic, Zhengyi Zhou,  Claudio Arezzo, and Gorapada Bera.


\section{Symplectic field theory framework}\label{Chapter02}
\subsection{Symplectic cobordisms}
Let $Y$ be a smooth $(2n-1)$-dimensional manifold. A contact form on $Y$ is a $1$-form $\lambda$ on $Y$ such that $d\lambda|_{\operatorname{Ker}(\lambda)}$ is non-degenerate, equivalently $\lambda\wedge (d\lambda)^{n-1}$ is a volume form on $Y$. The corresponding hyperplane distribution $\xi:=\operatorname{Ker}(\lambda)$ on $Y$ is called a (co-oriented) contact structure on $Y$ and the pair $(Y, \xi)$ is called a contact manifold. One can define contact structures without co-orientations. However, the contact structures that appear in this document are all co-oriented.

Let $(Y, \lambda)$ be a closed contact manifold. The Reeb vector field $R_\lambda$ of $\lambda$ is the unique vector field  on $Y$ characterized by 
\[d\lambda(R_\lambda,\cdot)=0 \text{ and } \lambda(R_\lambda)=1.\] 

A closed Reeb orbit of period $T>0$ is a smooth map $\gamma:\mathbb{R}/T\mathbb{Z}\to  Y$, modulo translations of the domain, such that $\gamma '(t)= R_\lambda(\gamma(t))$ for all $t\in \mathbb{R}/T\mathbb{Z}$. Let $\phi^t:Y\to Y$ be the flow of $R_{\lambda}$. The linearized flow induces a family of symplectic maps $d\phi^t:(\xi_p, d\lambda) \to (\xi_{\phi^t(p)}, d\lambda)$, where $\xi:=\operatorname{Ker}(\lambda)$ and $p\in Y$. A closed Reeb orbit $\gamma$  of period $T>0$ is non-degenerate \footnote{If $\gamma$ is nondegenerate, then its translates $\gamma(\cdot+\theta)$ are also non-degenerate. So the notion of nondegeneracy is well-defined.} if  the linearized return map $d\phi^T(\gamma(0)):(\xi_{\gamma(0)}, d\lambda) \to (\xi_{\gamma(0)}, d\lambda)$ has no eigenvalue equal to $1$. A contact form $\lambda$ on $Y$ is non-degenerate if all closed Reeb orbits on $Y$ are non-degenerate in this sense.

Let $(Y,\lambda)$ be a contact manifold. The symplectic manifold $(\mathbb{R}\times Y, d(e^r\lambda))$, where $r$ is the real-coordinate, is called the symplectization of $(Y,\lambda)$.

\begin{definition}[Symplectic cobordism]\label{cobordism}
Let $ (Y_{-}, \lambda_-)$ and $(Y_{+}, \lambda_+)$ be two $(2n-1)$-dimensional closed contact manifolds. A symplectic cobordism from  $ (Y_{-}, \lambda_-)$ to $(Y_{+}, \lambda_+)$ is a compact $2n$-dimensional symplectic manifold $(X, \omega)$ such that the following hold.
\begin{itemize}
\item The boundary $\partial X=\partial^- X \cup \partial ^+X$  admits an orientation-preserving diffeomorphism to $-Y_{-} \cup Y_{+}$. Here, the orientations on $Y_{+}$ and $Y_{-}$ come from their contact structures. The minus sign in front of $Y_-$ means that the orientation has been reversed. 
\item  The contact form $\lambda_-$ extends to a $1$-form (still denoted by $\lambda_-$) to a neighborhood of $\partial X_{-}$ such that $d\lambda_-=\omega$ on that neighbourhood. Similarly, the contact form $\lambda_+$ extends to a $1$-form (still denoted by $\lambda_+$) to a neighborhood of $\partial X_{+}$ such that $d\lambda_+=\omega$ on that neighborhood.
\end{itemize}
\end{definition}
In the definition of symplectic cobordism, we allow one or both of $Y_{-}$ and $Y_{+}$ to be empty.

\begin{example}
	Any closed symplectic manifold is a symplectic cobordism from the empty set to the empty set.  
\end{example}
\begin{example}
A Liouville domain is a pair $(X,\lambda)$, where $X$ is a compact oriented smooth manifold with (nonempty) boundary $\partial X$ and $\lambda$ is a $1$-form such that the exterior derivative $d\lambda$ is symplectic and the restriction of $\lambda$ to $\partial X$ is a positive\footnote{The vector field $V$ on $X$ defined by $d\lambda(V,\cdot)=\lambda(\cdot)$ points outwards along $\partial X$. Equivalently, the orientation induced by $\lambda$ on $\partial X$ agrees with the boundary orientation.} contact form. For every Liouville domain $(X,\lambda)$,  $(X,d\lambda)$ is a symplectic cobordism with $\partial^-X=\emptyset$ and $\partial^+X=\partial X$. An example of a Liouville domain is the standard ellipsoid $(E^{2n}(a_1,a_2,\dots, a_{n}),\lambda_{\mathrm{std}})$ defined by 
\[E^{2n}(a_1,a_2,a_3,\dots,a_n):=\bigg\{(z_1,\dots, z_{n})\in \mathbb{C}^{n}: \sum_{i=1}^{n}\frac{  \pi|z_i|^2}{a_i}\leq 1\bigg \}\] 
with 
\[\lambda_{\mathrm{std}}:=\frac{1}{2}\sum_{i=1}^{n}(x_idy_i-y_idx_i).\]
Here, $(E^{2n}(a_1,a_2,a_3,\dots,a_n),\omega_{\mathrm{std}}:=d\lambda_{\mathrm{std}})$ is a symplectic cobordism from the empty set to $(\partial E^{2n}, \lambda_{\mathrm{std}})$. 
\end{example}		
\begin{example}
Let $(X, \omega)$ be a closed $2n$-dimensional symplectic manifold. For each $S>0$, there exists a symplectic embedding $i:(E^{2n}(\epsilon,\epsilon S,\dots, \epsilon S), \omega_{\mathrm{std}})\to (X, \omega)$ for sufficiently small $\epsilon>0$. The manifold with boundary \[X\setminus \operatorname{int}(i(E^{2n}(\epsilon,\epsilon S,\dots, \epsilon S)))\] is a symplectic cobordism from $(\partial E^{2n}, \lambda_{\mathrm{std}})$ to the empty set.
\end{example}
\begin{example}
Let $(L, g)$ be an $n$-dimensional Riemannian manifold. The cotangent bundle $T^*L$ admits a canonical symplectic form denoted by $d\lambda_{\mathrm{can}}$. Let $(D^*L, d\lambda_{\mathrm{can}})$ be the symplectic unit codisk bundle of $L$ with respect to the metric $g$, i.e.,  
\[D^*L:=\Bigl\{v\in T^*L: \|v\|_g\leq 1 \Bigr\}.\]
	This is a symplectic cobordism from the empty set to the  unit cosphere bundle $(S^*L,\lambda_{\mathrm{can}})$ of $L$ given by 
	\[S^*L:=\Bigl \{v\in T^*L: \|v\|_g=1\Bigr\}.\]
\end{example}
It follows from the definition of symplectic cobordism that the vector field $V_+$ defined by $\omega(V_+,\cdot)=\lambda_+$ points outwards along $Y_{+}$, and the vector field $V_-$ defined by $\omega(V_-,\cdot)=\lambda_-$ points inwards along $Y_{-}$. Let $\psi^t_{V_+}:Y_{+}\to X$ be the backward flow of $V_+$ and $\psi^t_{V_-}:Y_{-}\to X$ be the forward flow of $V_-$. These flows yield symplectic embeddings 
\[\psi_{V_+}:((-\epsilon,0]\times Y_{+},d(e^r\lambda_+))\to (X,\omega),\, (r,p)\to  \psi_{V_+}(r,p):=\psi^r_{V_+}(p) \]
and 
\[\psi_{V_-}:([0,\epsilon)\times Y_{-},d(e^r\lambda_-))\to (X,\omega),\, (r,p)\to  \psi_{V_-}(r,p):=\psi^r_{V_-}(p) .\]
\begin{definition}[Symplectic completion]\label{completion}
	Let $(X, \omega)$ be a symplectic cobordism from $ (Y_{-}, \lambda_-)$ to $(Y_{+}, \lambda_+)$. The symplectic completion of $(X, \omega)$ is obtained by gluing the negative half-cylinder $((-\infty,0]\times  Y_-, d(e^r \lambda_-))$ to $X$ along $Y_-$ via $\psi_{V_-}$ and  the positive half-cylinder $[0,\infty)\times  Y_+, d(e^r \lambda_+))$ to $X$ along $ Y_+$ via $\psi_{V_+}$.  The symplectic form $\omega$ on $X$ extends to the symplectic form on the completion as 
	\begin{equation*}
		\widehat{\omega}:=
		\begin{cases}
			
			d(e^r\lambda_+ )& \text{on } [0,\infty)\times Y_+\\
			\omega & \text{on } X\\
			d(e^r\lambda_-) & \text{on } (-\infty,0]\times Y_-.
		\end{cases}
	\end{equation*}
We denote the symplectic completion of $(X,\omega)$ by $\widehat{X}$ and call it the completed symplectic cobordism.
\end{definition}
We will also need the following piecewise-smooth $2$-form on the symplectic completion $\widehat{X}$: 
\begin{equation}\label{deg-form}
	\tilde{\omega}:=
	\begin{cases}
		d\lambda_+ & \text{on } [0,\infty)\times Y_+\\
		\omega & \text{on } X\\
		d\lambda_- & \text{on } (-\infty,0]\times Y_-.
	\end{cases}
\end{equation}
\begin{example}\label{examplead}
	Let $(\mathbb{\bar{B}}^{2n}(1),\omega_{\mathrm{std}})$ be the closed unit ball in $\mathbb{C}^n$. It is a symplectic cobordism from the empty set to the unit sphere $\mathbb{S}^{2n-1}$ that carries the standard contact form $\lambda_{\mathrm{std}}:=\frac{1}{2}\sum_{i=1}^{n}(x_idy_i-y_idx_i)$. The symplectic completion of $(\mathbb{\bar{B}}^{2n}(1),\omega_{\mathrm{std}})$ is symplectomorphic to $(\mathbb{C}^n, \omega_{\mathrm{std}})$. An explicit symplectomorphism $\phi: \mathbb{C}^n\to \widehat{\mathbb{\bar{B}}^{2n}(1)}$ is given by
	\begin{equation*}
		\phi (z)=
		\begin{cases}
			z & \text{if } |z|\leq 1\\
			(\frac{1}{2}\log|z|,\frac{z}{|z|}) & \text{if } |z|\geq 1.
		\end{cases}
	\end{equation*}
\end{example}
\begin{definition}
	Let $(X,\omega)$ be a symplectic manifold. An embedded hypersurface $\Sigma$ in $X$ is called a contact type hypersurface if there exists a vector field $V$ defined on a neighborhood of $\Sigma$ that is transverse to $\Sigma$ and satisfies $\mathcal{L}_V\omega=\omega$, where $\mathcal{L}_V$ denotes the Lie derivative with respect to $V$. The vector field $V$ is called a Liouville vector field.
\end{definition}
\begin{remark}
The vector field $V$ induces a contact form $\lambda$ on $\Sigma$ defined by $\lambda:=\omega(V,\cdot)$.
\end{remark}
\begin{example}
The boundary of the standard ellipsoid $E^{2n}(a_1,a_2,\dots, a_{n})$ is a hypersurface of contact type in $(\mathbb{C}^n, \omega_{\mathrm{std}})$. The radial vector field in  $(\mathbb{C}^n, \omega_{\mathrm{std}})$ plays the role of $V$. More generally, the boundary of every starshaped domain in  $(\mathbb{C}^n, \omega_{\mathrm{std}})$ is a hypersurface of contact type for which the radial vector field is a Liouville vector field.
\end{example}
\begin{definition}[SFT-admissible almost complex structure]\label{SFT-admissible}
Let $(X, \omega)$ be a symplectic cobordism from $ (Y_{-}, \lambda_-)$ to $(Y_{+}, \lambda_+)$ and $\widehat{X}$ be its symplectic completion.	An almost complex structure $J$ on   $\widehat{X}$ is called  SFT-admissible if the following hold. 
\begin{itemize}
	\item $J$ is  compatible with the completed symplectic form $\widehat{\omega}$.
	\item There exists $k>0$ such that $J$ is translation invariant (in the $r$-direction) on $([k,\infty)\times  Y_+, d(e^r \lambda_+))$ and on $((-\infty,-k]\times  Y_-, d(e^r \lambda_-))$.
\item $J$ preserves $\xi_-:=\operatorname{Ker}(\lambda_-)$ and $\xi_+:=\operatorname{Ker}(\lambda_+)$ . Moreover, $J$ is cylindrical with respect to $Y_+$ and $Y_-$, i.e.,  $J$ maps $\partial_r$ to the Reeb vector fields on the cylindrical ends.
\end{itemize}
\end{definition}
\subsection{Punctured $J$-holomorphic nodal curves}
A punctured nodal Riemann surface is a tuple $(\Sigma,j, Z^+\cup Z^-,\Delta_{\mathrm{nd}}, \phi)$ where:
\begin{itemize}
	\item $(\Sigma,j)$ is a closed (not necessarily connected) Riemann surface with a finite number of pairwise distinct points organized in two collections $Z^+=(z^+_1,z^+_2, \dots, z^+_p)$ and $Z^-=(z^-_1,z^-_2, \dots, z^-_q)$.  We allow both sets $Z^+$ and $Z^-$ to be empty.
	\item $\Delta_{\mathrm{nd}}$ is a finite set consisting of an even number of points in $\Sigma\setminus (Z^+\cup Z^-)$ equipped with an involution $\phi:\Delta_{\mathrm{nd}} \to \Delta_{\mathrm{nd}}$. Each pair $\{z,\phi(z)\}$ for $z\in \Delta_{\mathrm{nd}}$ is called a node.
\end{itemize}
Let $\widehat{\Sigma}$ be the surface obtained by  identifying $z$ with  $\phi(z)$ for all $z\in \Delta_{\mathrm{nd}}$. We say the nodal Riemann surface is connected if $\widehat{\Sigma}$ is connected. The genus of $\widehat{\Sigma}$ is called the arithmetic genus of $(\Sigma,j, Z^+\cup Z^-,\Delta_{\mathrm{nd}}, \phi)$.

Let $(X, \omega)$ be a symplectic cobordism from $ (Y_{-}, \lambda_-)$ to $(Y_{+}, \lambda_+)$ and $J$ be an SFT-admissible almost complex structure on the completion $\widehat{X}$. A punctured $J$-holomorphic curve of genus $g$ in $(\widehat{X}, J)$ is a tuple $(u, \Sigma,j, Z^+\cup Z^-,\Delta_{\mathrm{nd}}, \phi)$ which consists of the following data: 
\begin{itemize}
	\item[(1)] A punctured nodal Riemann surface $(\Sigma,j, Z^+\cup Z^-,\Delta_{\mathrm{nd}}, \phi)$, where the two finite sets  $Z^+=(z^+_1,z^+_2, \dots, z^+_p)$ and $Z^-=(z^-_1,z^-_2, \dots, z^-_q)$  are allowed to be empty.
	
	\item[(2)] A map $u: \Sigma\setminus Z^+\cup Z^-\to \widehat{X}$ such that $du\circ j=J\circ du$ and  $u(z)=u(\phi(z))$ for each $z\in \Delta_{\mathrm{nd}}$.

	\item[(3)] There exist two tuples  of closed Reeb orbits $(\gamma^+_1,\gamma^+_2, \dots, \gamma^+_p)$ and $(\gamma^-_1,\gamma^-_2, \dots, \gamma^-_q)$ on $ (Y_{-}, \lambda_-)$ and $(Y_{+}, \lambda_+)$, respectively, such that the following hold. For every $i\in \{1,2,\dots, p\}$, if $(r, \theta)$ are polar coordinates near $z_i^+$ 
	then
	\[ u(r\exp(i\theta))\to (\infty, \gamma_i^+(T^+_i\theta/2\pi)) \]
uniformly in $\theta$ as $r\to \infty$, where $T^+_i$ is the period of the orbit $\gamma_i^+$. Similarly, for every $i\in \{1,2,\dots, q\}$, if $(r, \theta)$ are polar coordinates near $z_i^-$ 
	then
	\[ u(r\exp(i\theta))\to (-\infty, \gamma_i^+(-T^-_i\theta/2\pi))\]
uniformly in $\theta$ as $r\to -\infty$,  where $T^-_i$ is the period of the orbit $\gamma_i^-$.
\end{itemize}
In case of (2)--(3), we say the punctured $J$-holomorphic curve $u$ has positive ends on the Reeb orbits $(\gamma^+_1,\gamma^+_2, \dots, \gamma^+_p)$ and negative ends on the Reeb orbits $(\gamma^-_1,\gamma^-_2, \dots, \gamma^-_q)$. Sometimes we call $u$ positively asymptotic to the Reeb orbits $(\gamma^+_1,\gamma^+_2, \dots, \gamma^+_p)$ and negatively asymptotic to the Reeb orbits $(\gamma^-_1,\gamma^-_2, \dots, \gamma^-_q)$.  A punctured $J$-holomorphic nodal curve is called connected if the underlying nodal Riemann surface $(\Sigma,j, Z^+\cup Z^-,\Delta_{\mathrm{nd}}, \phi)$ is connected. A punctured $J$-holomorphic nodal curve is called smooth if the underlying nodal Riemann surface $(\Sigma,j, Z^+\cup Z^-,\Delta_{\mathrm{nd}},\phi)$ has no nodes, i.e., $\Delta_{\mathrm{nd}}$ is empty. A smooth connected punctured $J$-holomorphic curve is a  punctured $J$-holomorphic nodal curve where the underlying nodal Riemann surface $(\Sigma,j, Z^+\cup Z^-,\Delta_{\mathrm{nd}},\phi)$ is connected and has no nodes, i.e., $\Delta_{\mathrm{nd}}$ is empty. A smooth connected punctured $J$-holomorphic curve of genus zero with only one puncture is called an asymptotically cylindrical $J$-holomorphic plane. A smooth connected punctured $J$-holomorphic curve of genus zero with only two punctures is called an asymptotically cylindrical $J$-holomorphic cylinder or simply a $J$-holomorphic cylinder. 
\begin{example} Let  $(Y, \lambda)$ be a closed contact manifold. Let $\gamma$ be a closed Reeb orbit on $Y$ of period $T>0$. For every SFT-admissible almost complex structure $J$ on the symplectization $\mathbb{R}\times Y$, the map 
	$u:\mathbb{R}\times\mathbb{R}/\mathbb{Z}\to \mathbb{R}\times Y$
	defined by \[u(s,t):=(Ts,\gamma(Tt))\]
	is a punctured $J$-holomorphic curve of genus $0$ with two punctures, one negative and one positive, asymptotic to the Reeb orbit   $\gamma$. This is called the trivial cylinder over $\gamma$.
\end{example}
\begin{definition}\label{homologyclass}
	Let $(X, \omega)$ be a symplectic cobordism from $ (Y_{-}, \lambda_-)$ to $(Y_{+}, \lambda_+)$  and $J$ be an SFT-admissible almost complex structure on the completion $\widehat{X}$.  Let $u: \Sigma\setminus (Z^+\cup Z^-)\to \widehat{X}$ be a punctured $J$-holomorphic curve with positive asymptotics 
	$\Gamma^+=(\gamma^+_1,\gamma^+_2, \dots, \gamma^+_p)$ and negative asymptotics $\Gamma^-=(\gamma^-_1,\gamma^-_2, \dots, \gamma^-_q)$. Let $\overline{X}$ denote the compact manifold obtained from the symplectic completion $\widehat{X}$ by attaching copies of $\{\infty\}\times Y_+$ and $\{-\infty\}\times Y_-$ to its cylindrical ends. Let $\overline{\Sigma}$ be the compact oriented surface with boundary obtained by adding a circle to $\Sigma\setminus (Z^+\cup Z^-)$ for each puncture. The behavior of $u$ near the punctures implies that $u$ can be extended to a continuous map:
	\[\bar{u}: \overline{\Sigma} \to \overline{X}.\]
Let $\Phi: \overline{X}\to X$ be the deformation retraction that collapses the cylindrical ends  to $\{0\}\times Y_{\pm}$.  We say that the punctured $J$-holomorphic curve $u: \Sigma\setminus (Z^+\cup Z^-\to \widehat{X})$ represents the relative homology class $A\in H_2(X, \Gamma^+\cup \Gamma^-; \mathbb{Z})$ if the push-forward of the fundamental class $[\overline{\Sigma}]$ of $\bar{\Sigma}$ (rel. boundary) by the map
\[\Phi \circ \bar{u}:  \overline{\Sigma} \to X\]
equals $A$. In this case, we write $[u]=A$.	
\end{definition}
\subsection{Holomorphic buildings}
Let $(X,\omega)$ be a connected symplectic cobordism with no negative boundary\footnote{Here, we also allow $\partial X=\emptyset$ in which case $X$ is just a closed symplectic manifold.}, i.e., $\partial X=\partial ^+X$ (cf. Definition \ref{cobordism}).  Given a closed connected contact type hypersurface $Y$ in  $(X,\omega)$, i.e., $Y$ is a closed codimension-$1$ connected smooth submanifold in the interior of $X$ such that there exists a vector field $V$ defined on a neighborhood of $Y$ that is transverse to $Y$ and satisfies $\mathcal{L}_V\omega=\omega$, where $\mathcal{L}_V$ denotes the Lie derivative with respect to $V$. The vector field $V$ induces a contact form $\lambda$ on $Y$ defined by $\lambda:=\omega(V,\cdot)$. Suppose that $Y$ separates $X$ into two pieces, i.e., $X\setminus Y$ has two connected components, say $X_+$ and $X_-$. We assume that the component $(X_+, \omega^+:=\omega|_{X_+})$ is a symplectic cobordism with negative boundary $(Y, \lambda)$ and positive boundary $(\partial^+ X_+=\partial X,\lambda_+)$. The other connected component, denoted by $(X_-, \omega^-:=\omega|_{X_-})$, is a symplectic cobordism with positive boundary $(Y, \lambda)$ and empty negative boundary. Choose SFT-admissible almost complex structures $J_{\partial^+ X_+}$, $J_+,J_Y$, and $J_-$ on $(\mathbb{R}\times\partial^+ X_+,d(e^r\lambda_+),d\lambda_+)$, $(\widehat{X}_+,\hat{\omega}^+, \tilde{\omega}^+), (\mathbb{R}\times Y,d(e^r\lambda),d\lambda)$ and $(\widehat{X}_-,\hat{\omega}^-, \tilde{\omega}^-)$, respectively. We assume that $J_+$ and $J_-$ agree with $J_Y$ on the cylindrical ends and also  $J_+$ agrees with $J_{\partial^+ X_+}$.

Let $N_+$ be a nonnegative integer and consider the split completed cobordism $\bigsqcup (\widehat{X}_N,\hat{\Omega}_N,\tilde{\Omega}_N, J_N)$ defined by
\begin{equation*}
	(\widehat{X}_N,\hat{\Omega}_N,\tilde{\Omega}_N, J_N):=
	\begin{cases}
		(\mathbb{R}\times\partial^+ X_+,d(e^r\lambda_+),d\lambda_+, J_{\partial^+ X_+}) & \text{for } N\in \{M+1,M+2,\dots,N_+\},\\
		(\widehat{X}_+,\hat{\omega}^+,\tilde{\omega}^+, J_+) & \text{for } N=M,\\
		
		(\mathbb{R}\times Y,d(e^r\lambda),d\lambda, J_Y)& \text{for } N\in \{1,2,\dots,M-1\},\\
		(\widehat{X}_-, \hat{\omega}^-, \tilde{\omega}^-, J_-) & \text{for } N=0.\\
	\end{cases}
\end{equation*}
Given nonnegative integers $g$ and $N_+$,  a \textbf{holomorphic building of height\footnote{In our convention, unlike \cite{MR2026549}, the $N=0$ level of the building is regarded as the underground (basement or bottom) level and therefore does not contribute to the height.} $N_+$ and genus $g$} in $\bigsqcup (\widehat{X}_N,\hat{\Omega}_N,\tilde{\Omega}_N, J_N)$ is a tuple $\mathbb{H}=(u, \Sigma,j,\Gamma^+,\Delta_{\mathrm{br}}, \Delta_{\mathrm{nd}}, \phi, L)$ where:

\begin{itemize}
	\item $(\Sigma,j,\Delta_{\mathrm{br}}\cup \Gamma^+, \Delta_{\mathrm{nd}}, \phi)$ is a punctured nodal Riemann surface of arithmetic genus $g$ with punctures at $\Delta_{\mathrm{br}}\cup \Gamma^+$. Here, $\Delta_{\mathrm{br}}$, $\Gamma^+ $, and $\Delta_{\mathrm{nd}}$ are disjoint finite sets. Each of $\Delta_{\mathrm{br}}$ and $\Delta_{\mathrm{nd}}$ has an even number of points on $\Sigma$. The involution $\phi:\Delta_{\mathrm{nd}}\cup \Delta_{\mathrm{br}} \to \Delta_{\mathrm{nd}}\cup \Delta_{\mathrm{nd}}$ preserves  $\Delta_{\mathrm{br}}$ and $\Delta_{\mathrm{nd}}$. Each pair $\{z,\phi(z)\}$ for $z\in \Delta_{\mathrm{nd}}$ is called a \textbf{node} of $u$. Each pair $\{z,\phi(z)\}$ for $z\in \Delta_{\mathrm{br}}$ is called a \textbf{breaking pair} of $u$. The set $\Gamma^+$ describing the positive punctures is allowed to be empty; for example, this will be the case if $\partial^+ X_+=\emptyset$.
	\item $L$ is a locally constant function
	\[L:\Sigma\to \{0,\dots, N_+\}\]
	that attains every value in $\{ 0,\dots, N_+\}$ except possibly $0$. Moreover, $L$ satisfies:
	\begin{enumerate}
		\item $L(z)=L(\phi(z))$ for every $z\in \Delta_{\mathrm{nd}}$;
		\item  For every $z\in \Delta_{\mathrm{br}}$, the breaking pair  $\{z,\phi(z)\}$ can be labeled as $\{z^+,z^-\}$ such that	$L(z^+)-L(z^-)=1$;
		\item $L(z)=N_+$ for every $z\in \Gamma_+$.
	\end{enumerate}
	\item $u$ is a $(j, J_N)$-holomorphic curve
\[u:\Sigma\setminus (\Delta_{\mathrm{br}}\cup \Gamma_+)\to \bigsqcup (\widehat{X}_N,\hat{\Omega}_N,\tilde{\Omega}_N, J_N) \]
		
that maps $L^{-1}(N)$ into $\widehat{X}_N$ for each $N$ and has positive punctures at $\Gamma_+$ asymptotic to closed Reeb orbits on $\partial^+ X_+$.  Moreover, for every $z\in \Delta_{\mathrm{nd}}$ we have
	\[u(z)=u(\phi(z)).\]
For each breaking pair $\{z^+,z^-\}$ labeled such that $L(z^+)=L(z^-)+1$,  $u$ has a positive puncture at $z^+$ and a negative puncture at $z^-$ asymptotic to the same Reeb orbit.  
\end{itemize}
For each $N\in \{0, \dots,N_+\}$ define 
\[\Sigma^N:=(\Sigma\setminus (\Delta_{\mathrm{br}}\cup \Gamma_+))\cap L^{-1}(N). \]
Denote the restriction of $u$ to the subset $\Sigma^N$ by 
\begin{equation*}
	u^N:\Sigma^N\to (\widehat{X}_N, J_N):=
	\begin{cases}
		(\mathbb{R}\times\partial^+ X_+, J_{\partial^+ X_+}) & \text{for } N\in \{M+1,M+2,\dots,N_+\},\\
		(\widehat{X}_+, J_+) & \text{for } N=M,\\
		
		(\mathbb{R}\times Y, J_Y)& \text{for } N\in \{1,2,\dots,M-1\},\\
		(\widehat{X}_-, J_-) & \text{for } N=0.\\
	\end{cases}
\end{equation*}
The tuple $(u^N, \Sigma^N,j, (\Delta_{\mathrm{br}}\cup \Gamma_+)\cap L^{-1}(N),\Delta_{\mathrm{nd}}\cap L^{-1}(N), \phi)$ defines a punctured $J_N$-holomorphic nodal curve  in $(\widehat{X}_N, J_N)$ with punctures (positive/negative or both) at $(\Delta_{\mathrm{br}}\cup \Gamma_+)\cap L^{-1}(N)$, and nodes at  $\Delta_{\mathrm{nd}}\cap L^{-1}(N)$. We call $u^N$ the \textbf{$N$-th level} of the building $\mathbb{H}$. We call $u^N$ the \textbf{top level}, a \textbf{middle level}, or the \textbf{bottom level} if $N=N_+$, $0<N<N_+$, and $N=0$, respectively. For  $N<N_+$, the positive punctures of $u^N$ are in bijection with the negative punctures of $u^{N+1}$ and are asymptotic to the same set of Reeb orbits. The only unpaired punctures are $\Gamma_+$. If $(\gamma_1,\dots,\gamma_l)$ are Reeb orbits corresponding to the positive punctures $\Gamma_+$, then we say the holomorphic building $\mathbb{H}$ has positive ends on $(\gamma_1,\dots,\gamma_l)$. Sometimes we will write $\mathbb{H}=(u^0,u^1,\dots,u^{N_+})$ for simplicity to denote a holomorphic building. Figure \ref{holo-building} below depicts a holomorphic building of height $2$ and genus $1$ with a single positive end (puncture) and no negative ends.

\begin{figure}[h]
	\centering
	\includegraphics[width=12.2cm]{holomor.pdf}
	\caption{A holomorphic building of height $2$ and genus $1$ with one positive end. }\label{holo-building}
\end{figure}


A holomorphic building of height $N_+$ with $k$ \textbf{marked points} is a tuple 
\[\mathbb{H}=(u, \Sigma,j,\Delta_{\mathrm{br}}\cup\Gamma_+, \Delta_{\mathrm{nd}} , \Theta,\phi, L),\]
where $(u, \Sigma,j,\Delta_{\mathrm{br}}\cup\Gamma_+, \Delta_{\mathrm{nd}},\phi, L)$ is a holomorphic building of height $N_+$ and $\Theta$ is a finite ordered set of $k$ distinct points on $\Sigma\setminus (\Gamma_+\cup\Delta_{\mathrm{nd}}\cup \Delta_{\mathrm{br}}).$ A holomorphic building $(u, \Sigma,j,\Delta_{\mathrm{br}}, \Delta_{\mathrm{nd}} , \Theta,\phi, L)$ with marked points is \textbf{stable} if:
\begin{enumerate}
	\item There is no $N\in\{0,1,\dots,N_+\}$ for which the $N$th level $u^N$ consists entirely of a disjoint union of trivial cylinders over closed Reeb orbits without marked points or nodes on them.
	\item For each $N\in\{0,1,\dots,N_+\}$, $\Sigma^N$ does not contain any sphere with
	less than three special points (punctures, nodal or marked points), nor a torus
	without special points, on which $u^N$ is constant.
\end{enumerate}

By definition, a holomorphic building $\mathbb{H}$ consists of smooth connected punctured curves lying in different levels of a split completed symplectic cobordism. Each such curve is called a \textbf{curve component} of the building. Every holomorphic building $\mathbb{H}$ can be described as a finite \textbf{graph} where the vertices are the curve components and there is an edge between two curve components if they share a node or an asymptotic Reeb orbit. If $\mathbb{H}$ has genus $0$, then the corresponding graph is a tree.
\begin{definition}[Homology class of a holomorphic building {\cite[Section 2.6]{Cieliebak2018}}]\label{homologyclassofbuilding}
Consider a symplectic cobordism $(X,\omega)$ and a closed contact type separating hypersurface $Y\subset X$. As mentioned above, consider the
 split completed cobordism $\bigsqcup (\widehat{X}_N,\hat{\Omega}_N,\tilde{\Omega}_N, J_N)$, where
\begin{equation*}
	(\widehat{X}_N,\hat{\Omega}_N,\tilde{\Omega}_N, J_N):=
	\begin{cases}
		(\mathbb{R}\times\partial^+ X_+,d(e^r\lambda_+),d\lambda_+, J_{\partial^+ X_+}) & \text{for } N\in \{M+1,M+2,\dots,N_+\},\\
		(\widehat{X}_+,\hat{\omega}^+,\tilde{\omega}^+, J_+) & \text{for } N=M,\\
		
		(\mathbb{R}\times Y,d(e^r\lambda),d\lambda, J_Y)& \text{for } N\in \{1,2,\dots,M-1\},\\
		(\widehat{X}_-, \hat{\omega}^-, \tilde{\omega}^-, J_-) & \text{for } N=0.\\
	\end{cases}
\end{equation*}
Let $\overline{X}_N$ denote the compact manifold obtained from the symplectic completion $\widehat{X}_N$ by attaching copies of $\{\infty\}\times Y$ and $\{-\infty\}\times Y$ to its positive and negative cylindrical ends, respectively. By gluing the positive boundary components of  $\overline{X}_N$ to the negative boundary components of $\overline{X}_{N+1}$, we get a compact topological space 
\[\overline{X}:=\overline{X}_0\cup_{Y}\cdots\cup_Y \overline{X}_{N_+}.\]
The space $\overline{X}$ is homeomorphic to $X$. So we can identify homology class in  $\overline{X}$ with homology classes in $X$. 

Given a holomorhphic building $\mathbb{H}=(u, \Sigma,j,\Gamma^+\cup \Gamma^-,\Delta_{\mathrm{br}}, \Delta_{\mathrm{nd}}, \phi, L)$ in $\bigsqcup (\widehat{X}_N,\hat{\Omega}_N,\tilde{\Omega}_N, J_N)$. As in Definition \ref{homologyclass}, the $N$-th level $u^N:\Sigma^N\to (\widehat{X}_N, J_N)$ yeilds a continous map 
\[\bar{u}^N:\bar{ \Sigma}^N\to \overline{X}_N.\]
 By definition of holomorphic building, the positive boundary components of $\bar{u}^N$ correspond to the negative boundary components of $\bar{u}^{N+1}$. Gluing the positive boundary components of $\bar{u}^N$ with the negative boundary components of   $\bar{u}^{N+1}$ for all $N$, we get a continuous map 
\[\bar{u}:\bar{\Sigma}\to \overline{X}.\]
 We say the holomorphic building $\mathbb{H}$ represents the relative homology class $A\in H_2(X, \Gamma^+\cup \Gamma^-, \mathbb{Z})$ and write $[\mathbb{H}]=A$, if $[\bar{u}]=A$ in the sense of Definition \ref{homologyclass}.
 
\end{definition}
\subsection{Neck-stretching and Gromov--Hofer compactness} 
Consider a closed connected contact type hypersurface $Y$ in a closed symplectic manifold $(X,\omega)$, i.e., $Y$ is a closed smooth submanifold of codimension $1$ such that there exists a vector field $V$ defined on a neighborhood of $Y$ that is transverse to $Y$ and satisfies $\mathcal{L}_V\omega=\omega$, where $\mathcal{L}_V$ denotes the Lie derivative with respect to $V$. The vector field $V$ induces a contact form $\lambda$ on $Y$ defined by $\lambda:=\omega(V,\cdot)$. Suppose that $Y$ separates $X$ into two pieces, i.e., $X\setminus Y$ has two connected components. One of these components,  denoted by $(X_+, \omega^+=\omega|_{X_+})$, is a symplectic cobordism with negative boundary $(Y, \lambda)$ and empty positive boundary. The other connected component,  denoted by $(X_-, \omega^-:=\omega|_{X_-})$,  is a symplectic cobordism with positive boundary $(Y, \lambda)$ and empty negative boundary.  Choose a symplectic collar neighborhood $((-\delta,\delta)\times Y, d(e^r \lambda))$ of $Y$ in $(X,\omega)$. For $n\geq \delta$ and a suitable diffeomorphism $f:(-n,n)\to (-\delta,\delta)$ with $f'(t)>0$. By replacing   $((-\delta,\delta)\times Y, d(e^r \lambda))$  with $((-n,n)\times Y, d(e^f \lambda))$, we get a symplectic manifold $(X_n, \omega_n)$ that is symplectomorphic to the original $(X,\omega)$. Let $J_n$ be an almost complex structure on $(X_n, \omega_n)$ that agrees with a fixed $\omega$-compatible $J$ on $X$ outside the collar. Moreover, assume $J_n$ restricted to the neck $((-n,n)\times Y, d(e^f \alpha))$ agrees with a fixed  SFT-admissible almost complex structure $J_Y$ on  $((-\infty,\infty)\times Y, d(e^r \alpha))$. As $n \to \infty$, the family $J_n$ degenerates to SFT-admissible almost complex structures $J_+, J_-$ on the symplectic completions $(\widehat{X}_{+},\widehat{\omega}^+)$ and $(\widehat{X}_{-},\widehat{\omega}^-)$, respectively,  that agree with $J$ away from the cylindrical ends and with $J_Y$ on the cylindrical ends $\mathbb{R}_{\pm}\times Y$, and it degenerates  to $J_Y$ on the symplectization $(\mathbb{R}\times  Y,d(e^r\lambda))$. Furture detials can be found  in {\cite[Pages 13-14]{Cieliebak_2005}} and  {\cite[Section 3.4]{MR2026549}}.

\begin{theorem}[{\cite[Theorem 2.9]{Cieliebak_2005, Cieliebak2018}}]\label{sft}
In the setting above, suppose the contact type hypersurface $Y$ in $(X,\omega)$ is such that all closed Reeb orbits on it are non-degenerate (or Morse--Bott). For a nonnegative integer $N_+$, consider the split completed cobordism $\bigsqcup (\widehat{X}_N,\hat{\Omega}_N,\tilde{\Omega}_N, J_N)$, where
\begin{equation*}
	(\widehat{X}_N,\hat{\Omega}_N,\tilde{\Omega}_N, J_N):=
	\begin{cases}
		
		(\widehat{X}_+,\hat{\omega}^+,\tilde{\omega}^+, J_+) & \text{for } N=N_+,\\
		
		(\mathbb{R}\times Y,d(e^r\lambda),d\lambda, J_Y)& \text{for } N\in \{1,2,\dots,N_+-1\},\\
		(\widehat{X}_-, \hat{\omega}^-, \tilde{\omega}^-, J_-) & \text{for } N=0.\\
	\end{cases}
\end{equation*}
	
	Let $J_n$ be a family of almost complex structures on $(X,\omega)$ realized through neck-stretching along $Y$ in $(X,\omega)$. Every sequence of genus $g$ closed $J_n$-holomorphic curves  $u_n:(\Sigma,j)\to (X_n,\omega_n, J_n)$ in a fixed homology class $A\in H_2(X, \mathbb{Z})$ admits a subsequence, still denoted by $u_n$, that converges to a holomorphic building $\mathbb{H}=(u, \Sigma,j,\Delta_{\mathrm{br}}, \Delta_{\mathrm{nd}}, \phi, L)$ of height $N_+$ and genus $g$ in  $\bigsqcup (\widehat{X}_N,\hat{\Omega}_N,\tilde{\Omega}_N, J_N)$, for some integer $N_+\geq 0$, in the following sense:
	there is a sequence of orientation preserving diffeomorphisms  $\psi_n: (\Sigma,j)\to (\Sigma,j)$ and numbers $-n=r_n^{(0)}<r_n^{(1)}<\dots< r_n^{(N_+)}=0$ such that:
	\begin{enumerate}
		\item $\psi_n^*j\to j$ in $C^{\infty}_{loc}$ on $\Sigma\setminus  \Delta_{\mathrm{br}}$;
		\item For each $N\in\{0,1,\dots,N_+-1\}$, $r_n^{(N+1)}-r_n^{(N)}\to \infty$;
		\item For each $N\in\{0,1,\dots,N_+\}$, $u_n^N\circ \psi^{-1}_n\to u^N$ in $C^{\infty}_{loc}$ on $\Sigma^N$, where $u_n^N$ denotes the restriction of $u_n$ to $\Sigma^N$ shifted by $-r_n^{(N)}$;
		\item $\int_{\Sigma}u_n^*\omega_n \to \int_{\Sigma^{N_+}}(u^{N_+})^*\widehat{\omega}^+ $;
		\item The  holomorphic building $\mathbb{H}$ represents the homology class $A$, i.e., $[\mathbb{H}]=A\in H_2(X, \mathbb{Z})$ (cf. Definition \ref{homologyclassofbuilding}).
		
	\end{enumerate}
	
\end{theorem}
\begin{lemma}{\cite[Lemma 2.6]{Cieliebak_2005}}\label{nonconstantholclass}
Suppose	a nonconstant holomorphic building $\mathbb{H}$ represents the homology class $A\in H_2(X,\mathbb{Z})$ (cf. Definition \ref{homologyclassofbuilding}), then 
\[\langle [\omega],A\rangle>0.\]
\end{lemma}
\begin{remark}
The theorem above is stated when $(X,\omega)$ is a closed symplectic manifold. We will need it in the proof of Theorem \ref{extremal-lag-ball} and Theorem  \ref{intersection01}. In the proof of Theorem \ref{extremal-lag-toric} in Section \ref{sectiontoricproof}, we will need a statement for the case where $(X,\lambda)$ is a Liouville domain, $Y$ is a separating contact type hypersurface in $\operatorname{int}(X)$, and $u_n$ is a sequence of holomorphic planes asymptotic to some fixed Reeb orbit on $\partial X$ in the symplectic completion $\widehat{X}$. The above theorem also holds in this case  \cite{MR2026549}. In particular, the following holds.

  Let $\gamma$ be a closed Reeb orbit on $\partial X$. Let $J_n$ be a sequence of SFT-admissible almost complex structures on the completed cobordism $\widehat{X}$ obtained via stretching along the contact type hypersurface $Y$. Let  $u_n$ be a sequence of $J_n$-holomorphic planes in $\widehat{X}$ asymptotic to $\gamma$. Then a subsequence of $u_n$ convergence to a holomorphic building  $\mathbb{H}=(u^0,,u^1,\dots,u^{N_+})$ in the split completed cobordism $\bigsqcup (\widehat{X}_N,\hat{\Omega}_N,\tilde{\Omega}_N, J_N)$, for some nonnegative integers $N_+$ and $M$, where
\begin{equation*}
	(\widehat{X}_N,\hat{\Omega}_N,\tilde{\Omega}_N, J_N):=
	\begin{cases}
		(\mathbb{R}\times\partial X,d(e^r\lambda),d\lambda, J_{\partial X}) & \text{for } N\in \{M+1,M+2\dots,N_+\},\\
		(\widehat{X}_+,\hat{\omega}^+,\tilde{\omega}^+, J_+) & \text{for } N=M,\\
		
		(\mathbb{R}\times Y,d(e^r\lambda_Y),d\lambda_Y, J_Y)& \text{for } N\in \{1,2,\dots,M-1\},\\
		(\widehat{X}_-, \hat{\omega}^-, \tilde{\omega}^-, J_-) & \text{for } N=0.\\
	\end{cases}
\end{equation*}
The holomorphic building $\mathbb{H}=(u^0,,u^1,\dots,u^{N_+})$ has only one positive (unpaired) end on $\gamma$, i.e., $u^{N_+}$ has a single positive puncture which is asymptotic to $\gamma$. Moreover, the holomorphic building has genus zero. For details; see  \cite{MR2026549}.
\end{remark}
\section{Gromov--Witten invariants with local tangency constraints}\label{Chapter03}
In this section, we explain three variants of Gromov--Witten invariants for a symplectic manifold $(X,\omega)$. The first variant, coming from the work of Cieliebak--Mohnke \cite{Cieliebak_2007}, counts spheres in $X$ that satisfy the holomorphic curve equation $(du)^{0,1}=0$ and carry a generic local tangency constraint with respect to a real codimension $2$ submanifold $D$ (also known as a local divisor). We mention some cases where these invariants have been computed.  The second variant, introduced by Tonkonog \cite{Tonkonog:2018aa}, is defined to be the count of spheres in $X$ satisfying a Hamiltonianly perturbed holomorphic curve equation $(du-  X_H \otimes \beta )^{0,1}=0$ and carrying a generic local tangency constraint with respect to a local divisor. The third variant, coming from the work of McDuff--Siegel \cite{McDuff:2021aa}, counts punctured holomorphic spheres in a completed symplectic cobordism satisfying a generic local tangency constraint with respect to a local divisor. We do not claim any original result in this section but instead produce an adaptation to our context of the existing works.
\subsection{Local tangency constraint}
Consider a closed symplectic manifold $(X, \omega)$ and a point $p\in X$. Let $O(p)$ denote a small unspecified neighborhood of $p$ in $X$. Let $D\subset O(p)$ denote a local real codimension-$2$ submanifold containing the point $p$. We define 
\begin{equation*}
	\mathcal{J}_D(X,\omega):=\left\{J: 
	\begin{array}{l}
		\text{$J$ is a smooth  $\omega$-compatible almost complex structure on $X$},\\
		\text{$J$ is integrable in $O(p)$} ,\\
		\text{$D$ is $J$-holomorphic.} \\
		
	\end{array}
	\right\}.
\end{equation*}
\begin{definition}\label{tangdef}
Let $z_0\in \mathbb{CP}^1$, $J \in \mathcal{J}_{D}(X,\omega)$ and  $u:(\mathbb{CP}^1,i)\mapsto (X,J)$ be a $J$-holomorphic sphere  with $u(z_0)=p$. Choose a holomorphic function $g:O(p)\to \mathbb{C}$ such that $g(p)=0$, $dg(p)\neq 0$,  and $D=g^{-1}(0)$. Choose a holomorphic coordinate chart $h:\mathbb{C}\to \mathbb{CP}^1$ such that $h(0)=z_0$. For $k\in \mathbb{Z}_{\geq 1}$, we say  the curve $u$ satisfies the tangency constraint $\ll \mathcal{T}_D^{k-1}p\gg$ at $z_0$  if the function $g\circ u\circ h: \mathbb{C}\to \mathbb{C}$ satisfies
\[\frac{d^i}{d^iz}(g\circ u\circ h)\bigg|_{z=0}=0, \]
for all  $i=0,1,\dots,k-1.$
\end{definition}
 We denote the maximal such $k$ by $\operatorname{Ord}(u,D,z_0)$, which could be infinite, and call it the contact order of $u$ with $D$ at $p$. This notion of tangency does not depend on the choice of the functions $h$ and $g$. It only depends on the germ of $D$ near $p$; see {\cite[Section $6$]{Cieliebak_2007}} for a proof. 

By {\cite[Lemma 7.1]{Cieliebak_2007}},  the tangency constraint $\ll \mathcal{T}_D^{k-1}p\gg$ can be interpreted as a local intersection number of the image of $u$ with the divisor $D$ as follows: Choose\footnote{The symplectic form $\omega$ is exact on $O(p)$, so $u$ cannot be contained in the divisor $D$ by Stokes' theorem.} a small ball $B$ around $z_0$ such that $ u^{-1}(D)\cap B=\{z_0\}$. Smoothly perturb $u|_{B}$ away from $\partial B$ to make it transverse to $D$. The signed count of transverse intersections of $u|_{B}$ with $D$ equals $k=\operatorname{Ord}(u,D,z_0)$.

We are interested in defining Gromov--Witten-type invariants that count curves carrying a local tangency constraint with respect to a local divisor. For this, it is crucial to understand what can happen to the constraint $\ll \mathcal{T}_D^{k-1}p\gg$ as a curve satisfying this constraint degenerates to a nodal curve in the sense of Gromov {\cite[Definition 5.2.1]{MR2954391}}. The following lemma describes that we do not lose the constraint $\ll \mathcal{T}_D^{k-1}p\gg$; it just gets distributed over the non-constant components attached to a ghost in the nodal configuration.

\begin{lemma}[{\cite[Lemma 7.2]{Cieliebak_2007}}]\label{tang}
Let $J \in \mathcal{J}_{D}(X,\omega)$ and $u_n: \mathbb{CP}^1\to (X,J)$  be a sequence of $J$-holomorphic spheres carrying the constraint $\ll \mathcal{T}_D^{k-1}p\gg$ for each $n\in \mathbb{Z}_{\geq 1}$. Suppose that $u_n$ degenerates to a nodal configuration, denoted by $u$, in the Gromov topology. Suppose the constrained marked point lies on a ghost component $\bar{u}$ in $u$. Let $\{u_i\}_{i=1,2,\dots q}$ be the non-constant components of $u$ that are attached to $\bar{u}$ or attached to $\bar{u}$ through ghost components. Let $z_i$ be the special point of $u_i$ that realizes the node with $\bar{u}$ or the node with a ghost component attached to $\bar{u}$ through ghost components. Then we have 
	\[\sum_{i=1}^{q}\operatorname{ Ord}(u_i,z_i, D)\geq k.\] 
\end{lemma}

\subsection{Moduli spaces of rational curves with local tangency constraints}
This section defines moduli spaces consisting of rational curves carrying a local tangency constraint  $\ll \mathcal{T}_D^{k-1}p\gg$. We state some results regarding their regularity without providing proofs.

Let $k\in \mathbb{Z}_{\geq 1}$, $A\in H_2(X,\mathbb{Z})$, and  $J\in \mathcal{J}_{D}(X,\omega)$. Define 
\begin{equation*}
	\mathcal{M}^{J}_{X,A} \ll \mathcal{T}_D^{k-1}p\gg:=\left\{
	\begin{array}{l}
		u:(\mathbb{CP}^1,i)\to (X,J),\\
		du\circ i=J\circ du ,\\
		u(0)=p \text{  and satisfies $\ll \mathcal{T}_D^{k-1}p\gg$ at $0$}, \\
		
		u_{*}[\mathbb{CP}^1]=A\\
		
\end{array}
	\right\}\bigg/\sim.
\end{equation*}
where $u_1\sim u_2$ if and only if  $u_1=u_2\circ \phi$ for some $\phi \in \operatorname{Aut(\mathbb{CP}^1,0)}$. We denote the analogous moduli space of unparametrized somewhere injective curves by  $\mathcal{M}^{J,s}_{X,A} \ll \mathcal{T}_D^{k-1}p\gg $.

\begin{theorem}[{\cite[Proposition 6.9]{Cieliebak_2007}}]\label{smoothness}
Let $(X,\omega)$ be a closed symplectic manifold of dimension $2n$. For generic $J\in \mathcal{J}_{D}(X,\omega)$, the moduli space $\mathcal{M}^{J,s}_{X,A} \ll \mathcal{T}_D^{k-1}p\gg $ is a smooth oriented manifold of dimension
	\[2(n-3)+2c_1(A)-2n+2-2(k-1)=2c_1(A)-2-2k. \]
	Moreover, any two generic $J_0, J_1\in \mathcal{J}_{D}(X,\omega)$ can be connected by a path $\{J_t\}_{t\in[0,1]}$ such that the parametric moduli space 
	\[\mathcal{M}^{\{J_t\},s}_{X,A} \ll \mathcal{T}_D^{k-1}p\gg:=\bigg\{(t,u): t\in [0,1], u\in \mathcal{M}^{J_t,s}_{X,A} \ll \mathcal{T}_D^{k-1}p\gg\bigg\} \]
	is a smooth oriented manifold with boundary
	\[\partial \mathcal{M}^{\{J_t\},s}_{X,A} \ll \mathcal{T}_D^{k-1}p\gg=\mathcal{M}^{J_0,s}_{X,A} \ll \mathcal{T}_D^{k-1}p\gg\cup\, \mathcal{M}^{J_1,s}_{X,A} \ll \mathcal{T}_D^{k-1}p\gg.\]
The boundary orientation agrees with the orientation of $\mathcal{M}^{J_1,s}_{X,A} \ll \mathcal{T}_D^{k-1}p\gg $ and is opposite to the orientation of $\mathcal{M}^{J_0,s}_{X,A} \ll \mathcal{T}_D^{k-1}p\gg$.
\end{theorem}

\subsection{Gromov--Witten invariants with local tangency constraints} \label{loctangGW}
Following \cite{Cieliebak_2005}, we explain that curves with local tangency constraints lead to a definition of Gromov--Witten-type invariants.
\begin{definition}
Let $(X,\omega)$ be a closed $2n$-dimensional symplectic manifold. We say $X$ is monotone if there exist $k>0$ such that
\[\ c_1(X)=k[\omega], \]
where $c_1(X)$ is the first Chern class of $X$ and $[\omega]$ denotes the cohomology class of the symplectic form $\omega$. 
\end{definition}
\begin{definition}
Let $(X,\omega)$ be a closed $2n$-dimensional symplectic manifold. We say $X$ is semipositive if for every $A\in \pi_2(X)$ 
\[\big(\omega(A)>0,\ c_1(A)\geq 3-n \big)\text{ implies } c_1(A)\geq 0. \]
\end{definition}
Semipositive closed symplectic manifolds occur in abundance. Examples of such manifolds in various dimensions are given below.
\begin{itemize}
\item The complex projective space $(\mathbb{CP}^n,\omega_{\mathrm{FS}})$ is semipositive for every $n\in \mathbb{Z}_{\geq 1}$.
\item Every closed symplectic manifold of dimension less or equal to $6$ is semipositive.

\item Let $\omega_i$ be an area form on $\mathbb{CP}^1$. Then $\mathbb{CP}^1 \times \mathbb{CP}^1 \times \mathbb{CP}^1 \times \mathbb{CP}^1$ with $\omega_1\oplus\omega_2\oplus\omega_3\oplus\omega_4$ is seminpostive.
\item For $n>4$, $\mathbb{CP}^1 \times \mathbb{CP}^1 \times \dots\times \mathbb{CP}^1$ with symplectic form $\omega_1\oplus\omega_2\oplus\dots \oplus\omega_n$ is not semipositive if not all $\omega_i$ give the same volume to $\mathbb{CP}^1$.	
\end{itemize}
\begin{proposition}\label{positive}
	A closed semipositive symplectic manifold $(X,\omega)$ does not possess somewhere injective $J$-holomorphic spheres of negative Chern number for $J$ appearing in a generic $1$-parameter family of $\omega$-compatible almost complex structures.
\end{proposition}
\begin{proof}
Let  $J$ be an $\omega$-compatible almost complex structure on $X$ and 
	\[u:\mathbb{CP}^1\to X\] be a somewhere injective $J$-holomorphic sphere. Let $D_u$ denote the linearization of the nonlinear Cauchy Riemann operator at $u$. Note that $\operatorname{dim}(\operatorname{Aut}(\mathbb{CP}^1))=6$
	implies $\operatorname{dim}(\operatorname{ker}(D_u))\geq 6$.  If $\operatorname{dim}(\operatorname{coker}(D_u))\leq 1$, then
\[\operatorname{ind}(u):=2n+2c_1(u_{*}[\mathbb{CP}^1])=\operatorname{dim}(\operatorname{ker}(D_u))-\operatorname{dim}(\operatorname{coker}(D_u))\geq 5 .\]
This means
	\[c_1(u_{*}[\mathbb{CP}^1])\geq 3-n.\]
Since $(X, \omega)$ is semipositive, we have $c_1(u_{*}[\mathbb{CP}^1])\geq 0$. If $J$ belongs to a generic $1$-parameter family of $\omega$-compatible almost complex structures, then $\operatorname{dim}(\operatorname{coker}(D_u))\leq 1 $ holds by a standard transversality arguments \cite{MR2954391}.\qedhere
\end{proof}

\begin{theorem}\label{gromov-witten-tang}
Let $(X,\omega)$ be a semipositive closed symplectic manifold. Let $A\in H_2(X,\mathbb{Z})$ with $c_1(A)\geq 2$. The following hold.
	\begin{itemize}
		\item For generic $J\in \mathcal{J}_{D}(X,\omega)$, the moduli space $\mathcal{M}^{J}_{X, A} \ll \mathcal{T}_D^{c_1(A)-2}p\gg $ is a compact oriented smooth manifold of dimension $0$.
		\item The signed count
		\[\mathcal{N}^{}_{X, A}\ll \mathcal{T}^{c_1(A)-2}\gg:=\# \mathcal{M}^{J}_{X,A} \ll \mathcal{T}_D^{c_1(A)-2}p\gg\]
does not depend on the choice of $J\in \mathcal{J}_{D}(X,\omega)$, $p\in X$, and the local  divisor $D$. 
	\end{itemize}
\end{theorem}
\begin{proof}
By Theorem \ref{smoothness}, it is enough to prove that for generic $J\in  \mathcal{J}_{D}(X,\omega)$ the moduli space $ \mathcal{M}^{J}_{X, A} \ll \mathcal{T}_D^{c_1(A)-2}p\gg $ is compact and consists of somewhere injective curves.   Let $u\in  \mathcal{M}^{J}_{X, A} \ll \mathcal{T}_D^{c_1(A)-2}p\gg $ and suppose on the contrary that $u$ is not somewhere injective.  Then we can write  $u=\bar{u}\circ \phi$ for some  $m$-fold branched cover  $\phi:\mathbb{CP}^1\to \mathbb{CP}^1 $ and a simple $J$-holomorphic sphere $\bar{u}$, where $m$ is a postive integer. The sphere $\bar{u}$ represents the homology class $A/m$. Next we prove that $\bar{u}$ satisfies the constraint $\ll \mathcal{T}_D^{k-1}p\gg $ for some integer $k\geq \lceil (c_1(A)-2)/m\rceil$.  Let $u$ satisfy the constraint $\ll \mathcal{T}_D^{c_1(A)-2}p\gg$  at $z_{u}\in \mathbb{CP}^1$ and define $z_{\bar{u}}:=\phi(z_{u})$.  Let $l=\operatorname{Ord}(\phi, D, z_{u})$. Note that $\leq m$. Choose holomorphic coordinates around $z_{\bar{u}}\in \mathbb{CP}^1$ identifying $z_{\bar{u}}$ with $0\in \mathbb{C}$. Since the almost complex structure $J$ on $X$ is integrable near $p$ and $D$ is a complex hypersurface, we can choose holomorphic coordinates around $p\in X$ that identifies $p$ with $0 \in \mathbb{C}^{n}$ and $D$ with $\{0\}\times \mathbb{C}^{n-1}$ near the origin. Let $\pi_1:\mathbb{C}^{n}\to \mathbb{C}$ be the projection defined by $\pi_1(z_1,\dots,z_n)=z_1$. Expressing $\phi$, $\bar{u}$, and $u$ in these coordinates, we get
\[\phi(z)=z^l,\]
\begin{equation}\label{branchtan}
\pi_1\circ \bar{u}(z)=\sum_{j=0}^{\infty}b_jz^j,
\end{equation}
and 
 \[\pi_1\circ u(z)=\pi_1\circ \bar{u}(z)\circ \phi(z)=\sum_{j=0}b_jz^{lj}.\]
Since $u$ satisfies the constraint $\ll \mathcal{T}_D^{c_1(A)-2}p\gg$, we have 
\[\frac{d^i}{d^iz}(\pi_1\circ u)\bigg|_{z=0}=0, \]
for all  $i=0,1,\dots,c_1(A)-2$. This implies $b_0=0$ and
\[\sum_{j=1}^{\infty}lj(lj-1)(lj-2)\cdots(lj+1-i)b_jz^{lj-i}\bigg|_{z=0}=0,\]
for all  $i=0,1,\dots,c_1(A)-2$. If for a given $j$ there exists some $i\in\{0,1,\dots, c_1(A)-2\}$ such that $jl-i=0$, then $b_j=0$. But this is the case for every $j=1,2,\dots, \lfloor (c_1(A)-2)/l\rfloor$; one can take $i=lj$.  Therefore, $b_j=0$ for every $j=1,2,\dots, \lfloor (c_1(A)-1)/l\rfloor$. This reduces (\ref{branchtan}) to 
\[\pi_1\circ \bar{u}(z)=\sum_{j=k}^{\infty}b_jz^j,\]
for some integer $k\geq \lfloor (c_1(A)-2)/l\rfloor$+1. Thus, $\bar{u}$ satisfies the constraint  $\ll \mathcal{T}_D^{k-1}p\gg $ for some integer $k\geq  \lfloor \frac{c_1(A)-2}{l}\rfloor+1\geq \lfloor \frac{c_1(A)-2}{m}\rfloor+1$ as $m\geq l$.

The index of $\bar{u}$ is given by
\[\operatorname{ind}(\bar{u})=2\frac{c_1(A)}{m}-2-2k\leq 2\frac{c_1(A)}{m}-4-2\bigg\lfloor \frac{c_1(A)-2}{m}\bigg\rfloor.\]
This implies
\[\operatorname{ind}(\bar{u})\leq -4+\frac{2}{m}\leq -2,\]
for $m>1$. Since $J$ is generic and $\bar{u}$ is simple, $\operatorname{ind}(\bar{u})<0$ cannot happen. So we must have $m=1$, i.e., $u$ is somewhere injective.
	
For compactness, suppose on contrary that a sequence $u_n$  in $\mathcal{M}^{J}_{X, A} \ll \mathcal{T}_D^{c_1(A)-2}p\gg $ degenerates to a nodal configuration $u$  in the sense of Gromov {\cite[Definition 5.2.1]{MR2954391}}.  We prove that such a degeneration is a phenomenon of codimension at least $2$ and hence cannot happen when  $J\in \mathcal{J}_{D}(X,\omega)$ is generic or when $J$ appears in a generic $1$-parameter family of almost complex structures. 

The constrained marked point of $u_n$ is inherited by a component of the limiting nodal configuration $u$. Suppose the constrained marked point lies on a ghost component, denoted by $\bar{u}$. Let $\{u_i\}_{i=1,2,\dots, q}$ be the non-constant components of $u$ that are attached to $\bar{u}$ or attached to $\bar{u}$ through ghost components. Let $z_i$ be the special point of $u_i$ that realizes the node with $\bar{u}$ or the node with a ghost component attached to $\bar{u}$ through ghost components. By Lemma \ref{tang}, we have 
	\[\sum_{i=1}^{q}\operatorname{ Ord}(u_i,z_i, D)\geq c_1(A)-1.\] 
	
Deleting $\bar{u}$ from $u$ yields at least $q$ connected nodal configurations that are mutually disjoint in the sense that no two of them share a node. Let $C_i$ be the connected nodal configuration that contains $u_i$. Let $C_i$ represent the homology class $A_i$, i.e., the sum of the homology classes represented by the components of $C_i$ is equal to $A_i$. By {\cite[Proposition 6.1.2]{MR2954391}},
 every $C_i$ has an underlying simple nodal configuration $\bar{C}_i$  that is obtained by replacing every multiply-covered component of $C_i$ by its underlying somewhere injective component. Let  $\bar{C}_i$ represent the homology class $\bar{A}_i$. Then, for each $i$, we have 
	\[A_i= \sum_{j=1}^{l}m_j\bar{A}_{i,j}\]
	where $l, m_j $ are some nonnegative integers (that depend on $i$) and $\bar{A}_i=\sum_{j=1}^{l}\bar{A}_{i,j}$. By Proposition \ref{positive}, we have
\[c_1(\bar{A}_i)= \sum_{j=1}^{l}c_1(\bar{A}_{i,j})\leq \sum_{j=1}^{l}m_jc_1(A_{i,j})=c_1(A_i).\]
This implies
	\[\operatorname{ind}(\bar{C}_i)\leq 2c_1(A_i)-2-2\operatorname{ Ord}(u_i,z_i, D).\]
Hence
	\[\sum_{i=1}^{q}\operatorname{ind}(\bar{C}_i)\leq 2\sum_{i=1}^{q}c_1(A_i)-2q-2\sum_{i=1}^{q}\operatorname{ Ord}(u_i,z_i, D).\]
This implies
\[\sum_{i=1}^{q}\operatorname{ind}(\bar{C}_i)\leq 2c_1(A)-2q-2(c_1(A)-1)=-2(q-1).\]
	If  $q\geq 2$, then 
\[\sum_{i=1}^{q}\operatorname{ind}(\bar{C}_i)\leq -2 .\]
which cannot happen when $J$ is generic. 
If $q=1$, since $\bar{C}_1$ has a node, and each node reduces the index by at least $2$ by {\cite[Theorem 6.2.6]{MR2954391}}. Thus, $\operatorname{ind}(\bar{C}_1)<0$ which is a contradiction as $J$ is generic.
\end{proof}
\begin{example}
The standard complex structure on $(\mathbb{CP}^n,\omega_{\mathrm{FS}})$, denoted by $J_{\mathrm{std}}$, is regular by  {\cite[Proposition 7.4.3]{MR2954391}}. Also,  for the homology class $[\mathbb{CP}^1]$ in $\mathbb{CP}^n$. we have $c_1([\mathbb{CP}^1])=n+1$. So 
for any $d\in \mathbb{N}$ we have 
\[\mathcal{N}_{\mathbb{CP}^n, d[\mathbb{CP}^1]}\ll \mathcal{T}^{dn+d-2}\gg=\# \mathcal{M}^{J_{\mathrm{std}}}_{\mathbb{CP}^n,d[\mathbb{CP}^1]} \ll \mathcal{T}_D^{dn+d-2}p\gg.\]

The complex structure $J_{\mathrm{std}}$ induces a complex structure on the moduli space $\mathcal{M}^{J_{\mathrm{std}}}_{\mathbb{CP}^n,d[\mathbb{CP}^1]} \ll \mathcal{T}_D^{dn+d-2}p\gg$. This complex structure is compatible with its canonical orientation. Therefore, each element in $\mathcal{M}^{J_{\mathrm{std}}}_{\mathbb{CP}^n,d[\mathbb{CP}^1]} \ll \mathcal{T}_D^{dn+d-2}p\gg$ carries a positive orientation; see {\cite[Remark 3.2.5--6]{MR2954391}} for details. So the signed count $\mathcal{N}_{\mathbb{CP}^n, d[\mathbb{CP}^1]}\ll \mathcal{T}_D^{dn+d-2}\gg$ equals the unsigned count, i.e.,  the number of elements in $\mathcal{M}^{J_{\mathrm{std}}}_{\mathbb{CP}^n,d[\mathbb{CP}^1]} \ll \mathcal{T}_D^{dn+d-2}p\gg$.
\end{example}
\begin{theorem}[{\cite[Proposition 3.4]{Cieliebak2018} }]\label{count}
For every $n\in \mathbb{Z}_{\geq 1}$, we have
	\[\mathcal{N}_{\mathbb{CP}^n, [\mathbb{CP}^1]}\ll \mathcal{T}^{n-1}\gg=\# \mathcal{M}^{J_{\mathrm{std}}}_{\mathbb{CP}^n,[\mathbb{CP}^1]} \ll \mathcal{T}_D^{n-1}p\gg=(n-1)!.\]
	.
\end{theorem}
\begin{corollary}\label{simplecount}
	Let $(X, \Omega)=(\mathbb{CP}^n\times T^{2m}, \omega_{\mathrm{FS}}\oplus \omega_{\mathrm{std}})$, where $T^{2m}$ is the standard ${2m}$-torus and $L:=[\mathbb{CP}^1\times\{*\}]\in H_2(X,\mathbb{Z})$. For generic $p \in X$ and generic $J\in \mathcal{J}_{D}(X,\Omega)$, the moduli space $\mathcal{M}^{J}_{X, L} \ll \mathcal{T}_D^{n-1}p\gg $ has  $(n-1)!$ elements, counted with signs.
\end{corollary}
\begin{proof}
Let $J_{\mathrm{std}}$ be the standard complex structure on $\mathbb{CP}^n$ and $J_{\mathrm{m}}$ be an integrable almost complex structure on $T^{2m}$. A map $u:\mathbb{CP}^1\to X$ written as $u=(u_1,u_2)$ is $J_{\mathrm{std}}\oplus J_{\mathrm{m}}$-holomorphic if and only $u_1:\mathbb{CP}^1 \to \mathbb{CP}^n$  is  $J_{\mathrm{std}}$-holomorphic and $u_2:\mathbb{CP}^1 \to T^{2m}$ is $J_{\mathrm{m}}$-holomorphic. Since $\pi_2(T^{2m})=0$, the holomorphic map $u_2$ has zero symplectic energy, i.e.,
	\[\int_{\mathbb{CP}^1}u_2^{*}\omega_{\mathrm{std}}=0.\]
This means $u_2$ is constant. This, together with  {\cite[Proposition 7.4.3, Lemma 3.3.2]{MR2954391}} implies that the split almost complex structure $J_{\mathrm{std}}\oplus J_{\mathrm{m}}$ is regular in the sense that the linearized Cauchy--Riemann operator at every $J_{\mathrm{std}}\oplus J_{\mathrm{m}}$-holomorphic sphere in the homology class $L$ is surjective.

Let $p=(p_1,p_2)\in \mathbb{CP}^n\times T^{2m}$ and	$D_1$ be a local divisor in $(\mathbb{CP}^n,J_{\mathrm{std}})$ at $p_1$. For the divisor $D:=D_1\times T^{2m}$ at $p$, a $J_{\mathrm{std}}\oplus J_{\mathrm{m}}$-holomorphic sphere $u=(u_1,u_2)$ in the homology class $L$ satisfies the constraint $\ll \mathcal{T}_D^{n-1}p\gg$ if and only if $u_1\in \mathcal{M}^{J_\mathrm{std}}_{\mathbb{CP}^n, [\mathbb{CP}^1]} \ll \mathcal{T}_{D_1}^{n-1}p_1\gg $ and  $u_2=p_2$. Thus,  $\mathcal{M}^{J_{\mathrm{std}}\oplus J_{\mathrm{m}}}_{X, L} \ll \mathcal{T}_D^{n-1}p\gg $ has exactly $(n-1)!$ elements by Theorem \ref{count} and each element is counted positively.
	
Choose a path of $\Omega$-compatible almost complex structures $\{J_t\}_{t\in [0,1]}$ such that $J_0=J_{st}\oplus J_m$ and $J_1=J$, where $J$ is the given generic $\Omega$-compatible almost complex structure. By Theorem \ref{smoothness}, the parametric moduli space $\mathcal{M}^{\{J_t\}}_{X, L} \ll \mathcal{T}_D^{n-1}p\gg$ is an oriented $1$-dimensional cobordism between its ends at $t=0,1$. Moreover, this cobordism is compact by {\cite[Theorem 3.2]{Faisal:2024aa}}.  So
\[\# \mathcal{M}^{J_1}_{X, L} \ll \mathcal{T}_D^{n-1}p\gg=\# \mathcal{M}^{J_0}_{X, L} \ll \mathcal{T}_D^{n-1}p\gg=(n-1)!. \qedhere\]
\end{proof}
\begin{theorem}[\cite{MR4332489}]
For every $d\in \mathbb{Z}_{\geq 1}$, we have
	\[0<\mathcal{N}_{\mathbb{CP}^2, d[\mathbb{CP}^1]}\ll \mathcal{T}^{3d-2}\gg.\]
\end{theorem}
The invariant $\mathcal{N}_{X, A}\ll \mathcal{T}^{c_1(A)-2}\gg$ can be defined in a more general setting. let  $\mathcal{M}^{J,s}_{X,A,m} \ll \mathcal{T}_D^{k-1}p\gg$ be the moduli space of unparameterized simple $J$-holomorphic spheres in $X$ in the homology class $A$  with $m$ distinct (ordered) marked points and carrying the tangency constraint $\ll \mathcal{T}_D^{k-1}p\gg$, i.e.,
\begin{equation*}
	\mathcal{M}^{J,s}_{X,A,m} \ll \mathcal{T}_D^{k-1}p\gg:=\left\{
	\begin{array}{l}
	(u,0,z_1,\dots,z_m),\\
	z_1,\dots,z_m\in \mathbb{CP}^1\setminus 0,\\
		u:(\mathbb{CP}^1,i)\to (X,J),\\
		du\circ i=J\circ du ,\\
		u(0)=p \text{  and satisfies $\ll \mathcal{T}_D^{k-1}p\gg$ at $0$}, \\
		
		u_{*}[\mathbb{CP}^1]=A,\\
		\text{$u$ is somewhere injective}
		
\end{array}
	\right\}\bigg/\operatorname{Aut(\mathbb{CP}^1,0)}.
\end{equation*}
We have a well-defined evaluation map 
\[\operatorname{ev}:\mathcal{M}^{J,s}_{X,A,m} \ll \mathcal{T}_D^{k-1}p\gg\to \overbrace{X\times X\times\cdots\times X}^\text{ $m$ times }.\] 
\begin{theorem}[{\cite[Proposition 2.2.2]{MR4332489}}]\label{counttang}
Let $(X,\omega)$ be a semipositive symplectic manifold. For generic $J\in \mathcal{J}_{D}(X,\omega)$, the evaluation map 
\[\operatorname{ev}:\mathcal{M}^{J,s}_{X,A,m} \ll \mathcal{T}_D^{k-1}p\gg\to  X\times X\times\cdots\times X\]
is a pseudocycle of dimension $2c_1(A)-2-2k+2m$. Moreover, this pseudocycle is independent (up to pseudocycle bordism {\cite[Definition 6.5.1]{MR2954391}}) of the generic $J\in \mathcal{J}_{D}(X,\omega)$, the point $p$, and the local divisor $D$ needed to define the constraint  $\ll \mathcal{T}_D^{k-1}p\gg$.
\end{theorem}

Bordism classes of pseudocycles are equivalent to
integral homology classes; see {\cite[Theorem 3.3]{Schwarz:aa}}. So by Theorem \ref{counttang}, the pseudocycle $\operatorname{ev}$ determines a unique singular homology class
\[\big[\mathcal{\bar{M}}^{J,s}_{X,A,m} \ll \mathcal{T}_D^{k-1}p\gg\big]\in H_{2c_1(A)-2-2k+2m}(X^{\times m},\mathbb{Z})\]
which is independent of the choices of $J\in \mathcal{J}_{D}(X,\omega)$, the point $p$,  and the local divisor $D$.

Consider cohomology classes $\gamma_1,\dots, \gamma_m \in H^*(X,\mathbb{Z})$ of total degree $2c_1(A)-2-2k+2m$. We have a well-defined invariant given by the homological intersection number  defined by
\[\mathcal{N}_{X, A}\ll \mathcal{T}^{k-1}, \gamma_1, \gamma_2,\dots,\gamma_m\gg:=\operatorname{PD}\big(\pi_1^*\gamma_1\cup \pi_2^*\gamma_2\cup\dots\cup \pi_m^*\gamma_m\big)\cdot \big[\mathcal{\bar{M}}^{J,s}_{M,A,m} \ll \mathcal{T}_D^{k-1}p\gg\big]\in \mathbb{Z}. \]
Here $\pi_j: X\times X\times\cdots\times X\mapsto X$ denotes projection onto the $j$th-factor and $\operatorname{PD}$ denotes the Poincar$\acute{\text{e}}$ duality isomorphism.
We can recover the invariant $\mathcal{N}_{X, A}\ll \mathcal{T}^{c_1(A)-2}\gg$ defined in Theorem \ref{gromov-witten-tang} as follows. Choose $m=1$, $k=c_1(A)-1$, and a class $\gamma\in H^2(X,\mathbb{Z})$ such that $\gamma(A)\ne 0$. By the divisor axiom of Gromov--Witten theory {\cite[Proposition 7.5.7]{MR2954391}}, we have
\begin{equation}\label{tang-invariant1000}
	\mathcal{N}_{X, A}\ll \mathcal{T}^{c_1(A)-2}\gg=\frac{1}{\gamma(A)}\mathcal{N}_{X, A,1}\ll \mathcal{T}^{c_1(A)-2},  \gamma\gg\, \in \mathbb{Z}.
\end{equation}

\subsection{Hamiltonian-perturbed Gromov--Witten invariants with local tangency constraints}\label{enumerative}
Let $(X,\omega)$ be a compact monotone symplectic manifold. Following \cite[Sections 6.3--6.4]{Cieliebak2018}, we consider a Hamiltonianly perturbed version of the Gromov--Witten invariants with local tangency constraints introduced in Section \ref{loctangGW}, in which the holomorphic curve equation is perturbed by means of a coherent perturbation depending both on the domain curve and on the target manifold $X$. Such perturbations will be used later to overcome transversality issues in the proofs of our main theorems. The case of interest in this paper is $(X,\omega)=(\mathbb{CP}^n,\omega_{\mathrm{FS}})$. 

For $\ell\geq 3$, let  $\mathcal{M}_{l}$ be the moduli space of $l$ distnict numbered marked points   $ z_1,z_2,\dots ,z_{\ell}\in \mathbb{CP}^1$ considered upto $\operatorname{Aut}(\mathbb{CP}^1,i)$.  Fix positive integers $d, k,$  and $N$. Choose a closed smooth oriented divisor $\Sigma$ in $X$ that is Poincar$\acute{\text{e}}$ dual to $Nc_1(X)$. Choose a point $p\in X\setminus \Sigma$ and local divisor $D\subset X\setminus \Sigma$ containing $p$.  Let $ \mathcal{J}_{D}(X,\omega, \Sigma)$ be the space of all  $J\in \mathcal{J}_{D}(X,\omega)$ that preserve the tangent spaces of $\Sigma$, i.e., $\Sigma$ is $J$-holomorphic. Let $J\in \mathcal{J}_{D}(X,\omega, \Sigma)$. Consider the moduli space
\begin{equation}\label{hamilto-perturb}
	\mathcal{M}^{J}_{X, Nd} (K, \mathcal{T}_D^{k-1}p, \Sigma)_d:=\left\{
	\begin{array}{l}
	(\mathbb{CP}^1,z_0, z_1,z_2,\dots ,z_{Nd}, u), \text{ where }\\
		(\mathbb{CP}^1,z_0, z_1,z_2,\dots ,z_{Nd})\in \mathcal{M}_{Nd+1},\\
		u:(\mathbb{CP}^1,i)\to (X,J),\\
		(du-K)^{0,1}=0, \ 	c_1([u])=d ,\\
		u(z_0)=p \text{  and satisfies $\ll \mathcal{T}_D^{k-1}p\gg$ at $z_0$}, \text{and} \\
		u(z_i)\in \Sigma,\text{ for all } i=1,2,\dots ,Nd.
\end{array}
	\right\}.
\end{equation}
The coherent perturbation $K$ in the holomorphic curve equation is the one introduced in {\cite[Section 6.3--4, and Page 274]{Cieliebak2018}}. It is depending on  $\mathcal{M}_{Nd+2}\times X$, rouhgly speaking, we write $K=X_H\otimes\beta$ , where
\begin{itemize}
\item  $X_H$ is the Hamiltonian vector field of a time-independent Hamiltonian $H:M\to \mathbb{R}$ that is supported near the constrained point $p$. Moreover, $H$ is chosen so that $X_H$ is transverse to the local divisor $D$.
\item $\beta$ is a $1$-form on $\mathbb{CP}^1$ supported in an annulus around the point $z_0$. The annulus depends on the stable curve, i.e., on the position of the distinct marked points $z_1,z_2,\dots,z_{Nd}$. It is chosen so that the points $z_i$ sit in the complement of the annulus on the side opposite to where $z_0$ sits.  
\end{itemize}
Note that every solution $u$ of $(du-  K )^{0,1}=0$ is purely $J$-holomorphic near $z_0$, so the tangency constraint $\ll \mathcal{T}_D^{k-1}p\gg$ at $z_0$ is well-defined. By \cite[Proposition 9.1b]{Cieliebak_2007}, for a generic coherent perturbation $K$, the moduli space in (\ref{hamilto-perturb}) is a smooth oriented manifold of the expected dimension. We note that in \cite[Proposition 9.1b]{Cieliebak_2007}, the perturbation $K$ is taken to be a compatible almost complex structure depending only on the universal curve of genus-zero curves with $l$ marked points. However, the proof carries over with only minor modifications to the more general perturbation framework defined in \cite[Sections 6.3--6.4]{Cieliebak2018} and used here. Transversality for similar moduli spaces in the symplectic field theory setting is proved in \cite[Sections 7 and 9]{Cieliebak2018}.

Define $\langle \psi_{d-2}p\rangle_{X,d} ^\bullet$ to be the signed count 
\begin{equation}\label{gravi-invariant}
	\langle \psi_{d-2}p\rangle_{X,d} ^\bullet:=\frac{1}{(Nd)!}\#\mathcal{M}^{J}_{X, Nd} (K, \mathcal{T}_D^{d-2}p, \Sigma)_d.
\end{equation}
A non-constant $J$--holomorphic sphere $u$ in $X$ with $c_1([u])=d$ has precisely $dN$ intersection points with $\Sigma$ due to the positivity of intersection. There are $(Nd)!$ ways to order these intersection points without repetition to produce the distinct order marked point $z_1,z_2,\dots, z_{Nd}$. Therefore, we divide the right-hand side of (\ref{gravi-invariant}) by $(Nd)!$. Since $(X,\omega)$ is assumed to be monotone, it is semipositive. It follows from the arguments in Theorem \ref{gromov-witten-tang},  that the count (\ref{gravi-invariant}) is well-defined and does not depend on the auxiliary choices (cf. {\cite[Section 2.3]{Tonkonog:2018aa}}).

For $(X,\omega)=(\mathbb{CP}^n,\omega_{\mathrm{FS}})$ and the indecomposable homology class $A=[\mathbb{CP}^1]\in H_2(\mathbb{CP}^n,\mathbb{Z})$, we have the following.

\begin{theorem}[{\cite[Theorem 1.2--3]{Cieliebak_2007}}]\label{countlast}
For every $n\in \mathbb{Z}_{\geq 1}$, we have
	\[\langle \psi_{n-1}p\rangle_{\mathbb{CP}^n, n+1} ^\bullet=\mathcal{N}_{\mathbb{CP}^n, [\mathbb{CP}^1]}\ll \mathcal{T}^{n-1}\gg=(n-1)!.\]
\end{theorem}
\begin{proof}

The proof follows from a cobordism argument. The count on the left-hand side
is defined using a generic coherent perturbation \(K\), while the count on the
right-hand side is defined with the trivial perturbation. Choose a
homotopy of coherent perturbations \(\{K_t\}_{t\in[0,1]}\) such that
\(K_0=K\) and \(K_1=0\). Consider the parametric moduli space
\[
\begin{aligned}
\mathcal{M}^{J}_{\mathbb{CP}^n,\, N(n+1)}
\big(\{K_t\}_{t\in[0,1]}, \mathcal{T}_D^{n-1}p, \Sigma\big)_{n+1}
:=
\big\{(t,u) \;:\;& t\in [0,1], u\in
\mathcal{M}^{J}_{\mathbb{CP}^n,\, N(n+1)}
\big(K_t, \mathcal{T}_D^{n-1}p, \Sigma\big)_{n+1}
\big\}.
\end{aligned}
\]
For a generic choice of the homotopy $\{K_t\}_{t\in[0,1]}$, this
parametric moduli space is a smooth oriented cobordism between the moduli
spaces corresponding to $K_0=K$ and $K_1=0$. Moreover, every curve in this
parametric moduli space represents the indecomposable homology class
\[
A=[\mathbb{CP}^1]\in H_2(\mathbb{CP}^n;\mathbb{Z}).
\]
It follows that no bubbling can occur, and hence the parametric moduli space
is compact. Therefore, it defines a compact oriented smooth cobordism, which
identifies the two counts.
\end{proof}
\section{Punctured spheres with local tangency constraints in Liouville domains}\label{puncturecountlio}
A Liouville domain is a pair $(X,\lambda)$, where $X$ is a compact oriented smooth manifold with boundary $\partial X$ and $\lambda$ is a $1$-form such that the exterior derivative $d\lambda$ is symplectic and the restriction of $\lambda$ to $\partial X$ is a positive\footnote{The vector field $V$ on $X$ defined by $d\lambda(V,\cdot)=\lambda(\cdot)$ points outwards along $\partial X$.} contact form. The Reeb vector field $R_\lambda$ of $\lambda|_{\partial X}$ is the unique vector field on $\partial X$ defined by 
\[d\lambda(R_\lambda,\cdot)|_{\partial X}=0 \text{ and } \lambda(R_\lambda)=1.\] 
 
A closed Reeb orbit of period $T>0$ is a smooth map $\gamma:\mathbb{R}/T\mathbb{Z}\to \partial X$, modulo translations of the domain, such that $\gamma '(t)= R_\lambda(\gamma(t))$ for all $t\in \mathbb{R}/T\mathbb{Z}$. Let $\phi^t:\partial X\to \partial X$ be the flow of $R_{\lambda}$. Its linearization induces a family of symplectic maps $d\phi^t:(\xi, d\lambda|_{\partial X}) \to (\xi, d\lambda|_{\partial X})$, where $\xi:=\operatorname{Ker}(\lambda|_{\partial X})$. A closed Reeb orbit $\gamma$  of period $T>0$ is nondegenerate \footnote{If $\gamma$ is non-degenerate, then its translates $\gamma(\cdot+\theta)$ are also nondegenerate. So the notion of non-degeneracy is well-defined.} if the linearized return map $d\phi^T(\gamma(0)):(\xi_{\gamma(0)}, d\lambda|_{\partial X}) \to (\xi_{\gamma(0)}, d\lambda|_{\partial X})$ has no eigenvalues equal to $1$. A Liouville domain $(X,\lambda)$ is nondegenerate if all closed Reeb orbits on $\partial X$ are nondegenerate.

Let $(X,\lambda)$ be a nondegenerate Liouville domain and let $\widehat{X}$ be its symplectic completion (cf. Definition \ref{completion}). Pick a point $p\in \operatorname{Int}(X)$ and a germ of smooth symplectic divisor $D$ containing $p$. We define  $\mathcal{J}_D(X,\lambda)$ to be the set of all SFT--admissible almost complex structures on the completion $\widehat{X}$ that are integrable near $p$ and preserve the tangent spaces of $D$. Let  $\gamma$ be a closed Reeb orbit on $\partial X$, $k\in \mathbb{Z}_{\geq 1}$, $J\in \mathcal{J}_D(X,\lambda)$, and  $A\in H_2(X,\gamma, \mathbb{Z})$. Define  
\begin{equation*}
	\mathcal{M}^{J,s}_{\widehat{X}, A}(\gamma)\ll \mathcal{T}_D^{k-1}p\gg:=\left\{
	\begin{array}{l}
		u:(\mathbb{C},i) \to (\widehat{X},J),\\
		du\circ i=J\circ du  ,\\
		u \text{ is asymptotic to  $\gamma$ at $\infty$,} \\
	
		u \text{ satisfies $\ll \mathcal{T}_D^{k-1}p\gg$ at $0$,}\\
		u \text{ is somewhere injective.}
	\end{array}
	\right\}\Bigg/\operatorname{Aut}(\mathbb{C},0).
\end{equation*}
By standard arguments {\cite[Proposition 3.1]{Cieliebak2018}}, the moduli space $\mathcal{M}^{J,s}_{\widehat{X}, A}(\gamma)\ll \mathcal{T}_D^{k-1}p\gg$ is a finite dimensional  oriented manifold for generic $J\in\mathcal{J}_D(X,\lambda)$. If its virtual dimension is zero, then for generic $J$, the moduli space is a set of discrete signed points.  In general, the signed count
\[\# \mathcal{M}^{J,s}_{\widehat{X}, A}(\gamma)\ll \mathcal{T}_D^{k-1}p\gg\]
can fail to be well-defined and independent of the auxiliary choices of $J$, $p$, and $D$. Firstly, the space $\mathcal{M}^{J,s}_{\widehat{X}, A}(\gamma)\ll \mathcal{T}_D^{k-1}p\gg$ may not be compact. Secondly, even if the moduli space is compact, the count $\# \mathcal{M}^{J,s}_{\widehat{X}, A}(\gamma)\ll \mathcal{T}_D^{k-1}p\gg$ may depend on the choice of data $(J, p, D)$ because non-trivial holomorphic buildings can appear as we vary $(J, p, D)$ in a $1$-parameter family. McDuff--Siegel observed that for the signed count $\# \mathcal{M}^{J,s}_{\widehat{X}, A}(\gamma)\ll \mathcal{T}_D^{k-1}p\gg$ to be well-defined and to be independent of the auxiliary choice of $(J, p, D)$, it is enough that the moduli space  $\mathcal{M}^{J,s}_{\widehat{X}, A}(\gamma)\ll \mathcal{T}_D^{k-1}p\gg$ satisfies a condition called ``formal perturbation invariance''. Roughly speaking, under this condition, non-trivial holomorphic buildings can appear as we vary the auxiliary data in a $1$-parameter family, but such buildings do not interfere with the signed count in a generic situation.
\begin{definition}[cf. {\cite[Definition 2.4.1]{McDuff:2021aa}}]\label{perturbationinvariance}
Let $\mathcal{M}^{J,s}_{\hat{X}, A}(\gamma)\ll \mathcal{T}_D^{k-1}p\gg$ be a moduli space of zero virtual dimension. We say the formal class of curves belonging to the moduli space \[\mathcal{M}^{J,s}_{\hat{X}, A}(\gamma)\ll \mathcal{T}_D^{k-1}p\gg\] is formally perturbation invariant if there exists a generic SFT-admissible almost complex structure $J_{\partial X}$ on $\mathbb{R}\times \partial X$ such that the following holds. Let $\{J_t\}_{t\in [0,1]}\subset \mathcal{J}_D(X,\lambda)$ be any generic smooth one-parameter family of SFT-admissible almost complex structures on $\widehat{X}$ satisfying $J_t|_{[0,\infty)\times \partial X}=J_{\partial X}$ for all $t$, if  $\mathbb{H}$ is a holomorphic building in the SFT compactification of the parametric moduli space $\mathcal{M}^{\{J_t\},s}_{\hat{X}, A}(\gamma)\ll \mathcal{T}_D^{k-1}p\gg$ and satisfies the following:
\begin{itemize}
\item[(H1)] Each nonconstant smooth connected component $C$ of $\mathbb{H}$ in $\widehat{X}$ is a possibly multiply
    covered curve whose underlying simple curve $\bar{C}$ satisfies  $\operatorname{ind}(\bar{C})\geq -1$.
\item[(H2)] Each nonconstant smooth connected component $C$ of $\mathbb{H}$ in the symplectization $\mathbb{R}\times \partial X$ is a possibly multiply
    covered curve whose underlying simple curve $\bar{C}$ is either a trivial cylinder or satisfies $\operatorname{ind}(\bar{C})\geq 1$.
\end{itemize}
Then, one of the following holds.
\begin{itemize}
	\item[(C1)] $\mathbb{H}\in \mathcal{M}^{\{J_t\},s}_{\hat{X}, A}(\gamma)\ll \mathcal{T}_D^{k-1}p\gg$, i.e., $\mathbb{H}$ consists of a single smooth connected  component.
	\item[(C2)] $\mathbb{H}$ has only two levels. The bottom level in $\widehat{X}$ consists of a single smooth connected simple curve $C_{\mathrm{bot}}$ with $\operatorname{ind}(C_{\mathrm{bot}})=-1$. The top level in $\mathbb{R}\times \partial X$ consists of some trivial cylinders and a simple index $1$ component $C_{\partial X}$. Moreover, the moduli space of $J_{\partial X}$-holomorphic curves representing the formal class of the curve  $C_{\partial X}$, denoted by $\mathcal{M}^{J_{\partial X},s}_{ C_{\partial X}}$, is regular and satisfies  $\# (\mathcal{M}^{J_{\partial X},s}_{ C_{\partial X}}/\mathbb{R})=0$. 
\end{itemize}
We also say that the moduli space $\mathcal{M}^{J,s}_{\hat{X}, A}(\gamma)\ll \mathcal{T}_D^{k-1}p\gg$ is formally perturbation invariant if the formal class of curves belonging to the moduli space $\mathcal{M}^{J,s}_{\hat{X}, A}(\gamma)\ll \mathcal{T}_D^{k-1}p\gg$ is formally perturbation invariant.
\end{definition}
We will need the following theorem later in this section, especially for Theorem \ref{puctured-dsik}.
\begin{theorem}[cf. {\cite[Proposition 2.4.2]{McDuff:2021aa}}]\label{perturbation-invariance}
Let $(X,\lambda)$ be a Liouville domain with nondegenerate contact boundary $(\partial X,\lambda|_{\partial X})$. Suppose the moduli space  $\mathcal{M}^{J,s}_{\hat{X}, A}(\gamma)\ll \mathcal{T}_D^{k-1}p\gg$ has zero virtual dimension and is formally perturbation invariant with respect to some generic SFT-admissible almost complex structure $J_{\partial X}$ on $\mathbb{R}\times \partial X$ (cf. Definition \ref{perturbationinvariance}). Then, for generic $J\in \mathcal{J}_D(X,\lambda)$ satisfying $J|_{[0,\infty)\times \partial X}=J_{\partial X}$, the moduli space  $\mathcal{M}^{J,s}_{\hat{X}, A}(\gamma)\ll \mathcal{T}_D^{k-1}p\gg$ is a smooth compact oriented zero-dimensional manifold  and the signed count    
	\[\# \mathcal{M}^{J,s}_{\hat{X}, A}(\gamma)\ll \mathcal{T}_D^{k-1}p\gg\]
	does not depend on the choice of $J$ satisfying $J|_{[0,\infty)\times \partial X}=J_{\partial X}$ provided that  $\mathcal{M}^{J,s}_{\hat{X}, A}(\gamma)\ll \mathcal{T}_D^{k-1}p\gg$ is regular.
\end{theorem}
\begin{remark}
In the setting of Theorem \ref{perturbation-invariance}, consider two regular almost complex structures $J_0, J_1\in \mathcal{J}_D(X,\lambda)$ satisfying \[J_0|_{[0,\infty)\times \partial X}=J_1|_{[0,\infty)\times \partial X}=J_{\partial X}.\] By standard transversality techniques, for a generic  family $\{J_t\}_{t\in [0,1]}\subset \mathcal{J}_D(X,\lambda)$ connecting $J_0$ to $J_1$ and satisfying $J_t|_{[0,\infty)\times \partial X}=J_{\partial X}$ for all $t\in [0,1]$
, the parametric moduli space \[\mathcal{M}^{\{J_t\},s}_{\hat{X}, A}(\gamma)\ll \mathcal{T}_D^{k-1}p\gg\] is a  oriented $1$-dimensional cobordism and any non-trvial holomorphic building $\mathbb{H}$ that can appear in its compactification satisfies (H1) and (H2). By formal perturbation invariance, $\mathbb{H}$ must be as described in (C2). Theorem \ref{perturbation-invariance} implies that non-trivial buildings of this form  do not interfere with the count, i.e., we have
\[\# \mathcal{M}^{J_0,s}_{\hat{X}, A}(\gamma)\ll \mathcal{T}_D^{k-1}p\gg=\# \mathcal{M}^{J_1,s}_{\hat{X}, A}(\gamma)\ll \mathcal{T}_D^{k-1}p\gg.\]
\end{remark}
\subsection{The case of convex toric domains}\label{roudningfullyconvex}
An essential class of Liouville domains is that of convex toric domains in the standard symplectic vector space $(\mathbb{C}^n,\omega_{\mathrm{std}})$.  These are defined as follows. Consider the standard moment map $\mu:\mathbb{C}^{n}\to \mathbb{R}^{n}_{\geq 0}$ given by
\[\mu(z_1,z_2,\dots, z_n):= \pi(|z_1|^2,\dots, |z_n|^2).\]
A toric domain in $\mathbb{C}^{n}$ is a subset $X^{2n}_{\Omega}$ of $\mathbb{C}^{n}$ that can be written as $X^{2n}_{\Omega}=\mu^{-1}(\Omega)$, for some domain $\Omega \subset \mathbb{R}^{n}_{\geq 0}$. A convex toric domain is a toric domain $X^{2n}_{\Omega}$ for which the set
\[\widehat{\Omega}:=\{(x_1,x_2,\dots,x_n)\in \mathbb{R}^n: (|x_1|,|x_2|,\dots,|x_n|)\in \Omega\}\] 
is a convex subset of $\mathbb{R}^{n}$. We define the diagonal of a toric domain $X^{2n}_{\Omega}$ to be the number 
\[\operatorname{diagonal}(X^{2n}_{\Omega}):=\sup\{a>0:(a,a,\dots,a)\in \Omega\}.\]

For example, let $\Omega$ be the $n$-simplex in $\mathbb{R}^{n}_{\geq 0}$ with vertices $(0,\dots,0), (a_1,\dots,0),$ $(0,a_2, \dots,0)$ and $(0, \dots,a_n)$ where $0<a_1\leq a_2\leq a_3\leq \dots\leq a_n<\infty$. Then, $X^{2n}_{\Omega}$ is the standard $2n$-dimensional ellipsoid  defined by
\[E^{2n}(a_1,a_2,a_3,\dots,a_n):=\bigg\{(z_1,\dots, z_{n})\in \mathbb{C}^{n}: \sum_{i=1}^{n}\frac{  \pi|z_i|^2}{a_i}\leq 1\bigg \}.\] 

A compact convex toric domain  $X^{2n}_{\Omega}$ inherits the standard symplectic structure from  $(\mathbb{C}^n,\omega_{\mathrm{std}}=d\lambda_{\mathrm{std}})$. When $\partial X^{2n}_{\Omega}$ is smooth, the pair $(X^{2n}_{\Omega},d\lambda_{\mathrm{std}})$ is an exact symplectic cobordism with  positive contact type boundary $(\partial X^{2n}_{\Omega},\lambda_{\mathrm{std}}|_{\partial X^{2n}_{\Omega}})$ and no negative boundary. In other words, $(X^{2n}_{\Omega},\lambda_{\mathrm{std}})$ is a Liouville domain. Given a closed Reeb orbit $\gamma$ on $\partial X^{2n}_{\Omega}$, we define its action by 
\[\operatorname{\mathcal{A}}_{\Omega}(\gamma):=\int_{\gamma}\lambda_{\mathrm{std}}\in (0,\infty).\]
In this section, we assume that the convex toric domain $X_{\Omega}^{2n}$ is compact, unless stated otherwise.

In the spirit of Theorem \ref{perturbation-invariance}, for a given positive integer $k$ and a suitable closed Reeb orbit $\gamma$ on the boundary $\partial X^{2n}_{\Omega}$, we would like to have an invariant for $X^{2n}_{\Omega}$ that counts elements in  $\mathcal{M}^{J,s}_{X_{\Omega}^{2n}}(\gamma)\ll \mathcal{T}_D^{k-1}p\gg$. To make the moduli space rigid, $\gamma$ must have a high Conley--Zehnder index: by {\cite[Equation 2.2.1]{McDuff:2021aa}}, the expected dimension of the moduli space under consideration is given by
\[\operatorname{ind}\big(\mathcal{M}^{J,s}_{X_{\Omega}^{2n}}(\gamma)\ll \mathcal{T}^{k-1}p\gg\big)=n-3+\operatorname{CZ}^{\tau_{\mathrm{ext}}}(\gamma)+2-2n-2(k-1),\]
where $\tau_{\mathrm{ext}}$ is the homotopy class of symplectic trivializations of \[(\operatorname{Ker}(\lambda_{\mathrm{std}}|_{\partial X_{\Omega}^{2n}}),d \lambda_{\mathrm{std}}|_{\partial X_{\Omega}^{2n}} )\]
 along $\gamma$ that extend to a capping disk of $\gamma$. It follows that $\mathcal{M}^{J,s}_{X_{\Omega}^{2n}}(\gamma)\ll \mathcal{T}_D^{k-1}p\gg$ is rigid if and only if $\operatorname{CZ}^{\tau_{\mathrm{ext}}}(\gamma)=2k+n-1$. A priori the boundary $\partial X^{2n}_{\Omega}$ might not possess such an orbit $\gamma$. Using the ``rounding procedure'' explained in {\cite[Section 2.2]{MR3868228}}, one can produce closed Reeb orbits on $\partial X^{2n}_{\Omega}$ of the right action and the right Conley--Zehnder index by perturbing the boundary by a small amount. We explain this procedure below.
 
\textbf{The rounding procedure:} We perturb the toric domain to create useful Reeb dynamics. We do this in two steps. In the first step, we analyze the Reeb dynamics on the perturbed domain, and in the second step, we further perturb the domain to make the Reeb flow nondegenerate. Here, we are interested in four-dimensional compact convex toric domains $X^{4}_{\Omega}$. In this case, the domain $\Omega$ consists of points $(x_1,x_2)$ satisfying $0 \leq x_1\leq a$ and $0 \leq x_2\leq f(x_1)$ for some function  $f:[0,a]\to [0,b]$ which is non-increasing and concave, for some $a,b\in \mathbb{R}_{> 0}$. Following {\cite[Section 2.2]{MR3868228}}, we describe the rounding procedure for a four-dimensional compact convex toric domain, but it works for all dimensions.

\textbf{Step 1:} A priori, the boundary $\partial \Omega$ of $\Omega$ may not be smooth.  Replace $\Omega$ by a $C^0$-small perturbation, denoted by $\Omega^{\mathrm{FR}}$, such that $\Omega\subset \Omega^{\mathrm{FR}}$ and $\partial \Omega^{\mathrm{FR}}$ is  a smooth embedded curve in $\mathbb{R}_{\geq 0}^{2}$. More precisely, we perturb $\Omega$ to $\Omega^{\mathrm{FR}}$  such that $\partial \Omega^{\mathrm{FR}}$ is the graph of some smooth function $g:[0, a]\to [0,b]$ satisfying the following properties.
\begin{itemize}
	\item $g(0)=b$, and $g(a)=0$.
	\item $g$ is strictly decreasing and strictly concave down, i.e., $g''(x)<0$ for all $x\in [0,a]$.
	\item $-v\leq g'(0)<0$, and $ g'(1)<-1/v$ for some small $v>0$.
\end{itemize} 
The domain $X^{4}_{\Omega^{\mathrm{FR}}}:=\mu^{-1}(\Omega^{\mathrm{FR}})$ is a $C^0$-small perturbation of  $X^{4}_{\Omega}$ and is a convex toric domain. Next, we explain the Reeb dynamics on $(\partial X^{4}_{\Omega^{\mathrm{FR}}}, \lambda_{\mathrm{std}}).$
\begin{proposition}\label{Reebvectorfield0}
Let $G:\partial \Omega^{\mathrm{FR}}\to \mathbb{S}^1$ be the Gauss map. Then $G$ is a smooth embedding into $\mathbb{S}^1\cap \mathbb{R}^2_{>0}$, and for each $w\in \partial \Omega^{\mathrm{FR}}$ we have
\[\max_{v\in \Omega^{\mathrm{FR}}}\langle G(w),v\rangle=\langle G(w),w\rangle,\]
where $\langle \cdot,\cdot\rangle$ is the standard inner product on $\mathbb{R}^{2}$.
\end{proposition}
\begin{proof}
The Gauss map $G:\partial \Omega^{\mathrm{FR}}\to \mathbb{S}^1$ is given by 
\[G(w)=\frac{(-g'(x),1)}{\sqrt{1+g'(x)^2}}, \text{ where we write } w=(x,g(x)) \text{ for some $x\in [0,a]$}.\]
By the properties of $g$, clearly  $G$ is a smooth embedding into $\mathbb{S}^1\cap \mathbb{R}_{>0}$.

Note that the function $v\to \langle G(w),v\rangle$ has no maximum in $\Omega^{\mathrm{FR}}\setminus \partial \Omega^{\mathrm{FR}}$; for every $v\in \Omega^{\mathrm{FR}}\setminus \partial \Omega^{\mathrm{FR}}$, $(1+t)v \in \Omega^{\mathrm{FR}}\setminus \partial \Omega^{\mathrm{FR}}$ for small $t>0$ and $\langle G(w),(1+t)v\rangle> \langle G(w),v\rangle>0$. So we look for  maxima on $\partial \Omega^{\mathrm{FR}}$. 

Write $w_0=(x_0,g(x_0))$ for some $x_0\in [0,a]$. Define $l:[0,a]\to \mathbb{R}$ by 
\[l(x):=\langle G(w_0), (x,g(x))\rangle.\]
We have 
\[l'(x)=\langle G(w_0), (1,g'(x)\rangle =\frac{g'(x)-g'(x_0)}{\sqrt{1+g'(x_0)^2}}.\]
We have that $l'(x)\leq 0$ for $x\geq x_0$ and  $l'(x)\geq 0$ for $x\leq x_0$. Moreover, $l'(x)=0$ if and only if $x=x_0$. So, $l$ has a unique global maximum at $w_0=(x_0,g(x_0))$. Thus, 
\[\max\big\{\langle G(w),v\rangle: v\in  \Omega^{\mathrm{FR}}\big\}=\langle G(w),w\rangle,\]
for every $w\in \partial \Omega^{\mathrm{FR}}$.
\end{proof}
\begin{proposition}\label{Reebvectorfield}
Let $(\mu_1=\pi|z_1|^2,\theta_1)$ and $(\mu_2=\pi|z_2|^2,\theta_2)$ be  polar coordinates on the first and second factor in $\mathbb{C}^2$, respectively.
The Reeb vector field of $\lambda_{\mathrm{std}}$ on $\partial X^{4}_{\Omega^{\mathrm{FR}}}$ is given by 
\[R_{\lambda_{\mathrm{std}}}(z)=\frac{2\pi}{\langle G(w),w\rangle}  \sum_{w_i\neq 0}v_i\frac{\partial}{\partial \theta_i}, \]
for  $z\in \partial X^{4}_{\Omega^{\mathrm{FR}}}$, $w=\mu(z)=(w_1,w_2)$, and $G(w)=(v_1,v_2)$. 
\end{proposition}
\begin{proof}
Let $w\in \partial \Omega^{\mathrm{FR}}$ and $z \in \mu^{-1}(w)$. We want to determine the tangent space $T_z \partial X^{4}_{\Omega^{\mathrm{FR}}}$. Note that  $v\in T_z \partial X^{4}_{\Omega^{\mathrm{FR}}}$ if and only if $\langle d\mu(v),G(w)\rangle=0$, where $\mu=(\mu_1,\mu_2)$ is the moment map. 

For $z\in (\mathbb{C}\setminus \{0\})\times (\mathbb{C}\setminus \{0\})$, any $v\in T_z\mathbb{C}^2$ can be written as 
\[v= \bigg(a_1\frac{\partial}{\partial \mu_1}+a_2\frac{\partial}{\partial \mu_2}\bigg)+\bigg(b_1\frac{\partial}{\partial \theta_{1}}+b_2\frac{\partial}{\partial \theta_{2}}\bigg).\]
Write $G(w)=(v_1,v_2)$ and $w=(w_1,w_2)$. If none of the $w_i$ is zero, then  $\langle d\mu(v),G(w)\rangle=0$ implies that 
\[T_z \partial X^{4}_{\Omega^{\mathrm{FR}}}=\bigg\{\bigg(a_1\frac{\partial}{\partial u_1}+a_2\frac{\partial}{\partial u_2}\bigg)+\bigg(b_1\frac{\partial}{\partial \theta_{1}}+b_2\frac{\partial}{\partial \theta_{2}}\bigg):a_1v_1+a_2v_2=0, a_1,a_2,b_1,b_2\in \mathbb{R} \bigg\},\]
or equivalently  
\[T_z \partial X^{4}_{\Omega^{\mathrm{FR}}}=\bigg\{a\bigg(\frac{\partial}{\partial u_1}-\frac{v_1}{v_2}\frac{\partial}{\partial u_2}\bigg)+\bigg(b_1\frac{\partial}{\partial \theta_{1}}+b_2\frac{\partial}{\partial \theta_{2}}\bigg): a,b_1,b_2\in \mathbb{R} \bigg\}.\]
In general, for any $w=(w_1,w_2)\in \partial \Omega^{\mathrm{FR}}$ we have 
 \[T_z \partial X^{4}_{\Omega^{\mathrm{FR}}}=\big(\bigoplus_{w_i=0} \mathbb{C}\big)\bigoplus\bigg\{\sum_{w_i\neq 0}\bigg(a_i\frac{\partial}{\partial u_1}+b_i\frac{\partial}{\partial \theta_{1}}\bigg): \sum_{w_i\neq 0}a_iv_i=0 \bigg\}.\]
 
Next we determine the Reeb vector field $R_{\lambda_{\mathrm{std}}}$ on $\partial X^{4}_{\Omega^{\mathrm{FR}}}$. In polar coordinates, we have
 \[\lambda_{\mathrm{std}}=\frac{1}{2\pi}(\mu_1d\theta_1+\mu_2d\theta_2).\]
If a vector field $V\in T \partial X^{4}_{\Omega^{\mathrm{FR}}}$ satisfies  
 \[\lambda_{\mathrm{std}}(V)=1, \text{ and } d\lambda_{\mathrm{std}}(V,\cdot)|_{T \partial X^{4}_{\Omega^{\mathrm{FR}}}}=0,\]
then for every $z\in \mu^{-1}(w)$ with $w_1,w_2\neq 0$ we have
 \[V(z)=2\pi \bigg(b_1\frac{\partial}{\partial \theta_{1}}+b_2\frac{\partial}{\partial \theta_{2}}\bigg), \text{ where } b_1 w_1+b_2w_2=1, \text{ and }  b_1v_2-v_1b_2=0.\]
Thus, 
\[R_{\lambda_{\mathrm{std}}}(z)=V(z)=\frac{2\pi}{\langle G(w),w\rangle}  \bigg(v_1\frac{\partial}{\partial \theta_1}+v_2\frac{\partial}{\partial \theta_2}\bigg). \]
In general, we have
\[R_{\lambda_{\mathrm{std}}}(z)=\frac{2\pi}{\langle G(w),w\rangle}  \sum_{w_i\neq 0}v_i\frac{\partial}{\partial \theta_i}.\qedhere \]
\end{proof}
\begin{proposition}\label{Reebvectorfield3}
Consider $w\in \partial \Omega^{\mathrm{FR}}$ such that $w_1,w_2\neq 0$. If $G(w)=c(l,m)$ for some $(l,m)\in \mathbb{Z}^2_{\geq 1}$ and $c\in \mathbb{R}_{\geq 0}$, then the Lagrangian product torus $\mu^{-1}(w)$  is foliated by a Morse--Bott $S^1$-family of closed Reeb orbits. The closed Reeb orbits in $\mu^{-1}(w)$ are the $\operatorname{gcd}(l,m)$-fold covers of the underlying simple Reeb orbits. Moreover, we have
\[\operatorname{\mathcal{A}}_{\Omega^{\mathrm{FR}}}(\gamma)=\int_{\gamma}\lambda_{\mathrm{std}}|_{\partial X^4_{\Omega^{\mathrm{FR}}}}=\max_{v\in \Omega^{\mathrm{FR}}}\langle v,(l,m)\rangle\]
for every orbit $\gamma$ in the family $\mu^{-1}(w)$. If $G(w)\neq c(l,m)$ for some $(l,m)\in \mathbb{Z}^2_{\geq 1}$ and $c\in \mathbb{R}_{\geq 0}$, then $\mu^{-1}(w)$  contains no closed Reeb orbit. 
\end{proposition}
\begin{proof}
Consider $w\in \partial \Omega^{\mathrm{FR}}$. When $w_1,w_2\neq 0$, $\mu^{-1}(w)$ is a Lagrangian product torus that lies entirely on $\partial X^{4}_{\Omega^{\mathrm{FR}}}$. The Reeb vector field $R_{\lambda_{\mathrm{std}}}$ is tangent to $\mu^{-1}(w)$; this follows from the fact that the $2$-form $d\lambda_{\mathrm{std}}|_{\partial X^4_{\Omega^{\mathrm{FR}}}}$ has maximal rank and  $\mu^{-1}(w)$ is Lagrangian, or directly from Proposition \ref{Reebvectorfield}. If for some pair $(l,m)\in \mathbb{Z}^2_{\geq 1}$ the unit normal vector at $w_{l,m}:=w \in \partial X^{4}_{\Omega^{\mathrm{FR}}} $ is parallel to  $(l,m)$, then by Proposition \ref{Reebvectorfield}  there is a Morse--Bott $S^1$-family of closed Reeb orbits sweeping out the Lagrangian torus $\mu^{-1}(w_{l,m})\subset \partial X^{4}_{\Omega^{\mathrm{FR}}}$. Moreover, the closed Reeb orbits in $\mu^{-1}(w_{l,m})$ are the $\operatorname{gcd}(l,m)$-fold covers of the underlying simple Reeb orbits.

The action $\operatorname{\mathcal{A}}_{\Omega^{\mathrm{FR}}}(\cdot)$ is constant over the $S^1$-family of closed Reeb orbits $\mu^{-1}(w_{l,m})$; given $\gamma_0, \gamma_1 \in \mu^{-1}(w_{l,m})$ and a homotopy $u:[0,1]\times S^1\to \mu^{-1}(w_{l,m})$ interpolating between $\gamma_0$ and $\gamma_1$, since $\mu^{-1}(w_{l,m})$ is Lagrangian, by Stokes' theorem  we have 
\[\operatorname{\mathcal{A}}_{\Omega^{\mathrm{FR}}}(\gamma_1)- \operatorname{\mathcal{A}}_{\Omega^{\mathrm{FR}}}(\gamma_0)=\int_{[0,1]\times S^1}u^{*}d\lambda_{\mathrm{std}}|_{\partial X^4_{\Omega^{\mathrm{FR}}}}=0.\]
By Proposition \ref{Reebvectorfield} and Proposition \ref{Reebvectorfield0}, we have
\[\operatorname{\mathcal{A}}_{\Omega^{\mathrm{FR}}}(\gamma)=\int_{\gamma}\lambda_{\mathrm{std}}|_{\partial X^4_{\Omega^{\mathrm{FR}}}}=\max_{v\in \Omega^{\mathrm{FR}}}\langle v,(l,m)\rangle\]
for every orbit $\gamma \in \mu^{-1}(w_{l,m})$. \qedhere
\end{proof}
 
\textbf{Step 2:} Given $K>0$, following {\cite[Section 2.3]{Bourgeois2002}}, we perturb  $(X^{4}_{\Omega^{\mathrm{FR}}}, \lambda_{\mathrm{std}})$ slightly to make all closed Reeb orbits of action up to $K$ nondegenerate. Let $S\subset \partial X^{4}_{\Omega^{\mathrm{FR}}}$ be the union of the images of $S^1$-families of simple closed Reeb orbits of action less or equal to $K$. More precisely, by Proposition \ref{Reebvectorfield3}, 
\[S=\bigcup_{w_{l,m}\in \Omega^{\mathrm{FR}}} \mu^{-1}(w_{l,m}),\]
where the union is taken over all points $w_{l,m} \in \Omega^{\mathrm{FR}}$ for which there exist coprime positive integers $(l,m)$ such that $G(w_{l,m})=c (l,m)$ for some $c>0$, and $\max_{v\in \Omega^{\mathrm{FR}}}\langle v,(l,m)\rangle\leq K$. The Reeb flow induces an $S^1$-action on $S$, and the quotient $S/S^1=\cup \mu^{-1}(w_{l,m})/S^1 $ is a disjoint union of finitely many circles. Choose a smooth function $\bar{f}:\partial X^{4}_{\Omega^{\mathrm{FR}}}\to \mathbb{R}$ supported in a small neighborhood of $S$ that descends to a Morse function $f: S/S^1\to \mathbb{R}$ which is perfect on each component $\mu^{-1}(w_{l,m})/S^1$ of $S/S^1$. By {\cite[Lemma 2.3]{Bourgeois2002}}, using the smooth function $\bar{f}$ one can perturb the domain $X^{4}_{\Omega^{\mathrm{FR}}}$ such that after the perturbation each $S^1$-family  of closed Reeb orbits $\mu^{-1}(w_{l,m})$ in $S/S^1$ is resolved into two nondegenerate closed Reeb orbits corresponding to the critical points of $f|_{\mu^{-1}(w_{l,m})}$.  By {\cite[Lemma 2.4]{Bourgeois2002}}, one of these two orbits is elliptic and the other is hyperbolic. Let $e_{l,m}$ denote the elliptic orbit and $h_{l,m}$ denote the hyperbolic one coming from $\mu^{-1}(w_{l,m})$. We assume we have perturbed $ X^{4}_{\Omega^{\mathrm{FR}}}$ as above and still denote by $ X^{4}_{\Omega^{\mathrm{FR}}}$ the perturbed domain.

All the above orbits on $ X^{4}_{\Omega^{\mathrm{FR}}}$ are contractible.  For each closed Reeb orbit $\gamma$, we choose the symplectic trivialization (unique up to homotopy), denoted by $\tau_{\mathrm{ext}}$, of the contact distribution \[(\operatorname{Ker}(\lambda_{\mathrm{std}}|_{\partial X_{\Omega^{\mathrm{FR}}}^4}),d \lambda_{\mathrm{std}}|_{\partial X_{\Omega^{\mathrm{FR}}}^4} )\] along $\gamma$ that extends to a capping disk of $\gamma$ in $\partial X^{4}_{\Omega^{\mathrm{FR}}}$. By {\cite[Sec. 2.2, Ineq. (2-18)]{MR3868228}}, for the Reeb orbits $e_{l,m}$ and $h_{l,m}$ described above we have  
\begin{equation}\label{CZI}
	\operatorname{CZ}^{\tau_{\mathrm{ext}}}(e_{l,m})=2(l+m)+1,
\end{equation}
\[\operatorname{CZ}^{\tau_{\mathrm{ext}}}(h_{l,m})=2(l+m).\]

There are two different contact actions associated with the Reeb orbits $e_{l,m}$ and $h_{l,m}$.
\begin{itemize}
	\item The actions before perturbing the domain $ X^{4}_{\Omega^{\mathrm{FR}}}$ to a non-degenerate domain, which by Proposition \ref{Reebvectorfield3} are given by 
	\begin{equation}
		\operatorname{\mathcal{A}}_{\Omega^{\mathrm{FR}}}(e_{l,m})=\operatorname{\mathcal{A}}_{\Omega^{\mathrm{FR}}}(h_{l,m})=\max_{v\in \Omega^{\mathrm{FR}}}\langle v,(l,m)\rangle.
	\end{equation}
\item The actions, denoted by $\operatorname{\mathcal{\tilde{A}}}(e_{l,m})$ and $\operatorname{\mathcal{\tilde{A}}}(h_{l,m})$, after perturbing $X^{4}_{\Omega^{\mathrm{FR}}}$ (``perturbed actions''). By construction, the actions $\operatorname{\mathcal{\tilde{A}}}(e_{l,m})$ and $\operatorname{\mathcal{\tilde{A}}}(h_{l,m})$ are not equal---in fact, $\operatorname{\mathcal{\tilde{A}}}(e_{l,m})$ is greater than $\operatorname{\mathcal{\tilde{A}}}(h_{l,m})$; see {\cite[Lemma 2.3]{Bourgeois2002}}. 
\end{itemize}

 By {\cite[Section 2.3]{Bourgeois2002}}, the perturbation of $X^{4}_{\Omega^{\mathrm{FR}}}$ in Step $2$ can be chosen arbitrary small in $C^1$-norm, so the perturbed action $\operatorname{\mathcal{\tilde{A}}}(\cdot)$ is a small perturbation of $\operatorname{\mathcal{A}}_{\Omega^{\mathrm{FR}}}(\cdot)$. From now on, we assume that the perturbation of $X^{4}_{\Omega^{\mathrm{FR}}}$ is very small and so we will not differentiate among  $\operatorname{\mathcal{\tilde{A}}}(\cdot)$ and $\operatorname{\mathcal{A}}_{\Omega^{\mathrm{FR}}}(\cdot)$. This completes our description of the rounding procedure. 
  
 By {\cite[Lemma 2.4]{Choi_2014}}, the assignment $\Omega\to \max_{v\in \Omega}\langle v,(l,m)\rangle$ is continuous with respect to the Hausdorff metric on compact sets $\Omega$ of $\mathbb{R}^2$. So we have that  
\[\lim_{\Omega^{\mathrm{FR}}\to \Omega}\max_{v \in \Omega^{\mathrm{FR}}}\langle v ,(l,m)\rangle = \max_{v \in \Omega}\langle v,(l,m)\rangle.\]
In particular, if we choose the rounded domain $\Omega^{\mathrm{RF}}$ very close to $\Omega$ in $C^0$-norm, then
\[\max_{v\in \Omega^{\mathrm{FR}}}\langle v,(l,m)\rangle \approx \max_{v\in \Omega}\langle v,(l,m)\rangle.\]

Next, we discuss some consequences of the rounding procedure for the computation of the  Gutt--Hutchings capacities of toric domains and the count of $J$-holomorphic planes with local tangency constraints. First, we briefly recall the definition of the Gutt--Hutchings capacities.

 Let $(X,\lambda)$ be a connected Liouville domain. We consider its symplectic cohomology $\operatorname{SH}(X,\mathbb{Q})$  with coefficients in the field of rational numbers $\mathbb{Q}$. Roughly speaking, the underlying cochain complex $(\operatorname{CF} (X,\mathbb{Q}),\partial)$ is a $\mathbb{Q}$-vector space with one generator for each critical point of a chosen Morse function on $X$, and two generators $\hat{\gamma}$ and $\check{\gamma}$ for each periodic Reeb orbit $\gamma$ on $\partial X$. The differential $\partial$ counts Floer cylinders interpolating between the input and output orbits.

The cochain complex $\operatorname{CF} (X,\mathbb{Q})$ carries a natural increasing action filtration
\[
\mathcal{F}_{\leq a}\operatorname{CF} (X,\mathbb{Q}) \subset \mathcal{F}_{\leq a'}\operatorname{CF} (X,\mathbb{Q}), \qquad 0 \leq a < a'.
\]
The subcomplex $\mathcal{F}_{\leq 0}\operatorname{CF} (X,\mathbb{Q}) \subset \operatorname{CF} (X,\mathbb{Q})$, generated by action-zero elements, corresponds to the Morse cochain complex of $X$ over $\mathbb{Q}$. In particular, there is a canonical action-zero element
\[
e \in \mathcal{F}_{\leq 0}\operatorname{CF} (X,\mathbb{Q}),
\]
given by the minimum of the Morse function, which represents the unit on the chain level.

The complex $\operatorname{CF} (X,\mathbb{Q})$ has the structure of an $S^1$-complex; that is, there is a sequence of operations
\[
\delta_i: \operatorname{CF} (X,\mathbb{Q}) \to \operatorname{CF} (X,\mathbb{Q})\]
of degree $1-2i$, with $\delta_0=\partial$ (the Floer differential), satisfying the relation
\[
\sum_{i+j=k}\delta_i\circ\delta_j=0
\]
for all $k\in \mathbb{Z}_{\geq 0}$. 

The $S^1$-equivariant symplectic cohomology of $X$ with coefficients in $\mathbb{Q}$, denoted by $\operatorname{SH_{S^1}}(X,\mathbb{Q})$, is defined to be the cohomology of the cochain complex
\[
\bigl(\operatorname{CF_{S^1}}(X,\mathbb{Q}),\partial_{S^1}\bigr)
:=
\left(
\operatorname{CF}(X,\mathbb{Q})\otimes_{\mathbb{Q}}\mathbb{Q}(\!(u)\!)/u\mathbb{Q}[\![u]\!],
\sum_{i=0}^{\infty}\delta_i u^i
\right),
\]
where $u$
is a formal variable of degree $2$ and $\mathbb{Q}(\!(u)\!)$ denotes the field of formal Laurent series in $u$ with rational coefficients. 

For $k\in \mathbb{Z}_{\geq 1}$, the $k$-th Gutt--Hutchings capacity of $X$, introduced in \cite{MR3868228}, is defined by

\begin{equation}\label{Gutt--Hutchingscapacities}
c_k^{\mathrm{GH}}(X)
:=
\inf\left\{
a \,\middle|\,
\partial_{S^1}(x)=u^{-k+1}e
\text{ for some }
x\in \mathcal{F}_{\leq a}\operatorname{CF_{S^1}}(X,\mathbb{Q})
\right\}.
\end{equation}

Putting things together, we have the following. 
\begin{theorem}[{\cite[Theorem 1.6, Section 2.2]{MR3868228}}]\label{GH}
Let $X_{\Omega}^4$ be any four-dimensional compact convex toric domain. Given any integer $k\in \mathbb{Z}_{\geq 1}$, the rounding $X_{\Omega^{\mathrm{FR}}}^4$ can be arranged so that
\begin{equation}\label{final}
c^{\mathrm{GH}}_k(X_{\Omega^{\mathrm{FR}}}^4)=\min_{
		(l',m')\in \mathbb{Z}^2_{\geq 0}\atop
		l'+m'=k }  \max_{v\in \Omega^{\mathrm{FR}}}\langle v,(l',m')\rangle=\min_{\gamma} \operatorname{\mathcal{A}}_{\Omega^{\mathrm{FR}}}(\gamma),
\end{equation}
where the minimum on the right-hand side is taken over elliptic Reeb orbits $\gamma$ on $\partial X_{\Omega^{\mathrm{FR}}}^4$ with $\operatorname{CZ}^{\tau_{\mathrm{ext}}}(\gamma)=2k+1$, and $c^{\mathrm{GH}}_k(X_{\Omega^{\mathrm{FR}}}^4)$ denotes the $k$th Gutt-Hutchings capacity of $X_{\Omega^{\mathrm{FR}}}^4$. Moreover, the rounding $\Omega^{\mathrm{FR}}$ can be chosen arbitrary close to $\Omega$ in $C^0$-norm and 
\[c^{\mathrm{GH}}_k(X_{\Omega}^4)=\lim_{\Omega^{\mathrm{FR}}\to \Omega}c^{\mathrm{GH}}_k(X_{\Omega^{\mathrm{FR}}}^4)= \min_{
		(l',m')\in \mathbb{Z}^2_{\geq 0}\atop
		l'+m'=k }  \max_{v\in \Omega}\langle v,(l',m')\rangle.\]
\end{theorem}
\begin{remark}\label{simpleorbit}
Let $X_{\Omega}^4$ be any four-dimensional compact convex toric domain. Given an integer $k\in \mathbb{Z}_{\geq 1}$, by Theorem \ref{GH} and the discussion preceding it, there exist a rounding $\Omega^{\mathrm{FR}}$ that can be chosen arbitrary close to $\Omega$ in $C^0$-norm, some $(l,m)\in \mathbb{Z}^2_{\geq 0}$ with $l+m=k$, and an elliptic Reeb orbit $e_{l,m}$ on $\partial X_{\Omega^{\mathrm{FR}}}^4$ such that
\[\operatorname{\mathcal{A}}_{\Omega^{\mathrm{FR}}}(e_{l,m})=\max_{v\in \Omega^{\mathrm{FR}}}\langle v,(l,m)\rangle =c^{\mathrm{GH}}_k(X_{\Omega^{\mathrm{FR}}}^4).\]
Let $\bar{l}=l/\operatorname{gcd}(l,m)$, $\bar{m}=l/\operatorname{gcd}(l,m)$, and $\bar{k}=k/\operatorname{gcd}(l,m)$. By the above discussion, $e_{l,m}$ is the $\operatorname{gcd}(l,m)$-fold cover of the simple elliptic orbit $e_{\bar{k}}:=e_{\bar{l},\bar{m}}$ and 
\[\operatorname{\mathcal{A}}_{\Omega^{\mathrm{FR}}}(e_{\bar{k}})\leq \operatorname{\mathcal{A}}_{\Omega^{\mathrm{FR}}}(e_{l,m})=c^{\mathrm{GH}}_k(X_{\Omega^{\mathrm{FR}}}^4).\]
\end{remark}
\begin{theorem}[{\cite[Section 5, Proposition 5.6.1]{McDuff:2021aa}}]\label{puctured-dsik}
Let $X_{\Omega}^4$ be any four-dimensional compact convex toric domain. Choose a point $p\in \operatorname{int}(X_{\Omega}^4)$ and local divisor $D$ containing $p$. Given an integer $k\in \mathbb{Z}_{\geq 1}$,  let $\{X_{\Omega^{\mathrm{FR}}}^4$, $\bar{l}$, $\bar{m}$, $\bar{k}$, $e_{\bar{k}}\}$ be the corresponding data described in Remark \ref{simpleorbit}, with $e_{\bar{k}}$ a simple elliptic Reeb orbit. We have the following. 
\begin{itemize}
		\item The moduli space $\mathcal{M}^{J}_{X_{\Omega^{\mathrm{FR}}}^4}(e_{\bar{k}})\ll \mathcal{T}_D^{\bar{k}-1}p\gg$ has Fredholm index equal to zero and is formally perturbation invariant (cf. Definition \ref{perturbationinvariance}). In particular, by Theorem \ref{perturbation-invariance}, 
		the signed count 
		\[\#\mathcal{M}^{J}_{X_{\Omega^{\mathrm{FR}}}^4}(e_{\bar{k}})\ll \mathcal{T}_D^{\bar{k}-1}p\gg\]
is well-defined and is independent of the choice of generic almost complex structure $J \in  \mathcal{J}_D(X_{\Omega^{\mathrm{FR}}}^4,\lambda_{\mathrm{std}})$.
		\item For generic $J \in  \mathcal{J}_D(X_{\Omega^{\mathrm{FR}}}^4,\lambda_{\mathrm{std}})$, the signed count 
		\[\#\mathcal{M}^{J}_{X_{\Omega^{\mathrm{FR}}}^4}(e_{\bar{k}})\ll \mathcal{T}_D^{\bar{k}-1}p\gg\]
		is positive and agrees with the unsigned count.
		\item The $\tilde{\omega}_{\mathrm{std}}$-energy of $u \in \mathcal{M}^{J}_{X_{\Omega^{\mathrm{FR}}}^4}(e_{\bar{k}})\ll \mathcal{T}_D^{\bar{k}-1}p\gg$ satisfies 
		\[E_{\tilde{\omega}_{\mathrm{std}}}(u)=\operatorname{\mathcal{A}}_{\Omega^{\mathrm{FR}}}(e_{\bar{k}})\leq   c^{\mathrm{GH}}_k(X_{\Omega^{\mathrm{FR}}}^4)\approx c^{\mathrm{GH}}_k(X_{\Omega}^4), \]
		where $\tilde{\omega}_{\mathrm{std}}$ is the piecewise smooth $2$-form defined in (\ref{deg-form}). 
	\end{itemize}
\end{theorem}
\begin{remark}
	In the third bullet point in the above theorem, the first equality is due to Stokes' theorem. The ``$\leq$'' follows from Remark \ref{simpleorbit}, and ``$\approx$''  is due to Theorem \ref{GH}.
\end{remark}
\begin{remark}\label{somewhere-injective}
By Remark \ref{simpleorbit}, the closed Reeb orbit $e_{\bar{k}}$ in Theorem \ref{puctured-dsik} is simple. So, the moduli space $\mathcal{M}^{J}_{X_{\Omega^{\mathrm{FR}}}^4}(e_{\bar{k}})\ll \mathcal{T}_D^{\bar{k}-1}p\gg$ consists of only somewhere injective $J$-holomorphic planes. 
\end{remark} 
\begin{remark}\label{elliptic}
By Remark \ref{simpleorbit}, the closed Reeb orbit $e_{\bar{k}}$ in Theorem \ref{puctured-dsik} is elliptic.
\end{remark} 

\section{Punctured spheres with local tangency constraints in cotangent bundles}\label{intersect}
This section reviews essential facts useful for analyzing punctured spheres in cotangent bundles. Following \cite{Tonkonog:2018aa}, we define linear operations on the positive $S^1$-equivariant symplectic cohomology of the cotangent bundles of closed manifolds $L$ admitting a metric of non-positive sectional curvature. These operations are defined in terms of counts of punctured spheres carrying a local tangency constraint at a generic point. We will need these operations in order to count certain holomorphic buildings in terms of the  ``Borman--Sheridan class'' in Section \ref{Borman-}.
\begin{theorem}[{\cite[Lemma 2.2]{Cieliebak2018}}]\label{perturb-metric}
Let $L$ be an $n$-dimensional closed manifold that admits a metric $\tilde{g}$ of non-positive sectional curvature. For every $c>0$, there exists a $C^2$-small perturbation, denoted by  $g$, of $\tilde{g}$  in the space of Riemannian metrics on $L$ such that, with respect to $g$, every closed geodesic $\gamma$ of length less than or equal to $c$ is non-contractible, non-degenerate (as a critical point of the energy functional) and satisfies 
	\[0\leq \operatorname{\mu}(\gamma)\leq n-1,\]
	where $\operatorname{\mu}(\gamma)$ denotes the Morse index of $\gamma$. In a metric of negative curvature, every closed geodesic is non-contractible and non-degenerate of Morse index $0$.
\end{theorem}
\begin{remark}
We will apply Theorem \ref{perturb-metric} in Section \ref{proof for ball} when  $L$ is the standard $n$-torus with flat metric.
\end{remark}
 The following theorem will be useful in our arguments in the upcoming sections. 
\begin{theorem}[{\cite[Corollary 3.3]{Cieliebak2018}}] \label{count-positive-punctures}
Let $L$ be an $n$-dimensional manifold that admits a metric of non-positive sectional curvature. For $c>0$, equip $L$ with a metric such that all closed geodesics of length less than or equal to  $c$ are non-contractible and non-degenerate and have Morse index at most $n-1$. Pick a point $p\in T^*L$, a germ of a symplectic hypersurface $D$ that contains the point $p$, and a positive integer $k$.
For a generic SFT-admissible almost complex structure $J$ on the cotangent bundle $(T^*L, d\lambda_{\mathrm{can}})$ that is integrable near $p$ and preserves $D$, every (not necessarily simple) non-constant punctured $J$-holomorphic sphere in $T^*L$ which is asymptotic at the punctures to geodesics of length at most $c$ and satisfies the tangency constraint  $\ll \mathcal{T}_D^{k-1}p\gg$ has at least $k+1$ positive punctures.
\end{theorem}
Given an oriented closed Riemannian manifold $(L,g)$, the metric $g$ induces a bundle isomorphism $\hat{g}:TL\to T^*L$ given by 
\[\hat{g}(p,v):=(p,g_p(v,\cdot)).\]
For $T>0$, an arc-length parametrized closed curve $c:\mathbb{R}/T\mathbb{Z}\to L$ is a geodesic (of length $T$) if and only if its lift to the unit cosphere bundle $\gamma_c:=\hat{g}\circ c: \mathbb{R}/T\mathbb{Z}\to (S^*L,\lambda|_{TS^*L}) $ defined by 
\[\hat{g}\circ c(t):=(c(t),g_{c(t)}(c'(t),\cdot))\]
is a closed Reeb orbit of period $T$ on $(S^*L,\lambda|_{TS^*L})$; see {\cite[Section 2.3]{Frauenfelder_2018}} for details.

Let $i: L\hookrightarrow T^*L$ denote the inclusion of $L$ as the zero-section. A closed curve  $c:\mathbb{R}/T\mathbb{Z}\to L$ and a symplectic trivialisation $\kappa$ of $(i\circ c)^*TT^*L$  determine a loop of Lagrangian subspaces of $(\mathbb{R}^{2n}, \omega_{\mathrm{std}})$ given by
\[L_t:=\kappa_t\circ Di(T_{c(t)}L).\]
We denote the Maslov index of $L_t$ by $\mu^\kappa(c)$. The following theorem was originally proved in \cite{Viterbo1990}, and restatements can be found in \cite{Cieliebak2018} and \cite{Pereira:2022aa}.
\begin{theorem}[{\cite[Section 3]{Viterbo1990}}, {\cite[Lemma 2.1]{Cieliebak2018}}, {\cite[Theorem 3.28]{Pereira:2022aa}}]\label{special-trivalization}
	
	Let $\gamma_c: \mathbb{R}/T\mathbb{Z}\to S^*L $ be the closed Reeb orbit (possibly Morse--Bott) corresponding to the closed geodesic  $c: \mathbb{R}/T\mathbb{Z}\to L$. The following holds.
\begin{enumerate}
\item Every symplectic trivialization $\tau$ of $\gamma_c^*\xi$, where $\xi=(\operatorname{Ker}(\lambda_{\mathrm{can}}|_{TS^*L}), d\lambda_{\mathrm{can}}|_{TS^*L})$,  induces a symplectic trivialization $\kappa(\tau)$ of $(i\circ c)^*TT^*Q$ such that 
		\[\operatorname{CZ}^\tau(\gamma_c)+\mu^{\kappa(\tau)}(c)=\mu(c),\]
		where $\mu(c)$ is the Morse index of $c$. 
\item 	In the above, one can adjust $\tau$ so that 
		\[\mu^{\kappa(\tau)}(c)=0.\]
\end{enumerate}
\end{theorem}
\subsection{Positive $S^1$-equivariant  symplectic cohomology of cotangent bundles}
Let $(X,\lambda)$ be a non-degenerate Liouville domain. In this document, we are interested in the case where $X$ is the unit codisk bundle $D^*T^n$ of the torus $T^n$ with a Riemannian metric $g$ coming from a small perturbation of the flat metric as described in Theorem \ref{perturb-metric}. An interesting invariant assigned to $(X,\lambda)$ is the positive $S^1$-equivariant  symplectic cohomology, denoted by $\operatorname{SH^*_{S^1,+}}(X)$. There are two definitions of this invariant in the current literature. The first one is using the Hamiltonian Floer theory framework; see {\cite{MR3868228, Pereira:2022ab, MR3671507}} for a reminder in this framework. The second one uses the symplectic field theory (SFT) framework and is called the  \textit{linearized contact homology}  \cite{MR3671507}. The definition using the SFT-framework is more straightforward, modulo the virtual perturbation techniques needed to achieve transversality.  We briefly recall the  SFT-definition. Let 
\[\operatorname{CF^*_{S^1,+}}(X)\]
be the $\mathbb{Q}$-vector space generated by the good\footnote{Let $\gamma$ be a closed Reeb orbit and $\bar{\gamma}$ be the underlying simple closed Reeb orbit. The closed Reeb orbit $\gamma$ is good if $\operatorname{CZ}^\tau(\gamma)$ and $\operatorname{CZ}^\tau(\bar{\gamma})$ have the same parity for some trivialization $\tau$. Recall that the parity of $\operatorname{CZ}^\tau(\cdot)$ does not depend on $\tau$. } Reeb orbits on $(\partial X, \lambda)$. Let $J$ be an SFT-admissible almost complex structure on $\widehat{X}$ such that $J|_{[0,\infty)\times \partial X}$ is the restriction of some SFT-admissible almost complex structure $J$ on the symplectization $(\mathbb{R}\times \partial X, d(e^r\lambda))$. The differential $\partial:\operatorname{CF^*_{S^1,+}}(X) \to \operatorname{CF^*_{S^1,+}}(X)$ is defined by 
\begin{equation}\label{Floer-dif}
\partial \gamma:=\sum_{\eta} \langle \partial \gamma, \eta \rangle \eta , 
\end{equation}
where $\langle \partial \gamma, \eta \rangle$ is the count of index $1$ punctured $J$-holomorphic spheres (without asymptotic markers) in the symplectization $\mathbb{R}\times \partial X$ with one positive puncture asymptotic to $\gamma$, a negative puncture asymptotic to $\eta$, and some additional negative punctures which are augmented by asymptotically cylindrical $J$-holomorphic planes in $X$.

There are two important things to point out here.
\begin{itemize}
\item[-] To achieve $\partial \circ  \partial=0$ and make the chain complex $\operatorname{CF^*_{S^1,+}}(X)$ independent (up to chain homotopy) on the choice of the almost complex structure $J$ one requires a suitable virtual perturbation scheme to define the curve count involved.
\item[-] Suppose there are no closed Reeb orbits on $\partial X$ that are contractible in $X$; for example, this is the case when $(X,\lambda)$ is the unit codisk bundle of a closed Riemannian manifold that admits a metric of non-positive sectional curvature as in this case there are no contractible closed geodesics. The differential $\partial$ counts pure holomorphic cylinders, which are unbranched by the Riemann--Hurwitz formula, interpolating between the input and output closed Reeb orbits. It follows from the work of Wendl {\cite[Theorem B]{Wendl2023}} that these cylinders are generically transversally cut out; {\cite[Theorem B]{Wendl2023}} covers the case of unbranched covers of closed $J$-holomorphic spheres in closed symplectic manifolds, but this theorem can be carried out in the SFT-framework as well. So, one does not require any virtual perturbation scheme to define $\operatorname{CF^*_{S^1,+}}(X)$; in particular, this is the case for the unit codisk bundle $D^*T^n$ of the torus $T^n$ with respect to a metric of non-positive sectional curvature \cite{MR2475400, MR3671507} in which we are interested in this paper.
\end{itemize}
We define $\operatorname{SH^*_{S^1,+}}(X)$ to be the homology of the chain complex $(\operatorname{CF^*_{S^1,+}}(X), \partial)$. We point out that this definition of 
 $\operatorname{SH^*_{S^1,+}}(X)$ agrees with the definition of 
 $\operatorname{SH^*_{S^1,+}}(X)$ as the cohomology of an 
$S^1$-complex, similar to the description given in 
Section~\ref{roudningfullyconvex}; see \cite{MR3671507}.

From now on, we assume that $(X,\lambda)$ is the unit codisk bundle $D^*L$ of a closed Riemannian manifold $L$ of dimension $n$ that admits a metric of non-positive sectional curvature. Up to an arbitrary high length (action) truncation, by Theorem \ref{perturb-metric}, we assume that every closed geodesic (closed Reeb orbit) has Morse index (Conley--Zehnder index) between $0$ and $n-1$. From now on, we assume that all the generators of the chain complex $\operatorname{CF^*_{S^1,+}}(D^*L)$ have actions smaller than a fixed number $A>0$. The grading on $\operatorname{CF^*_{S^1,+}}(D^*L)$ used in this document is the one used in \cite{Tonkonog:2018aa}. This is given by 
\begin{equation}\label{grading}
|\gamma_c|:=n-1-\mu(c),
\end{equation}
where $\mu(c)$ is the Morse index of the closed geodesic $c$ that lifts to the generator  $\gamma_c\in \operatorname{CF^*_{S^1,+}}(D^*L)$. For example, $\operatorname{CF^0_{S^1,+}}(D^*L)$ is generated by closed Reeb orbits that project to closed geodesics of Morse index $n-1$.

There is an isomorphism between the symplectic completion of $(D^*L, d\lambda_{\mathrm{can}})$ and the full cotangent bundle $(T^*L,d\lambda_{\mathrm{can}})$. Choose a point $p$ on the zero-section $L$ in  $T^*L$, a local divisor $D$ containing $p$, and a SFT-admissible almost complex structure on $(T^*L,d\lambda_{\mathrm{can}})$ that is integrable near $p$ such that $D$ is holomorphic.  For $k\geq 3$, let  $\mathcal{M}_{k}$ be the moduli space of $k$ distnict numbered marked points   $ z_1,z_2,\dots ,z_{k}\in \mathbb{CP}^1$ considered upto $\operatorname{Aut}(\mathbb{CP}^1,i)$. For a domain 
\[(\mathbb{CP}^1,z_0,z_1,\dots,z_k)\in \mathcal{M}_{k+1},\]
we consider the corresponding curve
with punctures
$D=\mathbb{CP}^1\setminus \{z_1,\dots,z_k\}$
equipped with cylindrical ends. For a choice of $k\geq 2$ generators  $\gamma_1, \dots, \gamma_k \in \operatorname{CF^0_{S^1,+}}(D^*L)$, define 
\begin{equation}\label{Hamil-perturb}
	\mathcal{M}^J_{T^*L}(\gamma_1, \dots, \gamma_k)\ll \mathcal{T}_D^{k-2}p\gg :=\left\{
	\begin{array}{l}
		(D, u), \text{ where } D\in \mathcal{M}_{k+1},\\
		u:D \to (T^*L,J),\\
		(du-  X_H\otimes \beta)^{0,1}=0,\\
		u(z_0)=p \text{  and satisfies $\ll \mathcal{T}_D^{k-2}p\gg$ at $z_0$}, \text{ and } \\
		u \text{ is asymptotic to  $\gamma_i$ at $z_i$ for $i=1,\dots,k$.} 
	\end{array}
	\right\}.
\end{equation}
Here, the coherent perturbation $K:=X_H\otimes \beta$ in the holomorphic curve equation is chosen as in Section \ref{enumerative}; for a precise definition we refer to {\cite[Sections 6, 7, and 9]{Cieliebak2018}}. 
For a generic coherent perturbation $K$, these moduli spaces are transversally cut out by {\cite[Corollary 7.6 and Section 9]{Cieliebak2018}} (which follows the transversality scheme introduced in \cite{Cieliebak_2007}). 

The moduli spaces (\ref{Hamil-perturb}) are rigid. Following {\cite[Section 4.4]{Tonkonog:2018aa}}, these rigid moduli spaces yield linear maps \[\langle \cdot|\cdots|\cdot \rangle: \operatorname{CF^0_{S^1,+}}(D^*L)^{\otimes k}\to \mathbb{Q}\] defined by
\[\langle \gamma_1|\gamma_2|\dots|\gamma_k \rangle:=\# 	\mathcal{M}^J_{T^*L}(\gamma_1, \dots, \gamma_k)\ll \mathcal{T}_D^{k-2}p\gg\, \in \mathbb{Q}\] 
whenever $\gamma_i$ are generators of $\operatorname{CF^0_{S^1,+}}(D^*L)$ and extend linearly to the full complex. In fact, by {\cite[Theorem 4.2]{Tonkonog:2018aa}}, the maps $\langle \cdot|\cdots|\cdot \rangle$  belong to a $2$-family of linear maps $\{\psi_{m-1}^k\}_{m\geq 1, k\geq 2}$ that defines, for each integer $m\geq 1$, an $L_\infty$-algebra structure on the full complex $\operatorname{CF^*_{S^1,+}}(X)$, where $X$ belongs to a more general class of Liouville domains. These linear operations are called \textit{gravitational descendants} in {\cite{Tonkonog:2018aa}}.

By {\cite[Proposition 4.4]{Tonkonog:2018aa}}, for each $k\geq 2$ the maps 
 $\langle \cdot|\cdots|\cdot \rangle: \operatorname{CF^0_{S^1,+}}(D^*L)^{\otimes k}\to \mathbb{Q}$
descend to cohomological level operations 
\begin{equation}\label{operations}
\langle \cdot|\cdots|\cdot \rangle: \operatorname{SH^0_{S^1,+}}(D^*L)^{\otimes k}\to \mathbb{Q}
\end{equation}
that are independant of the choices of  $p$, $D$, and compactly supported homotopies of $J$---the auxiliary data needed to define the moduli spaces (\ref{Hamil-perturb}). We will need the operations $\langle \cdot|\cdots|\cdot \rangle$ to count certain holomorphic buildings in terms of the ``Borman--Sheridan class'' in Section \ref{Borman-}.
			
\section{Proof of Theorem \ref{extremal-lag-ball}}\label{proof for ball}
In this section, we prove that every extremal Lagrangian torus in the standard symplectic unit ball $(\bar{B}^{2n}(1),\omega_{\mathrm{std}})$  lies entirely on the boundary $\partial B^{2n}(1)$.
\subsection{Preparing a geometric setup}
Let $L\subset (\bar{B}^{2n}(1),\omega_{\mathrm{std}})$ be an extremal Lagrangian torus that intersects the interior of the ball $\bar{B}^{2n}(1)$.  We aim to arrive at a contradiction using the monotonicity property of the area of pseudoholomorphic curves (cf. Lemma \ref{monotoncity1}).

The complex projective space $\mathbb{CP}^n$ gets a symplectic form from the standard unit  sphere $(S^{2n-1},\omega_{\mathrm{std}})\subseteq (\mathbb{C}^{n},\omega_{\mathrm{std}})$ under symplectic reduction. This symplectic form is known as the Fubini--Study form, and we denote it by $\omega_{\mathrm{FS}}$. Each complex line in $\mathbb{CP}^n$ has sympletic area $\pi$ with respect to $\omega_{\mathrm{FS}}$. Let $\epsilon>0$ and set $\omega_{\mathrm{FS,\epsilon}}:=\frac{(1+\epsilon)}{\pi} \omega_{\mathrm{FS}}$. We have a symplectic inclusion
\[L\subset \bigl(B^{2n}(1+\epsilon), \omega_{\mathrm{std}}\bigr)=\bigl(\mathbb{CP}^n\setminus \mathbb{CP}^{n-1},\omega_{\mathrm{FS,\epsilon}}\bigr)\subset \bigl(\mathbb{CP}^n,\omega_{\mathrm{FS,\epsilon}}\bigr),\]
where $\mathbb{CP}^{n-1}$ is the hypersurface at infinity. We now see $L$ as a Lagrangian torus in the complex projective space $(\mathbb{CP}^n,\omega_{\mathrm{FS,\epsilon}})$. This enables us to apply a powerful neck-stretching argument of Cieliebak and Mohnke from \cite{Cieliebak2018}.

Pick a flat metric $g_0$ on $L$. For $a>0$, denote by $g_a$ the scaling of $g_0$ by $c$, i.e., $g_a:=a g_0$. By Weinstein neighborhood theorem, there exists a constant $a>0$ such that  we can symplectically embed the codisk bundle of $g_a$-radius $2$ of $L$ into $(B^{2n}(1+\epsilon/2), \omega_{\mathrm{std}})$.  By Theorem \ref{perturb-metric}, there exists a metric $g$ on  $L$ obtained by a small perturbation  of $g_a$ such that, for $g$, every closed geodesic $\gamma$ of length less than or equal to $(1+\epsilon)$ is non-contractible, non-degenerate (as a critical point of the energy functional), and satisfies
\[0\leq \operatorname{\mu}(\gamma)\leq n-1,\]
where $\operatorname{\mu}(\gamma)$ denotes the Morse index of $\gamma$. Let $(D^*L, d\lambda_{\mathrm{can}})$ be the symplectic unit codisk bundle of $L$ with respect to the perturbed metric $g$, i.e.,  
\[D^*L=\bigl\{v\in T^*L: \|v\|_g\leq 1 \bigr\}\]
and $\lambda_{\mathrm{can}}$ is the canonical $1$-form on $T^*L$. The unit cosphere bundle $\Pi:(S^*L,\lambda_{\mathrm{can}})\to L$ given by 
\[S^*L:=\Bigl \{v\in T^*L: \|v\|_g=1\Bigr\}\]
is a contact type boundary of $(D^*L,  d\lambda_{\mathrm{can}})$. So $(D^*L,  d\lambda_{\mathrm{can}})$ is a Liouville domain, and its symplectic completion is naturally isomorphic to the full cotangent space $(T^*L, d\lambda_{\mathrm{can}})$.

Since $g$ can be chosen to be a small perturbation of $g_a$,  we can identify a neighborhood of $L$ in $B^{2n}(1+\epsilon)$  with $D^*L$ via a symplectic embedding 
\[\Phi: (D^*L,d\lambda_{\mathrm{can}})\to (B^{2n}(1+\epsilon/2), \omega_{\mathrm{std}})\]
such that
\[\Phi|_{L}=\operatorname{Id}_L.\]

We will simply denote $(\Phi(D^*L),\Phi_*(d\lambda_{\mathrm{can}}))$ by $(D^*L,d\lambda_{\mathrm{can}})$ and the contact type hypersurface \[(\Phi(S^*L),\Phi_*\lambda_{\mathrm{can}}|_{S^*L}) \subset(\mathbb{CP}^{n},\omega_{\mathrm{FS,\epsilon}})\]
by $(S^*L,\lambda_{\mathrm{can}}|_{S^*L})$.

By our assumption, the Lagrangian torus $L$ intersects the open ball $B^{2n}(1)$. Fix a point $p\in L\cap B^{2n}(1)$. 
Let  $D^{2n}(\delta)$ be the closed disk of radius $\delta>0$ centered at the origin in $\mathbb{R}^{2n}$. For sufficiently small $\delta>0$, choose a symplectic embedding $\phi:D^{2n}(\delta)\to B^{2n}(1)$ such that 
\[\phi(0)=p,\ \phi^{-1}(L)=D^{2n}(\delta)\cap \mathbb{R}^n.\]
For an illustration of this geometric setup, see Figure \ref{geometric-setup-ball}.

\begin{figure}[h]
	\centering
	\includegraphics[width=8cm]{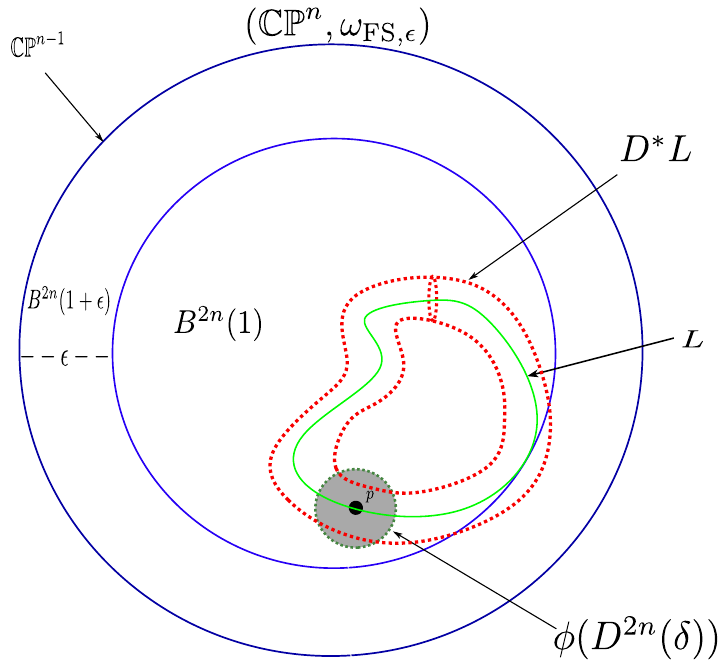}
	\caption{Geometric Setup}\label{geometric-setup-ball}
\end{figure}
The following monotonicity lemma is crucial to our argument.
\begin{lemma}{\cite[Lemma 3.4]{Cieliebak-Ekholm-Latsche-compactness2010}}\label{monotoncity1}
	Consider a symplectic manifold $(X, \omega)$ and a closed Lagrangian submanifold $L$. Fix a point $x\in X$ (which is allowed to lie on $L$) and an open neighborhood $U$ of $x$ in $X$. Let $J_{\mathrm{fix}}$ be an $\omega$-compatible almost complex structure on $U$.  Let $J$ denote any $\omega$-compatible almost complex structure on $X$ such that $J|_{U}=J_{\mathrm{fix}}$. Let $(S,j)$ be a closed connected Riemann surface with boundary $\partial S$. There exists a constant $C>0$ depending on $(X, \omega, J_{\mathrm{fix}},x,U)$ with the following property: For any non-constant continuous map $u:(S,\partial S)\mapsto (M,L)$ which is $(j,J)$-holomorphic on the interior, passing through the point $x\in X$ such that $u|_{u^{-1}(U)}:u^{-1}(U)\to U$ is a proper map, we have
	\[\int_{S}u^*\omega\geq C>0.\]
\end{lemma}
We will apply this lemma to the case $(X,L, \omega, J_{\mathrm{fix}},x,U)=(\mathbb{CP}^n,L,\omega_{\mathrm{FS,\epsilon}}, J_{\mathrm{std}},p,\phi(D^{2n}(\delta)) )$. In this case, we can choose $C=\frac{1}{2}\pi\delta^2$.

\subsection{Stretching the neck}
The contact type hypersurface 
\[(S^*L,\lambda_{\mathrm{can}}|_{S^*L}) \subset(\mathbb{CP}^{n},\omega_{\mathrm{FS,\epsilon}})\]
separates $\mathbb{CP}^{n}$ into two symplectic cobordisms
\[\bigl(\mathbb{CP}^{n}\setminus D^*L, \omega_{\mathrm{FS,\epsilon}}\bigr). \]
and 
\[\bigl(D^*L, d\lambda_{\mathrm{can}}\bigr).\]
Stretching the neck along $(S^*L,\lambda_{\mathrm{can}}|_{S^*L})$ in $\mathbb{CP}^n$ gives the split cobordism 
$\bigsqcup_{N=0}^{2} (\widehat{X}_N,\widehat \Omega_N)$, where 
\begin{equation*}
(\widehat{X}_N,\widehat \Omega_N):=
\begin{cases}
		\bigl(\widehat{\mathbb{CP}^{n}\setminus D^*L},\widehat{\omega}_{\mathrm{FS,\epsilon}}\bigr)\cong \bigl(\mathbb{CP}^{n}\setminus L,\omega_{\mathrm{FS,\epsilon}}\bigr) & \text{for } N=2,\\
		
		\bigl(\mathbb{R}\times S^*L,d(e^r \lambda_{\mathrm{can}}|_{S^*L})\bigr)& \text{for } N=1,\\
		\bigl(\widehat{D^*L}, \widehat{d\lambda}_{\mathrm{can}}\bigr)\cong \bigl(T^*L, d\lambda_{\mathrm{can}}\bigr)  & \text{for } N=0.\\
\end{cases}
\end{equation*}
Here
\begin{equation}\label{2-form}
	\widehat{\omega}_{\mathrm{FS,\epsilon}}:=
	\begin{cases}
\omega_{\mathrm{FS,\epsilon}}& \text{on } \mathbb{CP}^{n}\setminus D^*L,\\
		d(e^r\lambda_{\mathrm{can}}|_{S^*L}) & \text{on } (-\infty,0]\times  S^*L
\end{cases}
\end{equation}
and
\begin{equation*}
	\widehat{d\lambda}_{\mathrm{can}}:=
	\begin{cases}
		
		d(e^r\lambda_{\mathrm{can}}|_{S^*L})& \text{on } [0,\infty)\times S^*L,\\
		d\lambda_{\mathrm{can}} & \text{on } D^*L.
	\end{cases}
\end{equation*}

Choose a point $q\in L$ and a generic compatible almost complex structure  $J$ on $\mathbb{CP}^n$ that is integrable near $q$. We assume that $J$ restricted to a small neighborhood of the hypersurface $\mathbb{CP}^{n-1}$ at infinity is the standard complex structure $J_{\mathrm{std}}$ so that the hypersurface $\mathbb{CP}^{n-1}$ is $J$-holomorphic. More precisely, we assume that $J=J_{\mathrm{std}}$  on $\mathbb{CP}^{n}\setminus \bar{B}^{2n}(1+\epsilon/2)$. Moreover, $J=(\Phi|_{D^*L})_* J_{\mathrm{bot}}$ for some generic SFT-admissible almost complex structure $J_{\mathrm{bot}}$ on $(T^*L, d\lambda_{can})$.


Let $J_k$ be a family of almost complex structures on $\mathbb{CP}^n$ obtained via stretching $J$ along the hypersurface $S^*L$. We assume that $J_k$ constantly equals $J$ near the point $q$. As $k\to \infty$, $(\mathbb{CP}^n,J_k)$ splits into the cylindrical almost complex manifolds $(\mathbb{CP}^n\setminus L,J_{\infty})$, $(T^*L,J_{\mathrm{bot}})$ and $(\mathbb{R}\times S^*L,J_{\mathrm{cyl}})$. In particular, $J_\infty$ has the following property that will be crucial in our proof later:

\begin{itemize}
	\item $J_\infty$ restricted to a small neighborhood of the hypersurface $\mathbb{CP}^{(n-1)}$ at infinity is the standard complex structure $J_{\mathrm{std}}$. In particular, the hypersurface $\mathbb{CP}^{n-1}$ is $J_\infty$-holomorphic.
\end{itemize}

Fix a germ of a complex hypersurface $D$ near $q$. The rigid moduli space $\mathcal{M}^{J_k}_{\mathbb{CP}^n, [\mathbb{CP}^1]} \ll \mathcal{T}_D^{n-1}q\gg$ has a non-vanishing signed count of elements  for each $k$ by Theorem \ref{count}. Choose a sequence $u_k\in \mathcal{M}^{J_k}_{\mathbb{CP}^n, [\mathbb{CP}^1]} \ll \mathcal{T}_D^{n-1}q\gg$. By Theorem \ref{countlast}, we assume that each curve $u_k$ carries a Hamiltonian perturbation around the point $q$ as described in (\ref{hamilto-perturb}). By Theorem \ref{sft}, as $k \to \infty$, the sequence $u_k$ degenerates to a holomorphic building $\mathbb{H}=(\textbf{u}^0, \textbf{u}^1,\dots, \textbf{u}^{N_+})$ in $\bigsqcup_{N=0}^{N=N_+} (\widehat{X}_N,\widehat \Omega_N,\tilde{ \Omega}_N,J_N)$, for some integer $N_+\geq 0$, where 
\begin{equation*}
	(\widehat{X}_N,\widehat \Omega_N,\tilde{ \Omega}_N,J_N):=
	\begin{cases}
(\mathbb{CP}^{n}\setminus L, \omega_{\mathrm{FS,\epsilon}},\tilde{\omega}_{\mathrm{FS,\epsilon}}, J_{\infty}) & \text{for } N=N_+,\\
		
		(\mathbb{R}\times S^*L,d(e^r \lambda_{\mathrm{can}}),\pi^*d\lambda_{\mathrm{can}}|_{S^*L}, J_{\mathrm{cyl}})& \text{for } N\in \{1,\dots,N_+-1\},\\
		(T^*L, d\lambda_{\mathrm{can}},\tilde{d\lambda}_{\mathrm{can}}, J_{\mathrm{bot}}) & \text{for } N=0.\\
	\end{cases}
\end{equation*}
Here, we use the symplectic identifications 
\[\bigl(\widehat{D^*L}, \widehat{d\lambda}_{\mathrm{can}}\bigr)\cong (T^*L, d\lambda_{\mathrm{can}}),\]
and  
\begin{equation*}
	\bigl(\widehat{\mathbb{CP}^{n}\setminus D^*L},\widehat{\omega}_{\mathrm{FS,\epsilon}}\bigr)\cong (\mathbb{CP}^{n}\setminus L,\omega_{\mathrm{FS,\epsilon}}).
\end{equation*}
Moreover,  $\widehat{\omega}_{\mathrm{FS,\epsilon}}$ is defined in (\ref{2-form}), $\tilde{\omega}_{\mathrm{FS,\epsilon}}$ and $\tilde{d\lambda}_{\mathrm{can}}$ are the $2$-forms defined by  
\begin{equation*}
	\tilde{\omega}_{\mathrm{FS,\epsilon}}:=
	\begin{cases}
\omega_{\mathrm{FS,\epsilon}}& \text{on } \mathbb{CP}^{n}\setminus D^*L,\\
		\pi^*d\lambda_{\mathrm{can}}|_{S^*L} & \text{on } (-\infty,0]\times  S^*L,
\end{cases}
\end{equation*}
and 
\begin{equation*}
	\tilde{d\lambda}_{\mathrm{can}}:=
	\begin{cases}
d\lambda_{\mathrm{can}}& \text{on } D^*L,\\
		\pi^*d\lambda_{\mathrm{can}}|_{S^*L} & \text{on } [0,\infty)\times  S^*L,
\end{cases}
\end{equation*}
where $\pi: \mathbb{R}\times S^*L\to S^*L$ is the natural projection.

Since the building $\mathbb{H}$ is the limit of a sequence of genus zero rigid pseudoholomorphic curves in the homology class $[\mathbb{CP}^1]$, it has genus zero and represents the homology class $[\mathbb{CP}^1]$. Moreover, it has index zero, i.e., the sum of the indices of its curve components is equal to zero.

For some positive integer $k_N$, we write $\textbf{u}^N=(u^N_1, \dots, u^N_{k_N})$, where $u_i^N$ are the smooth connected punctured curves 
in the level $(\widehat{X}_N,\widehat \Omega_N,\tilde{ \Omega}_N,J_N)$ that constitute  the level $\textbf{u}^N$ of the building $\mathbb{H}$. Since the building $\mathbb{H}$ is the limit of a sequence of pseudoholomorphic spheres in the  homology class $[\mathbb{CP}^1]$, we have the following bound on the energy of $\mathbb{H}$ :
\begin{equation}\label{engery1}
	E(\mathbb{H}):=\sum_{N=0}^{N_+}\sum_{i=1}^{k_N}\int_{u_i^N}\tilde{ \Omega}_N=\sum_{i=1}^{k_{N^+}}\int (u_i^{N^+})^*\omega_{\mathrm{FS,\epsilon}}= \omega_{\mathrm{FS,\epsilon}}([\mathbb{CP}^1])=(1+\epsilon).
\end{equation}
If $\gamma$ is a closed Reeb orbit appearing as an asymptote of some curve component in the top level of the building,  say $u_i^{N^+}$, then 
\[(1+\epsilon)\geq \int (u_i^{N^+})^*\omega_{\mathrm{FS,\epsilon}}\geq\int_{(u_i^{N^+})^{-1}((-\infty,0]\times S^*L))}d(e^r\lambda_{\mathrm{can}})=\int_{(u_i^{N^+})^{-1}(\{0\}\times S^*L)}\lambda_{\mathrm{can}}\geq \int_{\gamma} \lambda_{\mathrm{can}}, \]
where the last inequality follows from the positivity of $d\lambda_{\mathrm{can}}|_{S^*L}$ on $J$-holomorphic curves in $(-\infty,0]\times S^*L$ (cf. {\cite[Lemma 2.6]{Cieliebak2018}}). Since the last integral is the action of $\gamma$ for the contact form on $S^*L=\partial D^*L$, this means that every closed Reeb orbit appearing in the holomorphic building $\mathbb{H}$ has an action at most  $(1+\epsilon)$ and is, therefore,  non-contractible, non-degenerate, and projects to a closed geodesic of Morse index between $0$ and $n-1$.
\subsection{Analyzing the resultant holomorphic building $\mathbb{H}$}\label{Analyze}
This section proves that, if $\epsilon>0$ is sufficiently small, then we have the following.
\begin{itemize}
	\item  There are no symplectization levels in  $\mathbb{H}$, i.e., the levels \[\textbf{u}^1, \textbf{u}^2,  \dots,\textbf{u}^{N_+-1}\] are all empty.
	\item The bottom level $\textbf{u}^0$ that sits in $(T^*L, J_{\mathrm{bot}})$ consists of a smooth connected $J_{\mathrm{bot}}$-holomorphic asymptotically cylindrical sphere $C_{\mathrm{bot}}$ with exactly $n+1$ positive punctures. Moreover, it inherits the tangency constraint  $\ll \mathcal{T}_D^{n-1}q\gg$ and the Hamiltonian perturbation supported around $q$.
	\item  The  top level $\textbf{u}^{N_+}$ that sits in $(\widehat{\mathbb{CP}^{n}\setminus D^*L}, J_{\infty})$ consists of $n+1$ asymptotically cylindrical somewhere injective $J_\infty$-holomorphic planes $u_1,u_2,\dots, u_n, u_{\infty}$ with negative ends on $S^*L$. The planes $u_1,u_2,\dots, u_n$ lie in the complement of the complex hypersurface $\mathbb{CP}^{n-1}$ and $u_{\infty}$ has a simple intersection with $\mathbb{CP}^{n-1}$.  Figure \ref{holomorphic building} illustrates the building $\mathbb{H}$.
	\end{itemize}
\begin{figure}[h]
	\centering
	\includegraphics[width=8cm]{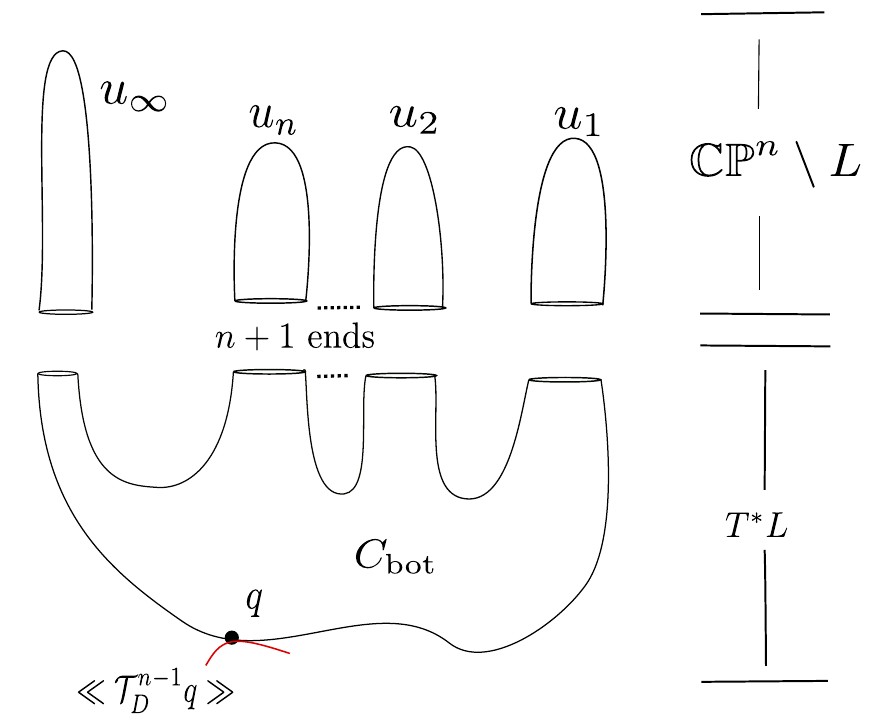}
	\caption{The holomorphic building $\mathbb{H}$}\label{holomorphic building}
\end{figure}
We establish this in the following sequence of lemmas.
\begin{lemma}
	The holomorphic building $\mathbb{H}$ has no node between its non-constant curve components. In particular, the building does not have a node at the constrained marked point $q$.
\end{lemma}
\begin{proof}
	Since $\mathbb{H}$ has genus zero, any node between non-constant components decomposes the building $\mathbb{H}$ into two pieces $A_1$ and $A_2$. These are represented by non-constant holomorphic curves (lying possibly at different levels). So, by {\cite[Lemma 2.6]{Cieliebak_2005}}, we have $\omega_{\mathrm{FS,\epsilon}}(A_1)>0, \omega_{\mathrm{FS,\epsilon}}(A_2)>0$. The building $\mathbb{H}$ has exactly one intersection with the $J_\infty$-holomorphic hypersurface $\mathbb{CP}^{n-1}$ at infinity because of the positivity of intersection and the fact that $\mathbb{H}$ is the result of breaking of a holomorphic sphere that has intersection number $+1$ with $\mathbb{CP}^{n-1}$. Therefore, one of these spheres, say $A_1$, lies in the complement of $\mathbb{CP}^{n-1}$. But since $\omega_{\mathrm{FS,\epsilon}}$ is exact in the complement of $\mathbb{CP}^{n-1}$, we must have $\omega_{\mathrm{FS,\epsilon}}(A_1)=0$ by Stokes' theorem. This is a contradiction.
\end{proof}
\begin{lemma}\label{countendsincotangentbundle}
	Let $C_{\mathrm{bot}}$ be the non-constant smooth punctured sphere in $T^*L$ that carries the tangency constraint $\ll \mathcal{T}_D^{n-1}q\gg$ at $q$. Then $C_{\mathrm{bot}}$ has at least $n+1$ positive asymptotically cylindrical ends.
\end{lemma}
\begin{proof} Suppose $C_{\mathrm{bot}}$ has positive ends on the Reeb orbits $\gamma_1,\dots,\gamma_{l}$, for some positive integer $l$. Choose a trivialization $\tau$ that extends to $C_{\mathrm{bot}}^*TT^*L$. In this trivialization the Fredholm index of $C_{\mathrm{bot}}$ (cf.\cite[Proposition 3.1]{Cieliebak2018}), taking the constraint $\ll \mathcal{T}_D^{n-1}q\gg$ into consideration, is given by
\[\operatorname{ind}(C_{\mathrm{bot}})=(n-3)(2-l)+\sum_{i=1}^{l}\operatorname{CZ^{\tau}}(\gamma_i)-2n+2-2(n-1).\]
From Theorem \ref{special-trivalization}, it follows that	
\[\sum_{i=1}^{l}\operatorname{CZ}^\tau(\gamma_i)+\sum_{i=1}^{l}\mu^{\kappa(\tau)}(c_i)=\sum_{i=1}^{l}\mu(c_i),\]
where $\mu(c_i)$ denotes the Morse index of the geodesic $c_i$ that lifts to $\gamma_i$ and  $\mu^{\kappa(\tau)}(c_i)$ denotes its Maslov index. Since $C_{\mathrm{bot}}$ is a null-homology of cycles $\gamma_1,\dots,\gamma_{l}$, we have \[\sum_{i=1}^{l}\mu^{\kappa(\tau)}(c_i)=0.\]
So
	\[\sum_{i=1}^{l}\operatorname{CZ_{\tau}}(\gamma_i)=\sum_{i=1}^{l}\mu(c_i) \leq
	l(n-1).\]
This implies
	\[\operatorname{ind}(C_{\mathrm{bot}})\leq 2(l-(n+1)).\]
On the other hand, the curve $C_{\mathrm{bot}}$ carries a Hamiltonian perturbation $H$ supported around the point $q$ which ensures its transversality. For generic $H$, we have $\operatorname{ind}(C_{\mathrm{bot}})\geq 0$. For this, it is necessary that $l\geq n+1$.
\end{proof}

{\cite[Lemma 3.2 and Corollary 3.3]{Cieliebak2018}} are closely related to our Lemma~\ref{countendsincotangentbundle}. However, in our setting the relevant curves satisfy a Hamiltonian-perturbed equation, and this feature makes the proof considerably simpler.

\begin{lemma}\label{count-disks}
	The building  $\mathbb{H}$ has at least $n+1$ negatively asymptotically cylindrical $J_{\infty}$-holomorphic planes in the top level  $\mathbb{CP}^n\setminus L$.
\end{lemma}
\begin{proof}
	The underlying graph of the building $\mathbb{H}$ is a tree since the building $\mathbb{H}$ has genus zero. By Lemma \ref{countendsincotangentbundle}, at least $n+1$ edges emanate from the vertex $C_{\mathrm{bot}}$ in the underlying graph. We order these edges from
	$1,2, \dots, n+1,\dots, m$, for some positive integer $m\geq n+1$. Let $C_i$ be the union of the curves that constitute the subtree emanating from the vertex  $C_{\mathrm{bot}}$ along the $i$th edge. Since the building has genus zero, all $C_i$ are topological planes. The ends of $C_{\mathrm{bot}}$ are on non-contractible Reeb orbits, so each $C_i$ must have a curve component in the top level $\mathbb{CP}^n\setminus L$, which is a $J_\infty$-holomorphic plane negatively asymptotic to a Reeb orbit on $S^*L$. So each $C_i$ corresponds to a negatively asymptotically cylindrical $J_{\infty}$-holomorphic plane in $\mathbb{CP}^n\setminus L$. The underlying graph of the building $\mathbb{H}$ is a tree implying that we have at least $n+1$ negatively asymptotically cylindrical $J_{\infty}$-holomorphic planes, denoted by $u_1,u_2,\dots, u_{n}, u_{\infty}$, in $\mathbb{CP}^n\setminus L$. 
\end{proof}

\begin{lemma}\label{count-intersections}
	Let $u_1,u_2,\dots,u_n, u_{n+1}, \dots, u_{m}$ be all of the negatively asymptotically cylindrical $J_{\infty}$-holomorphic planes in the top level  $\mathbb{CP}^n\setminus L$ of $\mathbb{H}$. At most one of these can intersect the hypersurface $\mathbb{CP}^{n-1}$ at infinity, and the intersection number will be $+1$. In particular, at least $n$ planes $u_1,u_2,\dots,u_n$ lie in the affine symplectic chart $B^{2n}(1+\epsilon)$.
\end{lemma}\begin{proof}
	Recall that the hypersurface at infinity $\mathbb{CP}^{n-1}$ is $J_{\infty}$-holomorphic, $\mathbb{H}$ is the limit of a sequence of pseudo-holomorphic spheres that has homological intersection number with $\mathbb{CP}^{n-1}$ equal to $+1$, and intersections between complex objects count positively. Therefore, exactly one curve component of $\mathbb{H}$  in the top level $\mathbb{CP}^n\setminus L$ intersects the hypersurface $\mathbb{CP}^{n-1}$.
\end{proof}
\begin{lemma}\label{excontradiction}
	Suppose $\epsilon< \frac{1}{n}$. Then the following holds. 
	\begin{itemize}
		\item The top level of the building $\mathbb{H}$  consists of  $n+1$ negatively asymptotically cylindrical $J_{\infty}$-holomorphic planes, denoted by  $u_1,u_2,\dots,u_n, u_{\infty}$. All these planes are somewhere injective. Moreover, the planes $u_1,u_2,\dots,u_n$ that do not intersect the hypersurface $\mathbb{CP}^{n-1}$ are asymptotic to simple Reeb orbits (equivalently simple geodesics on $L$). The plane $u_{\infty}$ is the component of the building that intersects the hypersurface $\mathbb{CP}^{n-1}$ at infinity.
		\item The building has no symplectization levels.
	\end{itemize}
	
\end{lemma}
\begin{proof}
	Let $u_1,u_2,\dots,u_{m}, u_{\infty}$ be the $J_{\infty}$-holomorphic planes contained in the top level  $\mathbb{CP}^n\setminus L$ of $\mathbb{H}$, where $m\geq n$ by Lemma \ref{count-disks}. By  Lemma \ref{count-intersections}, we can assume that $u_1,u_2,\dots,u_{m}$ do not intersect the hypersurface $\mathbb{CP}^{n-1}$ and hence are contained in the affine symplectic chart $(B^{2n}(1+\epsilon),\omega_{\mathrm{std}})$. By the symplectic identification 
	\begin{equation}\label{identification}
		\bigl(\widehat{\mathbb{CP}^{n}\setminus D^*L},\widehat{\omega}_{\mathrm{FS,\epsilon}}\bigr)\cong (\mathbb{CP}^{n}\setminus L,\omega_{\mathrm{FS,\epsilon}}),
	\end{equation}
	these  $J_\infty$-holomorphic planes can be compactified to $m$ smooth disks in $B^{2n}(1+\epsilon)$ with boundaries on $L$. The Lagrangian torus $L$ being extremal in $(\bar{B}^{2n}(1),\omega_{\mathrm{std}})$ implies that
	\[\sum_{i=1}^{m}\int u^*_i\omega_{\mathrm{FS,\epsilon}}=\sum_{i=1}^{m}\int u^*_i\omega_{\mathrm{std}}\geq m \frac{1}{n}. \]
	
	By (\ref{engery1}), the energy of the building $\mathbb{H}$ is $(1+\epsilon)$. Therefore, 
	\begin{equation}\label{engerybound}
		\sum_{i=1}^{m}\int u^*_i\omega_{\mathrm{std}}+\int u^*_{\infty}\omega_{\mathrm{FS,\epsilon}}\leq 1+\epsilon. 
	\end{equation}
	So
	\[\int u^*_{\infty}\omega_{\mathrm{FS,\epsilon}}\leq 1+\epsilon-m\frac{1}{n}. \]
	If $m>n$, then the assumption   $\epsilon< \frac{1}{n}$ implies 
	\[\int u^*_{\infty}\omega_{\mathrm{FS,\epsilon}}<0.\]
This is a contradiction because $u_{\infty}$ is a non-constant $J_\infty$-holomorphic curve and $J_\infty$ is compatible with $\omega_{\mathrm{FS,\epsilon}}$. So we must have $m\leq n$. By Lemma \ref{count-disks}, we have $m\geq n$. Therefore,  we have $m=n$, i.e., there are exactly $n+1$ planes $u_1,u_2,\dots,u_n, u_{\infty}$ in the top level $\mathbb{CP}^n\setminus L$ of the building $\mathbb{H}$.
	
Next we prove that each of these planes is somewhere injective. Assume that one of these planes, say $u_1$, is multiply covered with multiplicity $k>1$. Let $\bar{u}_1$ be the underlying somewhere injective plane. Since $L$ is extremal, we have 
	\[\sum_{i=1}^{n}\int u^*_i\omega_{\mathrm{std}}= k \int \bar{u}_1 ^*\omega_{\mathrm{std}}+\sum_{i=2}^{n}\int u^*_i\omega_{\mathrm{std}} \geq 1+(k-1)\frac{1}{n}. \]
So
	\[\int u^*_{\infty}\omega_{\mathrm{FS,\epsilon}}\leq \epsilon-(k-1)\frac{1}{n}. \]
	The assumption  $\epsilon< \frac{1}{n}$ again implies the contradiction
	\[\int u^*_{\infty}\omega_{\mathrm{FS,\epsilon}}<0.\]

One can write
	\[\sum_{i=1}^{n}\int u^*_i\omega_{\mathrm{std}}=\sum_{i=1}^{n}\int_{\gamma_i}\lambda_{\mathrm{std}}\]
	where $d \lambda_{\mathrm{std}}=\omega_{\mathrm{std}}$ and $\gamma_i$ with $i\neq \infty$ is the closed geodesic that lifts to the closed Reeb orbit to which $u_i$ is negatively asymptotic. By an argument similar to above, if one or more $\gamma_i$ are multiply covered, then  
	\[\int u^*_{\infty}\omega_{\mathrm{FS,\epsilon}}<0\]
which is a contradiction.
	
	The plane $u_{\infty}$ must be the component of the building that intersects the complex hypersurface $\mathbb{CP}^{n-1}$, otherwise all of the planes $u_1,u_2,\dots,u_{n}, u_{\infty}$  are contained in the affine symplectic chart $(B^{2n}(1+\epsilon),\omega_{\mathrm{std}})$. By the argument above, this implies the contradiction
	\[\int u_{\infty}^*\omega_{\mathrm{FS,\epsilon}}<0.\]
	
The intersection number of $u_{\infty}$ with the complex hypersurface $\mathbb{CP}^{n-1}$ is $+1$ by Lemma \ref{count-intersections}. This means $u_{\infty}$  is somewhere injective. 

Next, we argue, following arguments similar to above, that the top level of $\mathbb{H}$ contains no component other than the $n+1$ negatively asymptotically cylindrical $J_{\infty}$-holomorphic planes  $u_1,u_2,\dots,u_n, u_{\infty}$. 

Suppose, on the contrary, that there is a smooth non-constant curve component $C$ in the top level $\mathbb{CP}^n\setminus L$ other than the planes  $u_1,u_2,\dots,u_n, u_{\infty}$. By the previous discussion, $C$ must lie in the  affine symplectic chart $(B^{2n}(1+\epsilon),\omega_{\mathrm{std}})$. We compactify $C$ to a surface in $B^{2n}(1+\epsilon)$ with a boundary on $L$. By the Hurewicz theorem, we have that $\pi_2(B^{2n}(1+\epsilon),T^n)=H_2(B^{2n}(1+\epsilon),T^n)$. Since $L$ is extremal in $(\bar{B}^{2n}(1),\omega_{\mathrm{std}})$, this means that
\[\int C^*_i\omega_{\mathrm{std}}\geq \frac{1}{n}.\]
Also 
\[\int C^*_i\omega_{\mathrm{std}}+\sum_{i=1}^{n}\int u^*_i\omega_{\mathrm{std}}+\int u^*_{\infty}\omega_{\mathrm{FS,\epsilon}}\leq 1+\epsilon. \]
Combining the above two estimates, we get
\[0\leq\int u^*_{\infty}\omega_{\mathrm{FS,\epsilon}}\leq 1+\epsilon-\int C^*_i\omega_{\mathrm{std}}-\sum_{i=1}^{n}\int u^*_i\omega_{\mathrm{std}}\leq \epsilon-\frac{1}{n}. \]
For $\epsilon<1/n$, this is a contradiction. We conclude that the top level of $\mathbb{H}$ contains no component other than the $n+1$ negatively asymptotically cylindrical $J_{\infty}$-holomorphic planes  $u_1,u_2,\dots,u_n, u_{\infty}$.

	For the second bullet point, we prove that symplectization levels consist of trivial cylinders over closed Reeb orbits, and hence we disregard them since non-stable components are disregarded in the limit.

	The underlying graph of the building $\mathbb{H}$ is a tree because the building $\mathbb{H}$ has genus zero. By Lemma \ref{countendsincotangentbundle}, exactly $n+1$ edges emanate from the vertex  $C_{\mathrm{bot}}$ in the underlying graph. Let $C_1, C_2,\dots, C_n,C_{n+1}$ be the trees that emanate from $n+1$ ends of $C_{\mathrm{bot}}$. Each $C_i$ is a topological plane. Suppose a smooth connected punctured curve in $C_1$ has more than one negative end. By the maximum principle, this curve must have at least one positive puncture. Removing this curve component from $C_1$ will produce at least three connected components of the tree $C_1$, of which two are topological planes. The intersection of these two topological planes with the top level $\mathbb{CP}^n\setminus L$ contains at least two disks because $S^*L$ has no contractible closed Reeb orbits. Thus, we have more than $n+1$ disks in the top level, which is a contradiction by the previous discussion.
	
	Suppose $C_1$ has a smooth connected curve component with at least two positive ends. The trees emerging from the positive ends of this component intersect the top level by the maximum principle. The intersection contains a plane and a compact surface (not necessarily embedded) with a boundary on $L$. The existence of this compact surface contradicts the conclusions made in the previous paragraphs. Thus, every smooth connected component of each $C_i$ in the symplectization is a holomorphic cylinder with a negative and a positive end. This cylinder is trivial. It will become clear from the upcoming Lemmas.  
\end{proof}
Putting everything together, we conclude that the building $\mathbb{H}$ has only two levels. The bottom level of the building $\mathbb{H}$ in $T^*L$ contains a smooth connected punctured sphere $C_{\mathrm{bot}}$ with exactly $n+1$ positive ends that carries the tangency constraint $\ll \mathcal{T}_D^{n-1}q\gg$ and the Hamiltonian perturbation around $q$. The top level  $\mathbb{CP}^n\setminus L$ consists of $n+1$ negatively asymptotically cylindrical simple  $J_{\infty}$-holomorphic planes $u_1,u_2,\dots,u_{n}, u_{\infty}$. Since the building is connected, there are no curve components other than $u_1,u_2,\dots,u_{n}, u_{\infty}$,  and $C_{\mathrm{bot}}$. This completes the proof of the claim we made at the beginning of this section.
\subsection{Preparing a counting problem}\label{countingproblem}
By  Lemma \ref{count-intersections}, exactly one of the  asymptotically cylindrical $J_{\infty}$-holomorphic planes $u_1,u_2,\dots,u_n, u_{\infty}$ in $\mathbb{CP}^n\setminus L$ intersects the hypersurface $\mathbb{CP}^{n-1}$ at infinity. From now on, we assume it is the plane $u_{\infty}$.

Let $\gamma_1,\gamma_2,\dots,\gamma_n, \gamma_{\infty}$ be the asymptotic closed Reeb orbits of the $J_\infty$-holomorphic planes $u_1,u_2,\dots,u_n, u_{\infty}$, respectively. Define	
\begin{equation}\label{count-prob}
	\mathcal{M}^{J_\infty}_{\mathbb{CP}^n\setminus L}(\gamma_i):=\left\{
	\begin{array}{l}
		u:\mathbb{CP}^1\setminus \{\infty\} \to (\widehat{\mathbb{CP}^n\setminus D^*L},J_{\infty}),\\
		du\circ i=J_{\infty}\circ du  ,\\
		u \text{ is asymptotic to  $\gamma_i$ at $\infty$,} \\
	\text{and } [u]=[u_i]\in \pi_2(\mathbb{CP}^n,L).
	\end{array}
	\right\}\bigg/\operatorname{Aut}(\mathbb{CP}^1,\infty)
\end{equation}
where $i=1,2,\dots,n,\infty$.

Next, we compute the symplectic areas of the $J_\infty$-holomorphic planes $u_1,u_2,\dots,u_n,u_{\infty}$. Note that under the symplectomorphic identification
\[
\bigl(\widehat{\mathbb{CP}^{n}\setminus D^*L},\widehat{\omega}_{\mathrm{FS,\epsilon}}\bigr)\cong (\mathbb{CP}^{n}\setminus L,\omega_{\mathrm{FS,\epsilon}}),
\]
each of these planes admits a compactification to a smooth disk in $\mathbb{CP}^n$ with boundary on $L$. It therefore suffices to determine the symplectic areas of the corresponding compactified disks. By abuse of notation, we will identify the planes $u_i$ with the discs when convenient.
\begin{lemma}\label{estemate-engery} 
The disks $u_1,u_2,\dots,u_n, u_{\infty}$ satisfy
\[\int u^*_{i}\omega_{\mathrm{std}}=\frac{1}{n}, \]
for all $i=1,2,\dots,n$, and 
	\[\int u^*_{\infty}\omega_{\mathrm{FS,\epsilon}}= \epsilon. \]
\end{lemma}
\begin{proof}
	Since $L$ is extremal in the ball $(\bar{B}^{2n}(1),\omega_{\mathrm{std}})$ and the disks $u_1,u_2,\dots,u_n$ are contained in the affine symplectic chart $(\bar{B}^{2n}(1+\epsilon),\omega_{\mathrm{std}})$, we find
	\[\int u^*_i\omega_{\mathrm{std}}\geq \frac{1}{n} \]
	for all $i=1,2,\dots,n$.
	Here, note that the symplectic area of a disk $u:(D^2,\partial D^2)\to (\bar{B}^{2n}(1+\epsilon), L)$ depends only on its relative homotopy class $[u]\in \pi_2(\bar{B}^{2n}(1+\epsilon), L)$. Every disk  $u:(D^2,\partial D^2)\to (\bar{B}^{2n}(1+\epsilon), L)$ can be smoothly deformed to a disk $u:(D^2,\partial D^2)\to (\bar{B}^{2n}(1), L)$ without changing its relative homotopy class.
	
By (\ref{engery1}), the energy of the building $\mathbb{H}$ is $(1+\epsilon)$. Therefore, 
	\begin{equation}\label{engerybound1}
		\sum_{i=1}^{n}\int u^*_i\omega_{\mathrm{std}}+\int u^*_{\infty}\omega_{\mathrm{FS,\epsilon}}= 1+\epsilon. 
	\end{equation}
	The energy estimate (\ref{engerybound1}) implies
	\[0<\int u^*_{\infty}\omega_{\mathrm{FS,\epsilon}}= \epsilon. \]
	Suppose one of the disks  $u_1,u_2,\dots,u_n$, say $u_1$, has symplectic area strictly greater than $1/n$. Since $L$ is extremal in the ball $\bar{B}^{2n}(1)$, the symplectic area of $u_1$ is positive integer multiple of $1/n$. In particular, 
	\[\int u^*_{1}\omega_{\mathrm{std}}\geq 2\frac{1}{n}.\]

	The Equation (\ref{engerybound1}) and the assumption $\epsilon< \frac{1}{n}$ lead to the contradiction
	\[\int u^*_{\infty}\omega_{\mathrm{FS,\epsilon}}= \epsilon-\frac{1}{n}<0. \]
	Therefore, we must have
	 \[\int u^*_{i}\omega_{\mathrm{std}}=\frac{1}{n}, \]
for all $i=1,2,\dots,n$, and 
	\[\int u^*_{\infty}\omega_{\mathrm{FS,\epsilon}}= \epsilon. \qedhere\]
\end{proof}
\begin{lemma}\label{index-estemate}
The following holds. 
\begin{itemize}
\item 	Suppose the positive ends of $C_{\mathrm{bot}}$ are  asymptotic to the closed Reeb orbits $\gamma_1, \gamma_2, \dots, \gamma_{n},\gamma_{\infty}$. Then the  Fredholm index of $C_{bot}$ is zero and 
\[\operatorname{\mu}(c_i)=n-1,\]
for all $i=1,2,\dots, n, \infty$, where $\mu(c_i)$ denotes the Morse index of the geodesic $c_i$ that lifts to $\gamma_i$.
\item The $n+1$ negatively asymptotically cylindrical $J_{\infty}$-holomorphic planes  $u_1,u_2,\dots,u_n, u_{\infty}$  in the top level $\mathbb{CP}^n\setminus L$ have Fredholm index zero.
	\end{itemize}
\end{lemma}
\begin{proof}
The curve $C_{\mathrm{bot}}$ carries a Hamiltonian perturbation $H$ supported around the point $q$. For generic $H$, we have
	\[\operatorname{ind}(C_{\mathrm{bot}})\geq 0.\] 

Let  $\tau$ be a symplectic trivialization along the ends of $C_{\mathrm{bot}}$ that extends to $C_{\mathrm{bot}}^*TT^*L$. The Fredholm index of $C_{\mathrm{bot}}$ (cf.\cite[Prop. 3.1]{Cieliebak2018}), taking the constraint $\ll \mathcal{T}_D^{n-1}q\gg$ into consideration, is given by
\[\operatorname{ind}(C_{\mathrm{bot}})=(n-3)(2-(n+1))+\sum_{i=1}^{n}\operatorname{CZ^{\tau}}(\gamma_i)+\operatorname{CZ^{\tau}}(\gamma_\infty)-2n+2-2(n-1).\]
After simplifying and using $\operatorname{ind}(C_{\mathrm{bot}})\geq 0$  we obtain
\begin{equation}\label{indexfinal12}
	n^2-1\leq \sum_{i=1}^{n}\operatorname{CZ^{\tau}}(\gamma_i)+\operatorname{CZ^{\tau}}(\gamma_\infty).
\end{equation}   
On the other hand, Theorem \ref{special-trivalization} implies that	
\[\sum_{i=1}^{n}\operatorname{CZ^{\tau}}(\gamma_i)+\operatorname{CZ^{\tau}}(\gamma_\infty)+\sum_{i=1}^{n}\mu^{\kappa(\tau)}(c_i)+\mu^{\kappa(\tau)}(c_\infty)=\sum_{i=1}^{n}\mu(c_i)+\mu(c_\infty)\]
where $\mu(c_i)$ is the Morse index of the geodesic $c_i$ corresponding to $\gamma_i$ and  $\mu^{\kappa(\tau)}(c_i)$ denotes its Maslov index. Since $C_{\mathrm{bot}}$ is a null-homology of the cycles $\gamma_1,\dots,\gamma_{n}, \gamma_\infty$, we have \[\sum_{i=1}^{k}\mu^{\kappa(\tau)}(c_i)+\mu^{\kappa(\tau)}(c_\infty)=0.\]
This implies
\[\sum_{i=1}^{n}\operatorname{CZ^{\tau}}(\gamma_i)+\operatorname{CZ^{\tau}}(\gamma_\infty)=\sum_{i=1}^{n}\mu(c_i)+\mu(c_\infty).\]
So  (\ref{indexfinal12}) implies
\begin{equation}\label{indexfinal123}
	n^2-1\leq  \sum_{i=1}^{n}\mu(c_i)+\mu(c_\infty).
\end{equation} 
On the other hands, Theorem \ref{perturb-metric} implies that for all $i\in \{1,\dots,k,\infty\}$ we have
\[0\leq  \mu(c_i)\leq n-1.\]
So, the only solution to (\ref{indexfinal123}) is 
\[\operatorname{\mu}(\gamma_i)=n-1\]
for all $i=1,2,\dots,n, \infty.$ Hence, $\operatorname{ind}(C_{\mathrm{bot}})=0$. This completes the proof of the first bullet point.

The index of the holomorphic building $\mathbb{H}$ is zero. So
\begin{equation}\label{indd}
		\operatorname{ind}(C_{\mathrm{bot}})+\operatorname{ind}(u_\infty)+\sum_{i=1}^{n}\operatorname{ind}(u_i)=0
\end{equation}

The $n+1$ negatively asymptotically cylindrical $J_{\infty}$-holomorphic planes  $u_1,u_2,\dots,u_n, u_{\infty}$  in $\mathbb{CP}^n\setminus L$  are somewhere injective by Lemma \ref{excontradiction}. So for generic $J_{\infty}$, we have
	\[\operatorname{ind}(u_i)\geq 0\]
	for all $i=1,2,\dots,n,\infty$.
Thus, Equation (\ref{indd}) implies that 
\[\operatorname{ind}(u_i)= 0,\]
for all $i=1,2,\dots,n, \infty$. \qedhere
\end{proof}
\begin{remark}
It follows directly from the arguments used in the proofs of Lemmas 6.6--6.8 that these lemmas remain valid for any holomorphic building representing the class $[\mathbb{CP}^1]$ whose bottom level satisfies the tangency condition $\ll \mathcal{T}_D^{\,n-1} q \gg$, and not only for those obtained as limits under neck-stretching.
\end{remark}
\begin{theorem}\label{counting}
Suppose $L$ is extremal in the ball $(\bar{B}^{2n}(1),\omega_{\mathrm{std}})$ and  $\epsilon< \frac{1}{n}$. The moduli spaces $\mathcal{M}^{J_{\infty}}_{\mathbb{CP}^n\setminus L}(\gamma_i)$ are all compact. Moreover, for generic $J_{\infty}$, each one of these moduli spaces is an oriented smooth manifold of dimension zero. The signed count $\# \mathcal{M}^{J_{\infty}}_{\mathbb{CP}^n\setminus L}(\gamma_{i})$ does not depend on the choice of generic almost complex structure and depends only on the Hamiltonian isotopy class of $L$.
\end{theorem}
\begin{proof}
	By Lemma \ref{index-estemate}, for generic $J_{\infty}$, the virtual dimension of each $\mathcal{M}^{J_{\infty}}_{\mathbb{CP}^n\setminus L}(\gamma_i)$ is zero. By Lemma \ref{excontradiction}, each of $\mathcal{M}^{J_{\infty}}_{\mathbb{CP}^n\setminus L}(\gamma_i)$ consists of somewhere injective curves. Therefore, by standard trasversality arguments, for generic $J_\infty$, each $\mathcal{M}^{J_{\infty}}_{\mathbb{CP}^n\setminus L}(\gamma_i)$ is a smooth oriented manifold of dimension zero.
	
	To prove compactness of $\mathcal{M}^{J_{\infty}}_{\mathbb{CP}^n\setminus L}(\gamma_{\infty})$, assume	a sequence in $\mathcal{M}^{J_{\infty}}_{\mathbb{CP}^n\setminus L}(\gamma_{\infty})$ degenerates to a non-trivial holomorphic building $\mathbb{H_\infty}$. The building $\mathbb{H_\infty}$ has only one negative end, on $\gamma_{\infty}$. The building $\mathbb{H_\infty}$ must have some symplectization levels: If not, then it must have two non-constant components, say $u_1$ and $u_2$,  in $\mathbb{CP}^n\setminus L$. By positivity of intersection, only one of these, say $u_1$, can intersect the hypersurface $\mathbb{CP}^{n-1}$. The negative end of $\mathbb{H_\infty}$ asymptotic to $\gamma_{\infty}$ is either inherited by $u_1$ or $u_2$. So we can assume $u_1$ is a $J_\infty$-holomorphic sphere and $u_2$ has a negative end on $\gamma_{\infty}$. Since the symplectic form is exact on the complement of the hypersurface $\mathbb{CP}^{n-1}$, $u_1$ must intersect $\mathbb{CP}^{n-1}$. This means $u_2$ is a disk in the affine chart $B^{2n}(1+\epsilon)$ with boundary on $L$. By extremality of $L$, the symplectic area of $u_2$ must be at least $1/n$. This area is greater than the symplectic area of $\mathcal{M}^{J_{\infty}}_{\mathbb{CP}^n\setminus L}(\gamma_{\infty})$ determined in Lemma \ref{estemate-engery}. This is a contradiction. 
	
	So $\mathbb{H_\infty}$ has some symplectization levels that do not consist of purely trivial cylinders over closed Reeb orbits. The portion of $\mathbb{H_\infty}$ in the symplectization levels has only one negative end, on $\gamma_{\infty}$, and must have at least two positive ends; see Figure \ref{degneration} below for an illustration. The two positive ends correspond to at least two disks at the top level because there are no contractible geodesics on $L$. At most, one of the disks intersects the hypersurface at infinity. The remaining disks have energy at least $1/n$ by the extremality of $L$. This is a contradiction because the total energy of the building is at most $\epsilon$ which is strictly less than $ \frac{1}{n}$. By the same argument, the parametric moduli space $\mathcal{M}^{\{J^t_{\infty}\}_{t\in I}}_{\mathbb{CP}^n\setminus L}(\gamma_{\infty})$ is compact whenever the parameter space $I$ is compact. Hamiltonian isotopies preserve the symplectic area class of $L$. Therefore, we conclude that
\[\# \mathcal{M}^{J_{\infty}}_{\mathbb{CP}^n\setminus L}(\gamma_{\infty})=\# \mathcal{M}^{J'_{\infty}}_{\mathbb{CP}^n\setminus L'}(\gamma_{\infty})\]
	for any two generic $J_{\infty}, J'_{\infty} $ and $L'$ Hamiltonian isotopic to $L$.
	
	For $i=1,2,\dots,n$, each of the $\mathcal{M}^{J_{\infty}}_{\mathbb{CP}^n\setminus L}(\gamma_i)$ has energy $1/n$ by Lemma \ref{estemate-engery}. Since the Lagrangian torus $L$ is extremal in $(\bar{B}^{2n}(1),\omega_{\mathrm{std}})$, this is the minimal energy a punctured holomorphic curve can have in this setting. Hence curves in $\mathcal{M}^{J_{\infty}}_{\mathbb{CP}^n\setminus L}(\gamma_i)$ cannot split into non-trivial holomorphic building.	
\end{proof}
\begin{figure}[h]
	\centering
	\includegraphics[width=8cm]{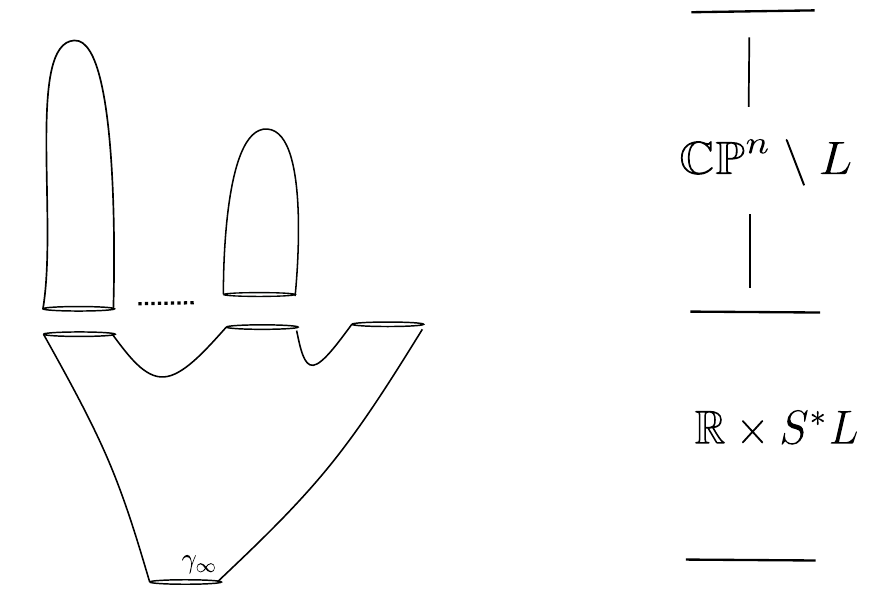}
	\caption{Degeneration in $\mathcal{M}^{J_{\infty}}_{\mathbb{CP}^n\setminus L}(\gamma_{\infty})$}\label{degneration}
\end{figure}

\subsection{The Borman--Sheridan  class entering the picture} \label{Borman-}

Let $\gamma$ denote a closed Reeb orbit of action less or equal to $1+\epsilon$ on $S^*L$. Moreover, assume it projects to a closed geodesic of Morse index $n-1$ on $L$. Let  $\tau$ be a symplectic trivialization of $TT^*L$. Define
\begin{equation*}
\mathcal{M}^{J_\infty}_{\mathbb{CP}^n\setminus L}(\gamma):=\left\{
	\begin{array}{l}
		u:\mathbb{CP}^1\setminus \{\infty\} \to (\widehat{\mathbb{CP}^n\setminus D^*L},J_{\infty}),\\
		du\circ i=J_{\infty}\circ du  ,\\
		u \text{ is asymptotic to $\gamma$ at $\infty$,} \\
		c_1^\tau(u)=1,\\
		0<\int_{u}\omega_{\mathrm{FS,\epsilon}}\leq \frac{1}{n}.
	\end{array}
	\right\}\bigg/\operatorname{Aut}(\mathbb{CP}^1,\infty).
\end{equation*}
\begin{lemma}\label{BS-moduli}
The moduli space $\mathcal{M}^{J_\infty}_{\mathbb{CP}^n\setminus L}(\gamma)$ is compact and zero dimensional. Moreover, the signed count 
\[\# \mathcal{M}^{J_\infty}_{\mathbb{CP}^n\setminus L}(\gamma)\]
is independent of the choice of the almost complex structure $J_\infty$.
\end{lemma}
\begin{proof}
	The proof follows the structure of the arguments in the proof of Theorem \ref{counting}. A non-trivial holomorphic building that can appear as a result of degeneration in the moduli space  $\mathcal{M}^{J_\infty}_{\mathbb{CP}^n\setminus L}(\gamma)$ requires a symplectic area strictly greater than $1/n$. Note that, as in the proof of Theorem \ref{counting}, any such building contains two non-constant components in the top level, one of which can be compactified to a disk with boundary on $L$. Such a disk has a symplectic area of at least $1/n$ by the extremality of $L$, leaving no symplectic area for the other component. Therefore, the moduli space  $\mathcal{M}^{J_\infty}_{\mathbb{CP}^n\setminus L}(\gamma)$ cannot undergo any nontrivial degeneration. By the same argument, the parametric moduli space $\mathcal{M}^{\{J^t_{\infty}\}_{t\in I}}_{\mathbb{CP}^n\setminus L}(\gamma)$ is compact whenever the parameter space $I$ is compact. Therefore, we conclude that for any two generic $J_{\infty}, J'_{\infty} $ we have
\[\# \mathcal{M}^{J_{\infty}}_{\mathbb{CP}^n\setminus L}(\gamma)=\# \mathcal{M}^{J'_{\infty}}_{\mathbb{CP}^n\setminus L}(\gamma).\]
One can check that the Fredholm index of the moduli space $\mathcal{M}^{J_\infty}_{\mathbb{CP}^n\setminus L}(\gamma)$ is zero. 
\end{proof}

The  Borman--Sheridan class {\cite{Tonkonog:2018aa}} of $D^*L\subset \mathbb{CP}^n$ is the symplectic cohomology class defined by
\begin{equation}\label{Borman-Sheridan}
\mathcal{BS} :=\sum_{\gamma}\,\#\mathcal{M}^{J_\infty}_{\mathbb{CP}^n\setminus L}(\gamma)\cdot \gamma \in \operatorname{SH}^0_{\mathrm{S^1,+}}(D^*L),
\end{equation}
where the sum is taken over $\gamma$ of degree zero (cf. Equation \ref{grading}).
\begin{lemma}
The Borman--Sheridan class $\mathcal{BS}$ is well-defined.
\end{lemma}
\begin{proof}
By Lemma \ref{BS-moduli}, the coefficients of $\mathcal{BS}$, defined by the signed counts
\[
\#\mathcal{M}^{J_\infty}_{\mathbb{CP}^n\setminus L}(\gamma),
\]
are well-defined and independent of the choice of almost complex structure $J_\infty$.
Moreover,  (\ref{Borman-Sheridan}) is closed with respect to the Floer differential \eqref{Floer-dif}. Indeed, its Floer differential counts the boundary points of a one-dimensional moduli space of the same type, and the signed count of such boundary points vanishes. Hence, $\mathcal{BS}$ defines a well-defined class.
\end{proof}

We note that, by Lemma \ref{estemate-engery}, the moduli spaces $\mathcal{M}^{J_\infty}_{\mathbb{CP}^n\setminus L}(\gamma_i)$ defined by (\ref{count-prob}) corresponding to the $J_\infty$-holomorphic planes $u_1,u_2,\dots,u_n, u_{\infty}$ in the top level of the building $\mathbb{H}$ in the previous section are contributing to the coefficients of Borman--Sheridan class (\ref{Borman-Sheridan}).

From Section 6.1---6.4, it follows that under neck-stretching along the boundary of a Weinstein neighborhood of $L$, the curves computing the count $\langle \psi_{n-1}p\rangle_{\mathbb{CP}^n, n+1} ^\bullet$ in Theorem \ref{countlast}  descends to a two-level holomorphic building $\mathbb{H}=(u_1,u_2,\dots,u_n, u_{\infty}, C_{\mathrm{bot}})$ with the top level consisting of $n+1$ somewhere injective rigid negatively asymptotically cylindrical $J_{\infty}$-holomorphic planes  $u_1,u_2,\dots,u_n, u_{\infty}$  in $\mathbb{CP}^n\setminus L$, and a somewhere injective and rigid asymptotically cylindrical $J_{\mathrm{bot}}$-holomorphic sphere $C_{\mathrm{bot}}$ with $n+1$ positive punctures in the bottom level $(T^*L,d\lambda_{\mathrm{can}})$. The moduli spaces of the curves $u_1,u_2,\dots,u_n, u_{\infty}$ contribute to the class $\mathcal{BS}$ defined by (\ref{Borman-Sheridan}). The moduli space containing the  curve $C_{\mathrm{bot}}$ contributes to the  linear operations $\langle \cdot|\cdot|\cdots|\cdot\rangle $ defined by (\ref{operations}).  By standard gluing results,  there is a sign preserving bijection between the count $\langle \psi_{n-1}p\rangle_{\mathbb{CP}^n, n+1} ^\bullet$  and the count of holomorphic building of this type $(u_1,u_2,\dots,u_n, u_{\infty}, C_{\mathrm{bot}})$, up to the ordering of the end asymptotics. By Theorem \ref{countlast} and gluing, we have  
\begin{equation}\label{rel}
\frac{1}{(n+1)!}	\langle \overbrace{ \mathcal{BS}|\mathcal{BS}|\dots|\mathcal{BS}}^{n+1 \text{ inputs}} \rangle =\langle \psi_{n-1}p\rangle_{\mathbb{CP}^n, n+1} ^\bullet\neq 0,
\end{equation}
where $\langle \cdot|\cdot|\cdots|\cdot\rangle $ are the linear operations defined in (\ref{operations}). This leads to the following conclusion. 
\begin{lemma}\label{count-imp}
There exist  generators $\gamma_1,\gamma_2,\dots,\gamma_n,\gamma_\infty \in \operatorname{SH}^0_{\mathrm{S^1,+}}(D^*L)$ such that
\[\langle \overbrace{ \gamma_1|\gamma_2|\dots|\gamma_n|\gamma_\infty}^{n+1 \text{ inputs}} \rangle\neq 0,\]
 and the coefficient of $\gamma_\infty$ in $\mathcal{BS}$ given by the signed count of elements in the moduli space 
\begin{equation*}
\mathcal{M}^{J_\infty}_{\mathbb{CP}^n\setminus L}(\gamma_\infty):=\left\{
	\begin{array}{l}
		u:\mathbb{CP}^1\setminus \{\infty\} \to (\widehat{\mathbb{CP}^n\setminus D^*L},J_{\infty}),\\
		du\circ i=J_{\infty}\circ du  ,\\
		u \text{ is asymptotic to $\gamma_\infty$ at $\infty$,} \\
		c_1^\tau(u)=1,\\
		\int_{u}\omega_{\mathrm{FS,\epsilon}}= \epsilon.
	\end{array}
	\right\}\bigg/\operatorname{Aut}(\mathbb{CP}^1,\infty)
\end{equation*}
is non-vanishing.
\end{lemma}
\begin{proof}
The proof is an immediate consequence of (\ref{rel}), the discussion preceding it, and the fact that all signs are positive.
\end{proof}
\subsection{Concluding the proof}\label{preparingdisk}
We specify the previous discussion to $L=T^n$ as an extremal torus. Pick an SFT-admissible almost complex structure $J$ on $\bigl(\widehat{D^*L}, \widehat{d\lambda}_{\mathrm{can}}\bigr)\cong \bigl(T^*T^n, d\lambda_{\mathrm{can}}\bigr)$, a point $p$ on the zero section $T^n$. Recall that the closed Reeb orbit $\gamma_\infty$ projects to a closed geodesic $c$ of Morse index $n-1$. Define
\begin{equation*}
	\mathcal{M}^{J}_{T^*T^n}(\gamma_\infty,T^n)\ll p\gg:=\left\{(u,t_0):
	\begin{array}{l}
		u:([0,\infty)\times S^1,i) \to (T^*T^n,J),\\
		du\circ i=J\circ du  ,\\
		u \text{ is asymptotic to $\gamma_\infty$ at $\infty$,} \\
		u(0,t)\in L \text{ for all $t\in S^1$,}\\
		u \text{ passes through $p$ at $(0,t_0)$,}
		
	\end{array}
	\right\}\bigg/\operatorname{Aut}([0,\infty)\times S^1).
\end{equation*}
Here, the marked point $t_0$ is allowed to vary on $\{0\}\times S^1$.
The Fredholm index (cf. {\cite[Section 8, Page 33]{Cieliebak:2007aa}}) of this moduli space is zero:
\[\operatorname{ind}(\mathcal{M}^{J}_{T^*T^n}(\gamma_\infty,T^n)\ll p\gg)=(n-3)(2-2)+\operatorname{CZ}^\tau(\gamma_\infty)-n+1=0.\]
Here, by Theorem \ref{special-trivalization}, we have 
\[\operatorname{CZ}^\tau(\gamma_\infty)=\mu(c)=n-1.\]


The Cieliebak--Latschev isomorphism {\cite[Section 7]{Cieliebak:2007aa}}; also see
 {\cite{Abouzaid-Viterbostheorem,Abbondandolo-Schwarz-Floerhomologyofcotangentbundles,Abbondandolo-Schwarz-Floerhomologyofcotangentbundlesloopproduct}},  computes the signed count  
 \[\#\mathcal{M}^{J}_{T^*T^n,1}(\gamma_\infty ,T^n)\ll p\gg=1\]
 to be equal to $1$. For a computation, see  Section \ref{counthalfcy}.
\begin{lemma}[{\cite[Section 7]{Cieliebak:2007aa}},  {\cite{Abouzaid-Viterbostheorem,Abbondandolo-Schwarz-Floerhomologyofcotangentbundles,Abbondandolo-Schwarz-Floerhomologyofcotangentbundlesloopproduct}}]\label{halfc2}
We have
\[\#\mathcal{M}^{J}_{T^*T^n,1}(\gamma_\infty ,T^n)\ll p\gg=1. \]
\end{lemma}
Applying gluing in the regular setting, as carried out e.g. in \cite{Pardon-Contacthomologyandvirtualfundamentalcycles}, and Lemma \ref{count-imp}, we can glue this half cylinder to $u_\infty\in \mathcal{M}^{J_\infty}_{\mathbb{CP}^n\setminus L}(\gamma_\infty)$ to get a $J_{\mathrm{glu}}$-holomorphic disk $\bar{u}_\infty:(D^2, \partial D^2)\to (\mathbb{CP}^n, L)$ with its boundary passing through the point $p$ on $L$. The symplectic area of $\bar{u}_\infty$ does not depend on the choice of representative of the relative class $[\bar{u}_\infty]\in \pi_2(\mathbb{CP}^n,L)$, therefore
\begin{equation}\label{estimate-final}
	\int \bar{u}^*_{\infty}\omega_{\mathrm{FS,\epsilon}}=\int u^*_{\infty}\omega_{\mathrm{FS,\epsilon}}= \epsilon.
\end{equation}
Recall that the Lagrangian torus $L$ under consideration is extremal in the ball $(B^{2n}(1),\omega_{\mathrm{std}})$ and $\epsilon<1/n$. So $\bar{u}_\infty$, a Maslov index two disk, has minimal\footnote{A non-trivial holomorphic building that can appear as a result of degeneration in the connected component of the moduli space containing $\bar{u}_\infty$ requires a symplectic area at least $1/n$. Note that any such building contains a non-constant disk with boundary on $L$. Such a disk has symplectic area at least $1 /n$ by the extremality of $L$.} symplectic area by (\ref{estimate-final}). Therefore, the count of disks in the connected component of the moduli space containing the disk $\bar{u}_\infty$ is well-defined. Moreover, this count is invariant under deformations of the almost complex structure $J_{\mathrm{glu}}$ that keep the hypersurface at infinity holomorphic. Also this count does not vanish because $\# \mathcal{M}^{J_{\infty}}_{\mathbb{CP}^n\setminus L}(\gamma_{\infty})\neq 0$ by Lemma \ref{count-imp} and $\#\mathcal{M}^{J}_{T^*T^n,1}(\gamma_\infty ,T^n)\ll p\gg\neq 0$ by Lemma \ref{halfc2}.  The conclusion is that the disk $\bar{u}_\infty$ survives under deformation of the almost complex structure $J_{\mathrm{glu}}$ supported near the Lagrangian $L$. 

We deform $J_{\mathrm{glu}}$ near $L$ to a new almost complex structure $J_{\mathrm{def}}$  which is $\omega_{\mathrm{FS,\epsilon}}$-compatible and agrees with  $\phi_{*}J_{\mathrm{std}}$ on $\phi(D^{2n}(\delta))$, i.e., $J_{\mathrm{def}}|_{\phi(D^{2n}(\delta))}=\phi_{*}J_{\mathrm{std}}$, where $J_{\mathrm{std}}$ is the standard complex structure of $\mathbb{C}^n$. 

 The $J_{\mathrm{def}}$-holomorphic disk $\bar{u}_\infty:(D^2, \partial D^2)\to (\mathbb{CP}^n, L)$ yields a proper $J_{\mathrm{std}}$-holomorphic map $\phi^{-1}\circ\bar{u}_{\infty}|_{\bar{u}_{\infty}^{-1}(\phi(D^{2n}(\delta)))}:\bar{u}_{\infty}^{-1}(\phi(D^{2n}(\delta)))\to D^{2n}(\delta)$. By Lemma \ref{monotoncity1} and (\ref{estimate-final}) we have
\[\epsilon \geq \int \bar{u}^*_{\infty}\omega_{\mathrm{FS,\epsilon}}\geq
	 \int_{\bar{u}_{\infty}^{-1}(\phi(D^{2n}(\delta)))} (\phi^{-1}\circ\bar{u}_{\infty})^*\omega_{\mathrm{std}}\geq \frac{1}{2}\pi  \delta^2.\]
This is a clear contradiction since we can choose $\epsilon$ arbitrarily small and keep $\delta>0$ fixed. This completes our proof. 
\section{Proof  of Theorem \ref{intersection01}}\label{proofof1.16}
We follow the same idea as in the proof of Theorem \ref{extremal-lag-ball} in Section \ref{proof for ball} but in a different setting. We only give a sketch of our proof; the details can be recovered from the proof of Theorem \ref{extremal-lag-ball}. 
\begin{enumerate}
	\item [\textbf{Step 01}] Let $L \subset (\mathbb{CP}^n,\omega_{\mathrm{FS}})$ be a monotone Lagragian torus. Suppose there exists a symplectic embedding $\Phi:B^{2n}(\frac{n}{n+1})\to \mathbb{CP}^{n}\setminus L$ and that $L$  intersects the complement of the closure of $\Phi(B^{2n}(\frac{n}{n+1}))$ in  $ \mathbb{CP}^{n}$. Pick a point $p$ on $L$ in the complement of the closure of $\Phi(B^{2n}(\frac{n}{n+1}))$. Let  $D^{2n}(\delta)$ be the closed disk of radius $\delta>0$ centered at the origin in $\mathbb{R}^{2n}$.  For sufficiently small $\delta>0$, choose a symplectic embedding $\phi:D^{2n}(\delta)\mapsto \mathbb{CP}^{n}\setminus \Phi(B^{2n}(\frac{n}{n+1}))$ such that 
	\[\phi(0)=p,\ \phi^{-1}(L)=D^{2n}(\delta)\cap \mathbb{R}^n.\]
	
	\item [\textbf{Step 02}]  By a theorem of Gromov, for generic $\omega_{\mathrm{FS}}$-compatible almost complex structure $J$ on $\mathbb{CP}^n$ and generic points $p_1, p_2 \in \mathbb{CP}^n$, there exists a $J$-holomorphic sphere in the homology class $[\mathbb{CP}^1]$ passing through $p_1$ and $ p_2$. Take $p_1$ as the point $\Phi(0)$ and $p_2$ to be a point on the Lagrangian $L$. 
	
	\item [\textbf{Step 03}]  
	Let  $0<\epsilon<1/(n+1)$. Pick a flat metric on $L$. After scaling it, we can symplectically embed the codisk bundle of radius $2$ into $ \mathbb{CP}^n\setminus \Phi(B^{2n}(\frac{n}{n+1}-\epsilon))$.  Perturb this metric according to Theorem \ref{perturb-metric} to a Riemannian metric $g$ such that with respect to $g$ every closed geodesic $\gamma$ of length less than or equal to $c=1$ is non-contractible, non-degenerate (as a critical point of the energy functional) and satisfies 
	\[0\leq \operatorname{\mu}(\gamma)\leq n-1,\]
	where $\operatorname{\mu}(\gamma)$ denotes the Morse index of $\gamma$. Since $g$ can be chosen to be a small perturbation of the flat metric, we can ensure that the unit codisk bundle $D^*_1L$ with respect to $g$ still symplectically embeds into $ \mathbb{CP}^n\setminus \Phi(B^{2n}(\frac{n}{n+1}-\epsilon))$.

	\item [\textbf{Step 04}] Take a neck-stretching family of almost complex structures $J_k$ on $\mathbb{CP}^n$ along the contact type hypersurface $S^*L$, the boundary of  $D^*_1L$ in $\mathbb{CP}^n$. We assume that $J_k=\Phi_*J_{\mathrm{std}}$ on $\Phi(B^{2n}(\frac{n}{n+1}-\epsilon)$ for all $k$, where  $\Phi_*J_{\mathrm{std}}$ denotes the push-forward of the standard complex structure on $\mathbb{C}^n$. Let $J_\infty$ be the almost complex structure on $\mathbb{CP}^n\setminus L$ obtained as the limit of $J_k$.
	
	\item [\textbf{Step 05}] As $k\to \infty$, the sequence of $J_k$-holomorhpic spheres in \textbf{Step 02} breaks to a holomorphic building with its top level in  $\mathbb{CP}^n\setminus L$,  bottom level in $T^*L$, and some symplectization levels in $\mathbb{R}\times S^*L$. Next, we explain in Step 06---08 that the top level of the building in $\mathbb{CP}^n\setminus L$ consists of two somewhere injective negatively asymptotically cylindrical $J_{\infty}$-holomorphic planes, denoted by $D_1$ and $D_2$, one of which passes through the point $\Phi(0)$. The bottom level consists of a smooth cylinder passing through the point $p_2$. There are no symplectization levels. Moreover, if  $D_1$ is the plane passing through $\Phi(0)$, then 
	\begin{equation}\label{need}
	\int D_1^*\omega_{\mathrm{FS}}=\frac{n}{n+1}, \ \int D_2^*\omega_{\mathrm{FS}}=\frac{1}{n+1}.
	\end{equation} 
	  For an illustration, see Figure \ref{holo}.
	\begin{figure}[h]
	 			\centering
	 			\includegraphics[width=6cm]{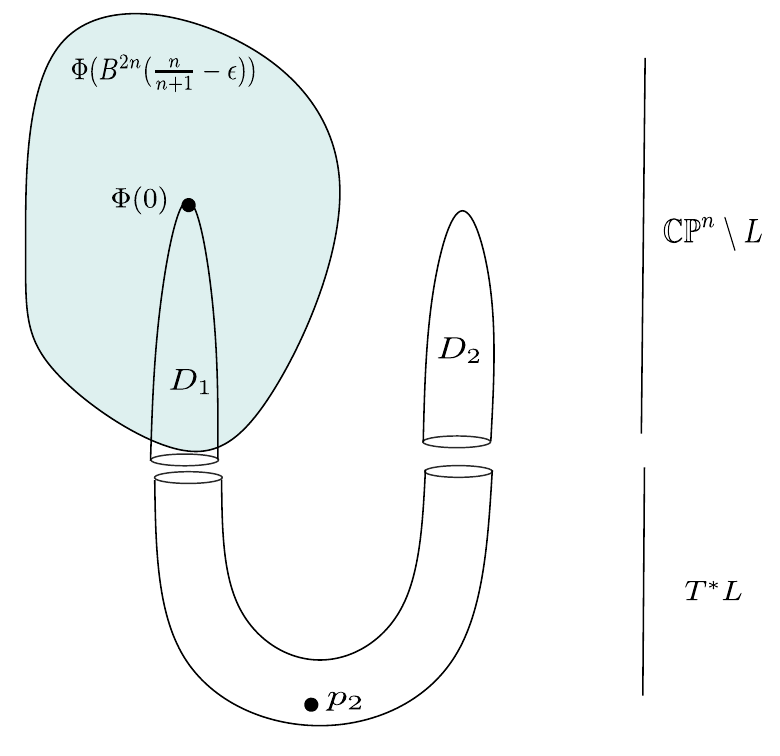}
	 			\caption{The holomorphic building}\label{holo}
 \end{figure}
	\item [\textbf{Step 06}] There is a connected smooth component $C_{\mathrm{bot}}$ in $T^*L$  of the resultant building  that passes through the point $p_2$ and has positive punctures asymptotic to closed Reeb orbits on $S^*L$. Let $C^s_{\mathrm{bot}}$ be the underlying simple punctured curve of $C_{\mathrm{bot}}$. Assume $C^s_{\mathrm{bot}}$ has $l$ positive punctures asymptotic to the closed Reeb orbits $\gamma_{1},\gamma_{2},\dots, \gamma_{l}$ that project to the closed geodesics $c_1,\dots, c_l$ repsectivly. Generically, the Fredholm index of $C^s_{\mathrm{bot}}$ (taking the point constraint at $p_2$ into account) satisfies  
	\[0\leq \operatorname{ind}(C^s_{\mathrm{bot}})=(n-3)(2-l)+\sum_{i=1}^{l}\operatorname{\mu}(c_i)-2n+2\leq 2(l-2),\] 
	where $\operatorname{\mu}(c_i)$ denotes the Morse index of $c_i$, which by \textbf{Step 03} satisfies $0\leq \operatorname{\mu}(c_i)\leq n-1$.
This implies that the curve  $C^s_{\mathrm{bot}}$, and hence 	$C_{\mathrm{bot}}$, must have at least two positive punctures. By the same argument as in the proof of Lemma \ref{count-disks}, the positive punctures of $C_{\mathrm{bot}}$ give rise to planer components of the building in the top level $\mathbb{CP}^n\setminus L$. Thus, there are at least two negatively asymptotically cylindrical $J_{\infty}$-holomorphic planes in the top level  $\mathbb{CP}^n\setminus L$.
	\item [\textbf{Step 07}] Let $D_1, D_2,\dots, D_m$ be the smooth components of the building in the top level $\mathbb{CP}^n\setminus L$. The energy of the whole building is $1$, i.e., 
	\[\sum_{j=1}^{m}\int D_j^*\omega_{\mathrm{FS}}=1.\]

	\item [\textbf{Step 08}] We have chosen the point $p_1=\Phi(0)$ in the complement of the Weinstein neighborhood $D^*_1L$. So there is a component, say $D_1$, in the top level $\mathbb{CP}^n\setminus L$ that passes through the point  $\Phi(0)$, the center of the embedded ball $\Phi(B^{2n}(\frac{n}{n+1}-\epsilon))$. Moreover, the boundary $\partial D_1$ maps to $L$. Therefore, by monotonicity 
	\[\frac{n}{n+1}-\epsilon\leq \int_{D_1^{-1}(\Phi(B^{2n}(\frac{n}{n+1})))} D_1^*\omega_{\mathrm{std}}\leq \int D_1^*\omega_{\mathrm{FS}}.\]
	On the other hand, $D_1, D_2,\dots, D_m$ can be compactified to surfaces in $\mathbb{CP}^n$ with boundaries on $L$. Since $L$ is monotone (hence extremal by {\cite[Corollary 1.7]{Cieliebak2018}}), we have
	\[\frac{1}{n+1}\leq \int D_j^*\omega_{\mathrm{FS}}\]
	for every $j\in \{1,2,\dots, m\}.$ Combining the above, we get
	\[\frac{n}{n+1}-\epsilon +\frac{m-1}{n+1}\leq \sum_{i=1}^{m} \int D_j^*\omega_{\mathrm{FS}}=1.\] 
	By \textbf{Step 03}, we have $\epsilon<1/(n+1)$. So the last estimate implies that $1\leq m\leq 2$ and both $D_1, D_2$ must be somewhere injective. On the other hand, by \textbf{Step 06}, at least two planer components are in the top level. Putting this together implies that the building consists of two somewhere injective negatively asymptotically cylindrical $J_{\infty}$-holomorphic planes $D_1,D_2$ in the top level  $\mathbb{CP}^n\setminus L$. The bottom level of the building consists of a smooth cylinder passing through the point $p_2$. There are no symplectization levels.   For an illustration, see Figure \ref{holo}.
	
	Note that $L$ is monotone, so the symplectic area of any disk with boundary on $L$ takes  values in the integer multiples of $1/(n+1)$. This, together with the above estimates, implies that
	\begin{equation}\label{need}
	\int D_1^*\omega_{\mathrm{FS}}=\frac{n}{n+1}, \int D_1^*\omega_{\mathrm{FS}}=\frac{1}{n+1}.
	\end{equation}  
	In particular, by monotonicity of $L$, $D_1$ is of Maslov index $2n$, and $D_2$ is of Maslov index $2$.
 \item[\textbf{Step 09}]
 Let $\gamma$ denote a closed Reeb orbit of action less than or equal to $1$ on $S^*L$. Moreover, assume it projects to a closed geodesic of Morse index $n-1$ on $L$. Let  $\tau$ be a symplectic trivialization of $TT^*L$. We associate to each such $\gamma$ two types of moduli spaces defined below:
\begin{equation}\label{type1}
\mathcal{M}^{J_\infty}_{\mathbb{CP}^n\setminus L}(\gamma)\ll\Phi(0)\gg:=\left\{
	\begin{array}{l}
		u:\mathbb{CP}^1\setminus \{\infty\} \to (\widehat{\mathbb{CP}^n\setminus D^*L},J_{\infty}),\\
		du\circ i=J_{\infty}\circ du  ,\\
		u \text{ is asymptotic to $\gamma$ at $\infty$,} \\
		u(0)=\Phi(0),\\
		c_1^\tau(u)=n.
	\end{array}
	\right\}\bigg/\operatorname{Aut}(\mathbb{CP}^1,\infty, 0),
\end{equation}
and 
\begin{equation}\label{type2}
\mathcal{M}^{J_\infty}_{\mathbb{CP}^n\setminus L}(\gamma):=\left\{
	\begin{array}{l}
		u:\mathbb{CP}^1\setminus \{\infty\} \to (\widehat{\mathbb{CP}^n\setminus D^*L},J_{\infty}),\\
		du\circ i=J_{\infty}\circ du  ,\\
		u \text{ is asymptotic to $\gamma$ at $\infty$,} \\
		c_1^\tau(u)=1.
	\end{array}
	\right\}\bigg/\operatorname{Aut}(\mathbb{CP}^1,\infty).
\end{equation}
Both these moduli spaces consist of somewhere injective curves. To see this, let $u\in \mathcal{M}^{J_\infty}_{\mathbb{CP}^n\setminus L}(\gamma)\ll\Phi(0)\gg$ be a degree $d$ cover of a somewhere injective curve $\bar{u}$. The symplectic area of $\bar{u}$ is at least $n/(n+1)$ because it satisfies the constraint $\ll\Phi(0)\gg$ and $L$ is monotone. This forces $d=1$.

These moduli spaces are zero-dimensional manifolds for generic $J_\infty$. We argue that the signed counts $\#\mathcal{M}^{J_\infty}_{\mathbb{CP}^n\setminus L}(\gamma)\ll\Phi(0)\gg$ and  $\#\mathcal{M}^{J_\infty}_{\mathbb{CP}^n\setminus L}(\gamma)$ are well-defined and do not depend on the choice of generic  $J_\infty$. First, we explain that $\mathcal{M}^{J_\infty}_{\mathbb{CP}^n\setminus L}(\gamma)\ll\Phi(0)\gg$ is compact for generic $J_\infty$.  Assume there is a splitting in this moduli space to a non-trivial building. Following the proof of Theorem \ref{counting}, the building must have at least two smooth connected components $u_1, u_2$ in the top level $\mathbb{CP}^n\setminus L$. One of these, say $u_1$, must pass through the point $\Phi(0)$. By monotonicity of the symplectic area of holomorphic curves, $u_1$ must have a symplectic area at least $n/(n+1)-\epsilon$. Also, after compactifying the component $u_2$ to a surface with boundary on $L$, the symplectic area of $u_2$ must be a positive integer multiple of $1/(n+1)$ by the monotonicity of $L$. The total area of $u_1$ and $u_2$ is at least $1-\epsilon$, where $\epsilon<1/(n+1)$ by \textbf{Step 02}. This is greater than the symplectic area of  $\mathcal{M}^{J_\infty}_{\mathbb{CP}^n\setminus L}(\gamma)\ll\Phi(0)\gg$, which is $n/(n+1)$, a contradiction. By the same arguments, the signed count of elements in $\mathcal{M}^{J_\infty}_{\mathbb{CP}^n\setminus L,D_1}\ll\Phi(0)\gg$ does not depend on the choice of generic $J_\infty$. One adopts the same arguments to show that  $\#\mathcal{M}^{J_\infty}_{\mathbb{CP}^n\setminus L}(\gamma)$ is well-defined and does not depend on the choice of generic  $J_\infty$.

These moduli spaces compute a well-defined class given by 
\[\mathcal{BS} :=\mathcal{BS}_1 + \mathcal{BS}_2\in \operatorname{SH}^0_{\mathrm{S^1,+}}(D^*L),\]   
where 
\[\mathcal{BS}_1:=\sum_{\gamma}\#\mathcal{M}^{J_\infty}_{\mathbb{CP}^n\setminus L}(\gamma)\ll\Phi(0)\gg \gamma \textit{ and } \mathcal{BS}_2:=\sum_{\gamma}\#\mathcal{M}^{J_\infty}_{\mathbb{CP}^n\setminus L}(\gamma)\cdot \gamma, \]
where both sums are taken over $\gamma$ of degree zero (cf. Equation \ref{grading}).

The somewhere injective negatively asymptotically cylindrical $J_{\infty}$-holomorphic planes $D_1$ and $D_2$ in \textbf{Step 08} are asymptotic to Reebs orbits that project to geodesics of Morse index $n-1$. The connected component of the moduli space containing $D_1$, denoted by $\mathcal{M}^{J_\infty}_{\mathbb{CP}^n\setminus L,D_1}\ll\Phi(0)\gg$, is of type (\ref{type1}). The connected component of the moduli space  containing $D_2$, denoted by $\mathcal{M}^{J_\infty}_{\mathbb{CP}^n\setminus L,D_1}$, is of type (\ref{type2}).  In particular, both are zero-dimensional manifolds, and the signed counts  $\#\mathcal{M}^{J_\infty,s}_{\mathbb{CP}^n\setminus L,D_1}$ and $\#\mathcal{M}^{J_\infty,s}_{\mathbb{CP}^n\setminus L,D_1}$ are well-defined and do not depend on the choice of generic $J_\infty$.  The moduli spaces $\mathcal{M}^{J_\infty}_{\mathbb{CP}^n\setminus L,D_1}$ and $\mathcal{M}^{J_\infty}_{\mathbb{CP}^n\setminus L,D_2}$ are contributing to $\mathcal{BS}_1$ and $\mathcal{BS}_2$, respectively. 
  \item [\textbf{Step 10}] The count of closed degree $1$ curves passing through two generic points in $\mathbb{CP}^n$ does not vanish. By the above steps, this count descends to a two-level holomorphic building $(D_1,D_2, C_{\mathrm{bot}})$ under stretching along a Weinstein neighborhood of $L$. The disks $D_1$ and $D_2$ are contributing to the class $\mathcal{BS}$, and the cylinder $C_{\mathrm{bot}}$ is contributing to the linear operations defined by (\ref{operations}). By standard gluing,  there is a sign preserving bijection between the count of closed degree one curves passing through two generic points in $\mathbb{CP}^n$ and the holomorphic buildings of the form $(D_1,D_2, C_{\mathrm{bot}})$. This implies that $\langle \mathcal{BS}| \mathcal{BS} \rangle$ does not vanish, i.e.,
 \begin{equation}\label{BS_1-notzero}
 	\langle \mathcal{BS}_1| \mathcal{BS}_1 \rangle +\langle \mathcal{BS}_1| \mathcal{BS}_2 \rangle+	\langle \mathcal{BS}_2| \mathcal{BS}_2 \rangle\neq 0.
 \end{equation} 
We must have $
\langle \mathcal{BS}_2 \mid \mathcal{BS}_2 \rangle = 0$. 
Indeed, if this pairing were nonzero, then a gluing argument would imply
a nonzero count of rigid $J$-holomorphic spheres in $\mathbb{CP}^n$
representing the line class $[\mathbb{CP}^1]$ and passing through a
generic point, which is wrong. 

Thus,  it follows from $
\langle \mathcal{BS}_2 \mid \mathcal{BS}_2 \rangle = 0$ and (\ref{BS_1-notzero}) that $\mathcal{BS}_1\neq 0$. So there exists a generator $\gamma \in \operatorname{SH}^0_{\mathrm{S^1,+}}(D^*L)$  such that  $\mathcal{M}^{J_\infty}_{\mathbb{CP}^n\setminus L}(\gamma)\ll\Phi(0)\gg$ has a non-vanishing signed count. Following Section \ref{preparingdisk}, one can produce a disk $\bar{D}_1:(D^2,\partial D^2)\mapsto (\mathbb{CP}^n, L)$ whose boundary passes through the point $p\in L$, where $p$ is the point we have chosen in \textbf{Step 01}. This disk has symplectic area equal to $n/(n+1)$ which by monotonicity of $L$ implies that it has Maslov index $2n$.  Moreover, we can assume that $\bar{D}_1$ is $J$-holomorphic for some $\omega_{\mathrm{FS}}$-compatible almost complex structure $J$ that agrees with  $\phi_{*}J_{\mathrm{std}}$ on $\phi(D^{2n}(\delta))$, where $J_{\mathrm{std}}$ is the standard complex structure of $\mathbb{C}^n$ , i.e., $J|_{\phi(D^{2n}(\delta))}=\phi_{*}J_{\mathrm{std}}$.
 The $J$-holomorphic disk $\bar{D}_1:(D^2, \partial D^2)\to (\mathbb{CP}^n, L)$ yields a proper $\phi_{*}J_{\mathrm{std}}$-holomorphic map $\bar{D}_1|_{\bar{D}_1^{-1}(\phi(D^{2n}(\delta)))}:\bar{D}_1^{-1}(\phi(D^{2n}(\delta)))\mapsto \phi(D^{2n}(\delta))$. By Lemma \ref{monotoncity1} and (\ref{need}), we have
 	\[\frac{n}{n+1}=	\int \bar{D}^*_1\omega_{\mathrm{FS}}\geq \int_{\bar{D}_1^{-1}(\phi(D^{2n}(\delta)))} \bar{D}^*_{1}\omega_{\mathrm{std}}+ \int_{\bar{D}_1^{-1}(\Phi(B^{2n}(\frac{n}{n+1}-\epsilon)))} \bar{D}^*_{1}\omega_{\mathrm{std}}\geq \frac{1}{2}\pi  \delta^2+ \frac{n}{n+1}-\epsilon.\]
 	This is a contradiction because we can choose $\epsilon$ arbitrarily small while keeping $\delta$ fixed. This completes the sketch of our proof.
\end{enumerate}
\section{Proof sketch of Theorem \ref{extremal-lag-cylinder}}\label{proofof1.8}
\subsection{Proof sketch part I}
In this part, we explain that, for any $k\in \mathbb{Z}_{\geq 1}$, we have
\[\mathrm{c}_{L}(B^{2k}(1)\times \mathbb{C}^m,\omega_{\mathrm{std}})=\mathrm{c}_{L}(B^{2k}(1),\omega_{\mathrm{std}}).\] 
We give a sketch of our proof; the details can be recovered from  Section \ref{Analyze} in the proof of Theorem \ref{extremal-lag-ball}.
\begin{enumerate}
\item [\textbf{Step 01}] Let $L\subset\operatorname{int}(B^{2k}(1))\times \mathbb{C}^m$ be a Lagrangian torus. After compactifying the ball $B^{2k}(1)$ to $\mathbb{CP}^k$ and cutting $\mathbb{C}^{m}$ by a large lattice,  we can assume that $L$ is a Lagrangian torus in $(\mathbb{CP}^k\times T^{2m}, \omega_{\mathrm{FS}}\oplus \omega_{\mathrm{std}})$ that lies in the complement of the hypersurface $\mathbb{CP}^{k-1}\times T^{2m}$. Here $\omega_{\mathrm{FS}}$ integrates to $1$ on lines in $\mathbb{CP}^k$.
\item [\textbf{Step 02}] Set $\Omega:=\omega_{\mathrm{FS}}\oplus \omega_{\mathrm{std}}$. By Corollary \ref{simplecount}, the count of $J$-holomorphic spheres in $\mathbb{CP}^k\times T^{2m}$ representing  the homology class $[\mathbb{CP}^1\times\{*\}]$ and satisfying the constraint $\ll \mathcal{T}_D^{k-1}q\gg$ at a generic point $q\in  \mathbb{CP}^k\times T^{2m}$ is non-zero. Therefor, for a generic $\Omega$-compatible almost complex structure $J$ on $\mathbb{CP}^k\times T^{2m}$ and generic $q\in \mathbb{CP}^k\times T^{2m}$, there exists a $J$-holomorphic sphere $u$ in the homology class $[\mathbb{CP}^1\times\{*\}]$ satisfying the constraint $\ll \mathcal{T}_D^{k-1}q\gg$ (cf. Definition \ref{tangdef}). 
	 
\item [\textbf{Step 03}] Pick a flat metric on $L$. After scaling it, we can symplectically embed the codisk bundle of radius $2$ into $\operatorname{int}(B^{2k}(1))\times T^{2m} = \mathbb{CP}^k\times T^{2m}\setminus \mathbb{CP}^{k-1} \times T^{2m}$. Perturb this metric according to Theorem \ref{perturb-metric} to a Riemannian metric $g$ such that, with respect to $g$, every closed geodesic $\gamma$ of length less than or equal to $c=1$ is non-contractible, non-degenerate (as a critical point of the energy functional) and satisfies 
	\[0\leq \operatorname{\mu}(\gamma)\leq n-1,\]
	where $\operatorname{\mu}(\gamma)$ denotes the Morse index of $\gamma$. Since $g$ can be chosen to be a small perturbation of the flat metric, we can ensure that the unit codisk bundle $D^*L$ with respect to $g$ still symplectically embeds into $\operatorname{int}(B^{2k}(1))\times T^{2m} = \mathbb{CP}^k\times T^{2m}\setminus \mathbb{CP}^{k-1} \times T^{2m}$.
	
\item [\textbf{Step 04}] Take a family of $\Omega$-compatible almost complex structures $J_n$ on $\mathbb{CP}^k\times T^{2m}$ that stretches the neck along the contact type hypersurface $S^*L:=\partial D^*L $ in $\mathbb{CP}^k\times T^{2m}$. We assume that $J_n$ restricted to a small neighborhood of the hypersurface $\mathbb{CP}^{(k-1)}\times T^{2m}$ at infinity is the standard complex structure $J_{\mathrm{std}}\oplus J_{\mathrm{std}}$, so that the hypersurface $\mathbb{CP}^{k-1}\times T^{2m}$ is $J_n$-holomorphic for all $n$. Let $J_\infty$ be the almost complex structure on $\mathbb{CP}^k\times T^{2m}\setminus L$ obtained as the limit of $J_n$. Then the hypersurface $\mathbb{CP}^{k-1}\times T^{2m}$ is $J_\infty$-holomorphic.
	
\item [\textbf{Step 05}] Choose a point $q$ on $L$ and a local symplectic divisor containing $q$. As $n\to \infty$, the sequence of $J_n$-holomorphic spheres in \textbf{Step 02} breaks into a holomorphic building $\mathbb{H}$ with its top level in  $\mathbb{CP}^k\times T^{2m}\setminus L$,  bottom level in $T^*L$, and some symplectization levels in $\mathbb{R}\times S^*L$.

	\item [\textbf{Step 06}] 
	There is a non-constant smooth punctured sphere $C_{\mathrm{bot}}$ in the bottom level of $\mathbb{H}$ in  $T^*L$ that inherits  the tangency constraint $\ll \mathcal{T}_D^{k-1}q\gg$ at $q$. As in Lemma \ref{countendsincotangentbundle},  the curve $C_{\mathrm{bot}}$ has at least $k+1$ positive asymptotically cylindrical ends. As in the proof of Lemma \ref{count-disks}, this means there are at least $k+1$ $J_\infty$-holomorphic disks, denoted by $D_1, D_2,\dots, D_{k+1}$, in the top level $\mathbb{CP}^k\times T^{2m}\setminus L$. Since the building is the limit of holomorphic spheres intersecting the $J_\infty$-holomorphic hypersurface $\mathbb{CP}^{k-1}\times T^{2m}$ with intersection number $+1$, this means at least $k$ disks, say  $D_1, D_2,\dots, D_{k}$, lie in complement $\operatorname{int}(B^{2k}(1))\times T^{2m}$ of $\mathbb{CP}^{k-1}\times T^{2m}$, since otherwise the total intersection number would be larger than $+1$. Here we use the fact that distinct holomorphic objects intersect positively.
	\item [\textbf{Step 07}]  The energy of the whole building is $1$. So
	\[\sum_{i=1}^k\int_{D^2}D_i^*\omega_{\mathrm{std}}\leq 1.\]
	This means that, for some disk $D_j$, we have
	\[\int_{D^2}D_j^*\omega_{\mathrm{std}}\leq \frac{1}{k}\sum_{i=1}^k\int_{D^2}D_i^*\omega_{\mathrm{std}}\leq \frac{1}{k}.\]
We conclude that
	\[\mathrm{c}_{L}(B^{2k}(1)\times \mathbb{C}^m,\omega_{\mathrm{std}})\leq \frac{1}{k}.\]
	On the other hand, the Lagrangian torus $S^1(1/k)\times \cdots\times S^1(1/k)\subset \bar{B}^{2k}(1)\times \mathbb{C}^m$ has symplectic area equal to $1/k$. Thus
\[\mathrm{c}_{L}(B^{2k}(1)\times \mathbb{C}^m,\omega_{\mathrm{std}})= \frac{1}{k}=\mathrm{c}_{L}(B^{2k}(1),\omega_{\mathrm{std}}).\] \end{enumerate}
\subsection{Proof sketch part II}
In this part, we explain that every extremal Lagrangian torus $L$ in $(B^{2k}(1)\times \mathbb{C}^m,\omega_{\mathrm{std}})$ is entirely contained in the boundary $\partial B^{2k}(1)\times \mathbb{C}^m$. The proof follows step by step the proof of Theorem \ref{extremal-lag-ball}, so we give a sketch only.
\begin{enumerate}
\item [\textbf{Step 01}] Let $L\subset B^{2k}(1)\times \mathbb{C}^m$ be an extremal Lagrangian torus that is not entirely contained in the boundary $\partial B^{2k}(1)\times \mathbb{C}^m$.  We aim to arrive at a contradiction.  Fix a point $p\in L\cap \operatorname{int}(B^{2n}(1))\times \mathbb{C}^m$. 
Let  $D^{2n}(\delta)$ be the closed disk of radius $\delta>0$ centered at the origin in $\mathbb{R}^{2n}$. For sufficiently small $\delta>0$, choose a symplectic embedding $\phi:D^{2n}(\delta)\to \operatorname{int}(B^{2n}(1))\times \mathbb{C}^m$ such that 
\[\phi(0)=p,\ \phi^{-1}(L)=D^{2n}(\delta)\cap \mathbb{R}^n.\]
 \item [\textbf{Step 02}] Let $\epsilon>0$, extend the cylinder $B^{2k}(1)\times \mathbb{C}^m$ to $B^{2k}(1+\epsilon)\times \mathbb{C}^m$. After compactifying  $B^{2k}(1+\epsilon)$ to $(\mathbb{CP}^k, (1+\epsilon)\omega_{\mathrm{FS}})$ and cutting $\mathbb{C}^{m}$ by a large lattice,  we can assume that $L$ is a Lagrangian torus in $(\mathbb{CP}^k\times T^{2m}, \Omega_{\epsilon})$, where $\Omega_{\epsilon}:=(1+\epsilon)\omega_{\mathrm{FS}}\oplus \omega_{\mathrm{std}}$. Moreover, $L$ lies in the complement of the hypersurface $\mathbb{CP}^{k-1}\times T^{2m}$. Pick a flat metric on $L$. After scaling it, we can symplectically embed the codisk bundle of radius $2$ into $B^{2k}(1+\epsilon/2)\times  T^{2m}$. Perturb this metric according to Theorem \ref{perturb-metric} to a Riemannian metric $g$ such that, with respect to $g$, every closed geodesic $\gamma$ of length less than or equal to $c=1+\epsilon$ is non-contractible, non-degenerate (as a critical point of the energy functional) and satisfies 
	\[0\leq \operatorname{\mu}(\gamma)\leq n-1,\]
	where $\operatorname{\mu}(\gamma)$ denotes the Morse index of $\gamma$. Since $g$ can be chosen to be a small perturbation of the flat metric, we can ensure that the unit codisk bundle $D^*L$ with respect to $g$ still symplectically embeds into $B^{2k}(1+\epsilon/2)\times  T^{2m} $.
For an illustration of this geometric setup, see Figure \ref{geometric-setup-cylinder}.
\begin{figure}[h]
	\centering
	\includegraphics[width=8cm]{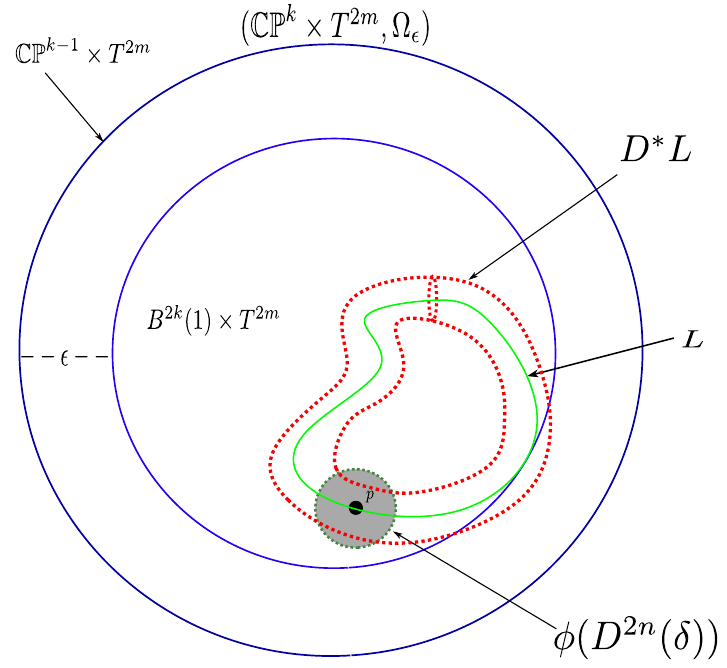}
	\caption{Geometric Setup}\label{geometric-setup-cylinder}
\end{figure}

\item [\textbf{Step 03}] By Corollary \ref{simplecount}, for a generic $\Omega_{\epsilon}$-compatible almost complex structure $J$ on $\mathbb{CP}^k\times T^{2m}$ and a generic $q\in \mathbb{CP}^k\times T^{2m}$, there exists a $J$-holomorphic sphere $u$ in the homology class $[\mathbb{CP}^1\times\{*\}]$ satisfying the contstraint $\ll \mathcal{T}_D^{k-1}q\gg$ (cf. Definition \ref{tangdef}). By Theorem \ref{countlast}, we assume that each curve $u$ carries a Hamiltonian perturbation around the point $q$ as described in (\ref{hamilto-perturb}). 

\item [\textbf{Step 04}] Take an $\Omega_{\epsilon}$-compatible family of almost complex structures $J_n$ on $\mathbb{CP}^k\times T^{2m}$ that stretch the neck along the contact type hypersurface $S^*L:=\partial D^*L$ in $\mathbb{CP}^k\times T^{2m}$. We assume that $J_n$ restricted to a small neighborhood of the hypersurface $\mathbb{CP}^{k-1}\times T^{2m}$ at infinity is the standard complex structure $J_{\mathrm{std}}\oplus J_{\mathrm{std}}$. so that the hypersurface $\mathbb{CP}^{k-1}\times T^{2m}$ is $J_n$-holomorphic for all $n$. Let $J_\infty$ be the almost complex structure on $\mathbb{CP}^k\times T^{2m}\setminus L$ obtained as the limit of $J_n$. Then the hypersurface $\mathbb{CP}^{k-1}\times T^{2m}$ is $J_\infty$-holomorphic.
	
	\item [\textbf{Step 05}] Choose a point $q$ on $L$ and a local symplectic divisor containing $q$. As $n\to \infty$, the sequence of $J_n$-holomorhpic spheres in \textbf{Step 02} breaks into a holomorphic building $\mathbb{H}$ with its top level in  $\mathbb{CP}^k\times T^{2m}\setminus L$,  bottom level in $T^*L$, and some symplectization levels in $\mathbb{R}\times S^*L$.

	\item [\textbf{Step 06}] From the analysis in Section $\ref{Analyze}$  it follows that $\mathbb{H}$ is a two-level holomorphic building with the top level consisting of $k+1$ somewhere injective rigid negatively asymptotically cylindrical $J_{\infty}$-holomorphic planes  $u_1,u_2,\dots,u_k, u_{\infty}$  in $\mathbb{CP}^k\times T^{2m}\setminus L$, and a somewhere injective and rigid asymptotically cylindrical $J_{\mathrm{bot}}$-holomorphic sphere $C_{\mathrm{bot}}$ with $k+1$ positive punctures in the bottom level $(T^*L,d\lambda_{\mathrm{can}})$. Only one of these planes, denoted by $u_\infty$, in the top level, intersects the complex hypersurface $\mathbb{CP}^k\times T^{2m}$. Moreover,  we have
\[\int u^*_{i}\omega_{\mathrm{std}}=\frac{1}{k}, \]
for all $i=1,2,\dots,k$, and 
	\[\int u^*_{\infty}\Omega_{\epsilon}= \epsilon. \]
	\item [\textbf{Step 07}]  From Section \ref{Borman-} it follows that there exists a generator $\gamma_\infty \in \operatorname{SH}^0_{\mathrm{S^1,+}}(D^*L)$ such that its coefficient in the Borman--Sheridan class $\mathcal{BS}$ given by the signed count of elements in the moduli space 
\begin{equation*}
\mathcal{M}^{J_\infty}_{\mathbb{CP}^k\times T^{2m}\setminus L}(\gamma_\infty):=\left\{
	\begin{array}{l}
		u:\mathbb{CP}^1\setminus \{\infty\} \to (\widehat{\mathbb{CP}^k\times T^{2m}\setminus D^*L},J_{\infty}),\\
		du\circ i=J_{\infty}\circ du  ,\\
		u \text{ is asymptotic to $\gamma_\infty$ at $\infty$,} \\
		c_1^\tau(u)=1,\\
		\int_{u}\Omega_{\epsilon}= \epsilon.
	\end{array}
	\right\}\bigg/\operatorname{Aut}(\mathbb{CP}^1,\infty)
\end{equation*}
is non-vanishing. Following Section \ref{preparingdisk}, one gets a disk $\bar{u}_\infty:(D^2,\partial D^2)\to (\mathbb{CP}^k\times T^{2m}, L)$ whose boundary passes through the point $p\in L$, where $p$ is the point we have chosen in \textbf{Step 01}. Moreover, we can assume that $\bar{u}_\infty$ is $J$-holomorphic for some $\Omega_\epsilon$-compatible almost complex structure $J$ that agrees with  $\phi_{*}J_{\mathrm{std}}$ on $\phi(D^{2n}(\delta))$, where $J_{\mathrm{std}}$ is the standard complex structure of $\mathbb{C}^n$ , i.e., $J|_{\phi(D^{2n}(\delta))}=\phi_{*}J_{\mathrm{std}}$.
 The $J$-holomorphic disk $\bar{u}_\infty:(D^2, \partial D^2)\to (\mathbb{CP}^k\times T^{2m}, L)$ yields a proper $J_{\mathrm{std}}$-holomorphic map $\phi^{-1}\circ\bar{u}_{\infty}|_{\bar{u}_{\infty}^{-1}(\phi(D^{2n}(\delta)))}:\bar{u}_{\infty}^{-1}(\phi(D^{2n}(\delta)))\to D^{2n}(\delta)$. By Lemma \ref{monotoncity1}, we have
\[\epsilon \geq \int \bar{u}^*_{\infty}\Omega_{\epsilon}\geq
	 \int_{\bar{u}_{\infty}^{-1}(\phi(D^{2n}(\delta)))} (\phi^{-1}\circ\bar{u}_{\infty})^*\omega_{\mathrm{std}}\geq \frac{1}{2}\pi  \delta^2.\]
This is a contradiction since we can choose $\epsilon$ arbitrarily small and keep $\delta$ fixed, and it completes the sketch of our proof. 
\end{enumerate}
\section{Proof of Theorem \ref{extremal-lag-toric}}\label{sectiontoricproof}
\subsection{Preparing a geometric framework}

Let $(X^4_{\Omega},\omega_{\mathrm{std}})$ be a four-dimensional toric domain such that there exists an ellipsoid $E^{4}(a,b)$ satisfying  $X^4_{\Omega} \subseteqq E^{4}(a,b)$ and $\operatorname{diagonal }(X^4_{\Omega})=\operatorname{diagonal }(E^{4}(a,b))$.
 Suppose $L\subset (X_{\Omega}^4,\omega_{\mathrm{std}})$ is an extremal Lagrangian torus that intersects the interior of $X_{\Omega}^4$. We aim to arrive at a contradiction using the monotonicity property of the area of pseudoholomorphic curves (cf. Lemma \ref{monotoncity1}).

Enclose $X_{\Omega}^4$ in a four-dimensional ellipsoid $E^4(a,b)$ of the same diagonal, where $0<a\leq b$. We recall that the $k$-th Gutt--Hutchings capacity of the ellipsoid $E^4(a,b)$, denoted by $c_{k}^{\mathrm{GH}}(E^4(a,b))$ and defined in (\ref{Gutt--Hutchingscapacities}), is the $k$-th term in the sequence of positive integer multiples of $a$ and $b$ arranged in nondecreasing order with repetition {\cite[Example 1.8]{MR3868228}}. By density, we can assume that the ellipsoid $E^4(a,b)$ is a rational ellipsoid. So there exists\footnote{By the scaling property of the Gutt--Hutchings capacity, it is enough to prove this for rational ellipsoids of the form $E^4(1,b)$. We can write $b=p/q$ for some coprime integers $p$ and $q$, and we take $k=p+q$.} $k\in \mathbb{Z}_{\geq 1}$ such that 
\[c_{k}^{\mathrm{GH}}(E^4(a,b))=k \operatorname{diagonal}(E^4(a,b))=k \operatorname{diagonal}(X_{\Omega}^4). \]
Extend the ellipsoid $E^4(a,b)$ to the ellipsoid $E^{\mathrm{ext}}:=E(a(1+\epsilon),b(1+\epsilon))$ for some fixed  $\epsilon>0$. The latter is of diagonal equal to $(1+\epsilon)\operatorname{diagonal}(X_{\Omega}^4)$, and  
\begin{equation}\label{GH-cap}
	c_{k}^{\mathrm{GH}}(E^{\mathrm{ext}})=k(1+\epsilon) \operatorname{diagonal}(X_{\Omega}^4).
\end{equation}
Moreover, we have the following proper symplectic  inclusion:
\[\bigl(X_{\Omega}^4,\omega_{\mathrm{std}}\bigr)\subset \bigl(E^{\mathrm{ext}},\omega_{\mathrm{std}}\bigr).\]
Let $E^{\mathrm{ext}}_{\mathrm{FR}}$ be a  rounding (cf. Subsection \ref{roudningfullyconvex}) of $E^{\mathrm{ext}}$ such that Theorem \ref{GH} and Theorem \ref{puctured-dsik} hold for the positive integer $k$ that satisfies Equation (\ref{GH-cap}). We still have a proper symplectic inclusion
\[\bigl(X_{\Omega}^4,\omega_{\mathrm{std}}\bigr)\subset \bigl(E^{\mathrm{ext}}_{\mathrm{FR}},\omega_{\mathrm{std}}\bigr).\]
We now see $L$ as a Lagrangian torus in the interior of the rounded extended ellipsoid $E^{\mathrm{ext}}_{\mathrm{FR}}$.  By assumption, $L$ is extremal in $(X_{\Omega}^4,\omega_{\mathrm{std}})$. This means that for every smooth disk $u:(D^2,\partial D^2)\to (E^{\mathrm{ext}}_{\mathrm{FR}}, L)$ we have 
\[\int_{D^2}u^*\omega_{\mathrm{std}}>0 \implies \int_{D^2}u^*\omega_{\mathrm{std}}\geq \operatorname{diagonal}(X_{\Omega}^4).\]
Here, we point out that the symplectic area of a disk $u:(D^2,\partial D^2)\to (E^{\mathrm{ext}}_{\mathrm{FR}}, L)$ depends only on its relative homotopy class $[u]\in \pi_2(E^{\mathrm{ext}}_{\mathrm{FR}}, L)$. Every disk  $u:(D^2,\partial D^2)\to (E^{\mathrm{ext}}_{\mathrm{FR}}, L)$ can be smoothly deformed to a disk $u:(D^2,\partial D^2)\to (X_{\Omega}^4, L)$ without changing its relative homotopy class.

Let $g$ be the standard flat metric on the torus $L$ and $(D_{1+\delta_0}^*L, d\lambda_{\mathrm{can}})$ be the symplectic codisk bundle of $(L,g)$ of radius $1+\delta_0$, for $\delta_0>0$, i.e., 
\[D_{1+\delta_0}^*L:=\bigl\{v\in T^*L: \|v\|_g\leq 1+\delta_0 \bigr\},\]
and $\lambda_{\mathrm{can}}$ is the canonical $1$-form on $T^*L$. The unit cosphere bundle $\Pi:(S^*L,\lambda_{\mathrm{can}})\to L$ given by 
\[S^*L:=\bigl \{v\in T^*L: \|v\|_g=1\bigr\}\]
is a contact type boundary of the unit symplectic codisk $(D^*L,  d\lambda_{\mathrm{can}})$. Moreover, $(D^*L,  d\lambda_{\mathrm{can}})$ is a Liouville domain whose symplectic completion is isomorphic to $(T^*L, d\lambda_{\mathrm{can}})$.

 For some $\delta_0>0$, after scaling $g$ by a constant, Weinstein's neighborhood theorem ensures that we can identify a neighborhood of $L$ in $\operatorname{int} (E^{\mathrm{ext}}_{\mathrm{FR}})$  with  $D_{1+\delta_0}^*L$ via a symplectic embedding 
\[\Phi: (D_{1+\delta_0}^*L,d\lambda_{\mathrm{can}})\to (\operatorname{int} (E^{\mathrm{ext}}_{\mathrm{FR}}), \omega_{\mathrm{std}})\]
such that
\[\Phi|_{L}=\operatorname{Id}_L.\]

By assumption, the Lagrangian torus $L$ intersects the interior of $X_{\Omega}^4$ (hence the interior of $E^4(a,b)$ as well). Fix a point $p \in L\cap \operatorname{int} (E^4(a,b))$. Let  $D^{2n}(\delta)$ be the closed disk of radius $\delta>0$ centered at the origin in $\mathbb{R}^{2n}$. For sufficiently small $\delta>0$, choose a symplectic embedding $\phi:D^{2n}(\delta)\to \operatorname{int} (E^4(a,b))$ such that 
\[\phi(0)=p,\ \phi^{-1}(L)=D^{2n}(\delta)\cap \mathbb{R}^n.\]
This geometric setup is illustrated in Figure \ref{geometric-setup1} below.

\begin{figure}[h]
	\centering
	\includegraphics[width=6cm]{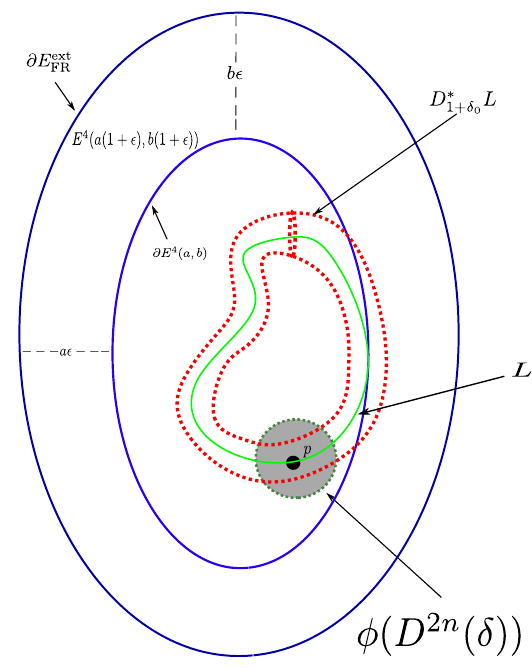}
	\caption{Geometric Setup}\label{geometric-setup1}
\end{figure}

\subsection{Perturbing a contact type hypersurface}\label{perturbsurface}
The contact type  hypersurface 
\[(\Phi(S^*L),\Phi_*\lambda_{\mathrm{can}}|_{S^*L}) \subset(E^{\mathrm{ext}}_{\mathrm{FR}},\omega_{\mathrm{std}})\]
has degenerate Reeb dynamics. Closed Reeb orbits come in $S^1$-families. In general, as explained in Cieliebak--Frauenfelder--Paternain \cite{Cieliebak:2009aa}, a stable hypersurface in a symplectic manifold can not be perturbed to a nondegenerate one; there is a possibility of losing the stability of the hypersurface under a generic $C^\infty$-perturbation. However, the space of closed hypersurfaces of contact type in a symplectic manifold forms an open subset of the space closed hypersurfaces in the $C^{\infty}$-topology. Therefore, we can perturb $(\Phi(S^*L),\Phi_*\lambda_{\mathrm{can}}|_{S^*L})$ into a contact-type hypersurface whose Reeb flow is non-degenerate. Below, we describe one perturbation that is suitable for our purpose.

Recall that the unit cosphere bundle $(S^*L,\lambda_{\mathrm{can}}|_{S^*L})$ for the flat metric on $L$ is of Morse--Bott type, closed Reeb orbits comes in $S^1$-families. For $T>0$, an arclength parametrized closed curve $c:\mathbb{R}/T\mathbb{Z}\to L$ is a geodesic of length $T$ with respect to the flat metric $g$ if and only its lift to the unit cosphere bundle $\gamma_c:=\hat{g}\circ c: \mathbb{R}/T\mathbb{Z}\to (S^*L,\lambda_{\mathrm{can}}|_{S^*L}) $ defined by 
\[\hat{g}\circ c(t):=(c(t),g_{c(t)}(c'(t),.))\]
is a closed Reeb orbit of period (action) $T$; see {\cite[Section 2.3]{Frauenfelder_2018}} for details.

Each Morse--Bott $ S^1$-family of closed Reeb orbits has a unique Reeb orbit, that under the projection $\Pi:(S^*L,\lambda_{\mathrm{can}})\to L$, projects to a closed geodesic on $L$ that passes through the point $p\in L$,  chosen in the intersection of $L$ and the interior of $X_{\Omega}^4$.

Let $\phi^t: S^*L\to S^*L$ be the Reeb flow for the contact form $\lambda_{\mathrm{can}}|_{S^*L}$. For $T>0$, define
\[N_T:=\{x\in S^*L: \phi^T(x)=x \}.\]
Then $N_T$ is a two-dimensional submanifold of $S^*L$. The Reeb flow $\phi^t$ induces an $S^1$-action on $N_T$; the quotient  $S_T:=N_T/S^1$  is a smooth one-dimensional manifold\footnote{In general, the quotient  $S_T:=N_T/S^1$ is an orbifold with singularity groups $\mathbb{Z}_k$ corresponding to closed Reeb orbits with period $T/k$, covered $k$ times.}. We fix $T=c_{k}^{\mathrm{GH}}(E^{\mathrm{ext}})$ and define
\[N_{\leq T}:=\bigcup_{T'\leq T} N_{T'}.\]
Then $S_{\leq T}:=N_{\leq T}/S^1$ is a smooth manifold consisting of finitely many Morse--Bott $ S^1$-families of closed Reeb orbits.  Choose a smooth nonnegative function $\bar{f}$ supported in a small tabular neighborhood of $N_{\leq T}$ that descends to a Morse function $f: S_{\leq T}\to \mathbb{R}$ having two critical points on each Morse--Bott circle with maximum at the closed Reeb orbit whose projection to $L$ is the geodesic passing through the point $p\in L$. 

Perturb the contact form $\lambda_{\mathrm{can}}|_{S^*L}$ to $\bar{\lambda}=(1+\alpha \bar{f}) \lambda_{\mathrm{can}}|_{S^*L}$, where $\alpha$ is positive real number. If $\alpha$ is sufficiently small, then by {\cite[Lemma 2.3]{Bourgeois2002}} all closed Reeb orbits of the perturbed form of action at most $c_{k}^{\mathrm{GH}}(E^{\mathrm{ext}})$ are nondegenerate and correspond to critical points of the Morse function $f$. Since $f$ has only two critical points on each Morse--Bott family in $S_{\leq T}$, each Morse--Bott circle is replaced by two closed Reeb orbits. 

Let $\Gamma$ be a family of orbits in $S_{\leq T}$ and let $\gamma$ be a  Reeb orbit obtained from perturbing $\Gamma$ as above. By {\cite[Lemma 2.4]{Bourgeois2002}}, the Conley--Zehnder index of $\gamma$ is given by 
\begin{equation}\label{ind}
	\operatorname{CZ}^\tau(\gamma)=\operatorname{CZ}^\tau(\Gamma)+\operatorname{index}_{\gamma}(f),
\end{equation} 
where $\operatorname{CZ}^\tau(\Gamma)$ is the Conley--Zehnder index of the family $\Gamma$ and $\operatorname{index}_{\gamma}(f)$ is the Morse index of $f$ at the critical point $\gamma$. For the flat metric, $\operatorname{CZ}^\tau(\Gamma)=0$ for the trivialization $\tau$ satisfying the second bullet point of Theorem \ref{special-trivalization}. In this trivialization, we have 
\begin{equation}\label{ind-f}
	\operatorname{CZ}^\tau(\gamma)=\operatorname{index}_{\gamma}(f).
\end{equation} 
In particular, each Morse--Bott circle of orbits splits into two closed Reeb orbits, one hyperbolic and one elliptic. The elliptic Reeb orbit corresponds to a closed geodesic on $L$, for the flat metric $g_0$, passing through the point $p$ according to  Equation (\ref{ind-f}). Moreover, we have 
\begin{equation}\label{ind01}
	\operatorname{CZ}_\tau(\gamma)\in \{0,1\}
\end{equation}
for every closed Reeb orbit $\gamma$ of action at most $c_{k}^{\mathrm{GH}}(E^{\mathrm{ext}})$.

The cotangent bundle $(T^*L, d\lambda_{\mathrm{can}})$ is canonically isomorphic to the unit codisk bundle $D^*L$  of $(L,g)$ with the positive cylindrical end $([0,\infty)\times S^*L,d(e^r\lambda_{\mathrm{can}}|_{S^*L}) )$ attached along its boundary. Define 
\[X:=D^*L\cup \big\{(r,x)\in [0,\infty)\times S^*L:r\leq \log(1+\alpha \bar{f}(x))\big\}.\]
 Under the identification of $S^*L$ with $\partial X$ given by 
 \[x\mapsto (\log(1+\alpha \bar{f}(x)),x),\]
we have $(1+\alpha \bar{f})\lambda_{\mathrm{can}}|_{TS^*}=\lambda_{\mathrm{can}}|_{\partial X}$. By construction, the new contact type hypersurface
\[(\partial X,\lambda_{\mathrm{can}}|_{\partial X}) \subset(T^*L, d\lambda_{\mathrm{can}})\]
has the property that every closed Reeb orbit  $\gamma$ of action less or equal to $c_{k}^{\mathrm{GH}}(E^{\mathrm{ext}})$ is non-contractible, non-degenerate, and satisfies 
\[\operatorname{CZ}_\tau(\gamma)\in \{0,1\}.\]
Moreover, every closed Reeb orbit of action at most $c_{k}^{\mathrm{GH}}(E^{\mathrm{ext}})$ projects to a closed geodesic for the flat metric on $L$. A closed Reeb orbit with $\operatorname{CZ}_\tau(\gamma)=1$ projects to the unique geodesic, for the flat metric, in its homotopy class that passes through the a priori chosen point $p\in L$. For small $\alpha$, $\partial X$ is a $C^1$-small perturbation of $S^*L$.

By choosing $\alpha$ small enough, we can ensure that the symplectic embedding 
 \[\Phi: (D_{1+\delta_0}^*L,d\lambda_{\mathrm{can}})\to (\operatorname{int} (E^{\mathrm{ext}}_{\mathrm{FR}}), \omega_{\mathrm{std}})\] 
restricts to a symplectic embedding 

\[\Phi: (X,d\lambda_{\mathrm{can}})\to (\operatorname{int} (E^{\mathrm{ext}}_{\mathrm{FR}}), \omega_{\mathrm{std}}) \text{ with } \Phi|_{L}=\operatorname{Id}_L.\]
From now on, we will denote $\Phi( X)$ by $X$ and $ \Phi_*\lambda_{\mathrm{can}}$ by $\lambda_{\mathrm{can}}$.

\subsection{Stretching the neck and more}
Consider the contact type  hypersurface 
\[(\partial X,\lambda_{\mathrm{can}}|_{\partial X}) \subset(E^\mathrm{ext}_{\mathrm{FR}},\omega_{\mathrm{std}}).\]
Stretching the neck along this hypersurface, we get the split cobordism 
$\bigsqcup_{N=0}^{3} (\widehat{X}_N,\widehat \Omega_N)$, where 
\begin{equation*}
(\widehat{X}_N,\widehat \Omega_N):=
\begin{cases}
		\bigl(\mathbb{R}\times \partial E^\mathrm{ext}_{\mathrm{FR}}, d(e^r\lambda_{\mathrm{std}})\bigr) & \text{for } N=3,\\
		\bigl(\widehat{E^\mathrm{ext}_{\mathrm{FR}}\setminus X}, \widehat{\omega}_{\mathrm{std}}\bigr)\cong \bigl(\widehat{E^\mathrm{ext}_{\mathrm{FR}}}\setminus L ,\widehat{\omega}_{\mathrm{std}}\bigr) & \text{for } N=2,\\
		
		\bigl(\mathbb{R}\times \partial X,d(e^r \lambda_{\mathrm{can}})\bigr)& \text{for } N=1,\\
		\bigl(T^*L, d\lambda_{\mathrm{can}}\bigr)\cong \bigl(\widehat{X}, \widehat{d\lambda}_{\mathrm{can}})\bigr) & \text{for } N=0.\\
\end{cases}
\end{equation*}
%
%

Choose SFT-admissible almost complex structures $J_{\mathrm{top}}$, $J_{\mathrm{bot}}$, $J_{\mathrm{int}}$, and $J^\mathrm{ext}_{\mathrm{top}}$ on $\bigl(\widehat{E^{\mathrm{ext}}_{\mathrm{FR}}\setminus X},\widehat{\omega}_{\mathrm{std}}\bigr)$, $(T^*L, d\lambda_{\mathrm{can}})$, $\bigl(\mathbb{R}\times \partial X, d(e^r\lambda_{\mathrm{can}}|_{\partial X})\bigr)$, and $\bigl(\mathbb{R}\times \partial E^\mathrm{ext}_{\mathrm{FR}}, d(e^r\lambda_{\mathrm{std}}) \bigr)$, respectively. Let $J_i$ be a family of SFT-admissible almost complex structures on the symplectic completion $\widehat{E^\mathrm{ext}_{\mathrm{FR}}}$ that stretches the neck along  $\partial X$. As $i\to \infty$, $(\widehat{E^\mathrm{ext}_{\mathrm{FR}}}, J_i)$ splits into the cylindrical almost complex manifolds $\bigl(\widehat{E^{\mathrm{ext}}_{\mathrm{FR}}\setminus X},J_{\mathrm{top}}\bigr)$, $(T^*L,J_{\mathrm{bot}})$, $(\mathbb{R}\times \partial X,J_{\mathrm{int}})$, and $\bigl(\mathbb{R}\times \partial E^\mathrm{ext}_{\mathrm{FR}}, J^\mathrm{ext}_{\mathrm{top}} \bigr).$

Pick a point $y\in L$ and choose $J_i$ to be constantly equal to some integrable almost complex structure near $y$ for every $i$. Fix a local complex divisor $D$ containing $y$. By Theorem \ref{puctured-dsik}, there is a sequence of rigid $J_i$-holomorphic planes in $\widehat{E^{\mathrm{ext}}_{\mathrm{FR}}}$ asymptotic to an elliptic simple Reeb orbit $e_k$ on $\partial E^{\mathrm{ext}}_{\mathrm{FR}}$. Moreover, each $u_i$ carries the local tangency constraint $\ll \mathcal{T}_D^{k-1}y\gg$. By the Gromov--Hofer compactness theorem, as $i \to \infty$, the sequence $u_i$ degenerates to a holomorphic building $\mathbb{H}=(\textbf{u}^0, \textbf{u}^1,\dots, \textbf{u}^{N_+})$ in $\bigsqcup_{N=0}^{N=N_+} (\widehat{X}_N,\widehat \Omega_N,\tilde{ \Omega}_N,J_N)$, for some integer $N_+\geq 0$, where 
\begin{equation*}
(\widehat{X}_N,\widehat \Omega_N,\tilde{ \Omega}_N,J_N):=
\begin{cases}
		\bigl(\mathbb{R}\times \partial E^\mathrm{ext}_{\mathrm{FR}}, d(e^r\lambda_{\mathrm{std}}),d\lambda_{\mathrm{std}}, J^\mathrm{ext}_{\mathrm{top}}\bigr) & \text{for } N\in \{N_+-l+1,\dots,N_+\},\\
		(\widehat{E^\mathrm{ext}_{\mathrm{FR}}\setminus X}, \widehat{\omega}_{\mathrm{std}},\tilde{\omega}_{\mathrm{std}}, J_{\mathrm{top}}) & \text{for } N=N_+-l,\\
		
		(\mathbb{R}\times \partial X,d(e^r \lambda_{\mathrm{can}}),d\lambda_{\mathrm{can}}, J_{\mathrm{int}})& \text{for } N\in \{1,\dots,N_+-l-1\},\\
		(T^*L, d\lambda_{\mathrm{can}},d\lambda_{\mathrm{can}}^{\mathrm{up}}|_{\partial X}, J_{\mathrm{bot}}) & \text{for } N=0.\\
\end{cases}
\end{equation*}
Here, we recall  that $d\lambda_{\mathrm{can}}$, $\omega^c_{\mathrm{std}}$, and $\tilde{\omega}^{\mathrm{up}}_{\mathrm{std}}$ are the $2$-forms defined as:
\begin{equation*}
	d\lambda_{\mathrm{can}}^{\mathrm{up}}|_{\partial X}:=
	\begin{cases}
		
		d\lambda_{\mathrm{can}}|_{\partial X} & \text{on } [0,\infty)\times  \partial X,\\
		d\lambda_{\mathrm{can}} & \text{on } X,
		
	\end{cases}
\end{equation*}
\begin{equation*}
	\widehat{\omega}_{\mathrm{std}}:=
	\begin{cases}
		
		d(e^r\lambda_{\mathrm{std}} )& \text{on } [0,\infty)\times \partial E^\mathrm{ext}_{\mathrm{FR}},\\
		\omega_{\mathrm{std}} & \text{on } E^\mathrm{ext}_{\mathrm{FR}}\setminus X,\\
		d(e^r\lambda_{\mathrm{can}}|_{\partial X}) & \text{on } (-\infty,0]\times  \partial X,

	\end{cases}
\end{equation*}
and
\begin{equation*}
	\tilde{\omega}_{\mathrm{std}}:=
	\begin{cases}
		d\lambda_{\mathrm{std}}& \text{on } [0,\infty)\times \partial E^\mathrm{ext}_{\mathrm{FR}},\\
		\omega_{\mathrm{std}} & \text{on } E^\mathrm{ext}_{\mathrm{FR}}\setminus X,\\
		d\lambda_{\mathrm{can}}|_{\partial X} & \text{on } (-\infty ,0]\times  \partial X.
	\end{cases}
\end{equation*}

Since the building $\mathbb{H}$ is the limit of a sequence of genus zero asymptotically cylindrical rigid curves, it has genus zero. Also, it has index zero, i.e., the sum of indices of its curve components is equal to zero.

For some positive integer $k_N$, we write $\textbf{u}^N=(u^N_1, \dots, u^N_{k_N})$, where $u_i^N$ are the smooth connected punctured curves 
in $(\widehat{X}_N,\widehat \Omega_N,\tilde{ \Omega}_N,J_N)$ that constitute the level $\textbf{u}^N$ of the building. Since the neck-stretching limit $\mathbb{H}$ was constructed by the curves from Theorem \ref{puctured-dsik}, we get the following bound on the energy of $\mathbb{H}$:
\begin{equation}\label{engery2}
	E(\mathbb{H}):=\sum_{N=0}^{N_+}\sum_{i=1}^{k_N}\int_{u_i^N}\tilde{ \Omega}_N\leq c_{k}^{\mathrm{GH}}(E^{\mathrm{ext}})=k(1+\epsilon) \operatorname{diagonal}(X_{\Omega}^4).
\end{equation}
This, in particular, implies that each closed Reeb orbit appearing in the holomorphic building has an action at most  $c_{k}^{\mathrm{GH}}(E^{\mathrm{ext}})$ and is, therefore,  non-contractible, non-degenerate, and projects to a closed geodesic of Morse index either $0$ or $1$.
\begin{figure}[h]
	\centering
	\includegraphics[width=8cm]{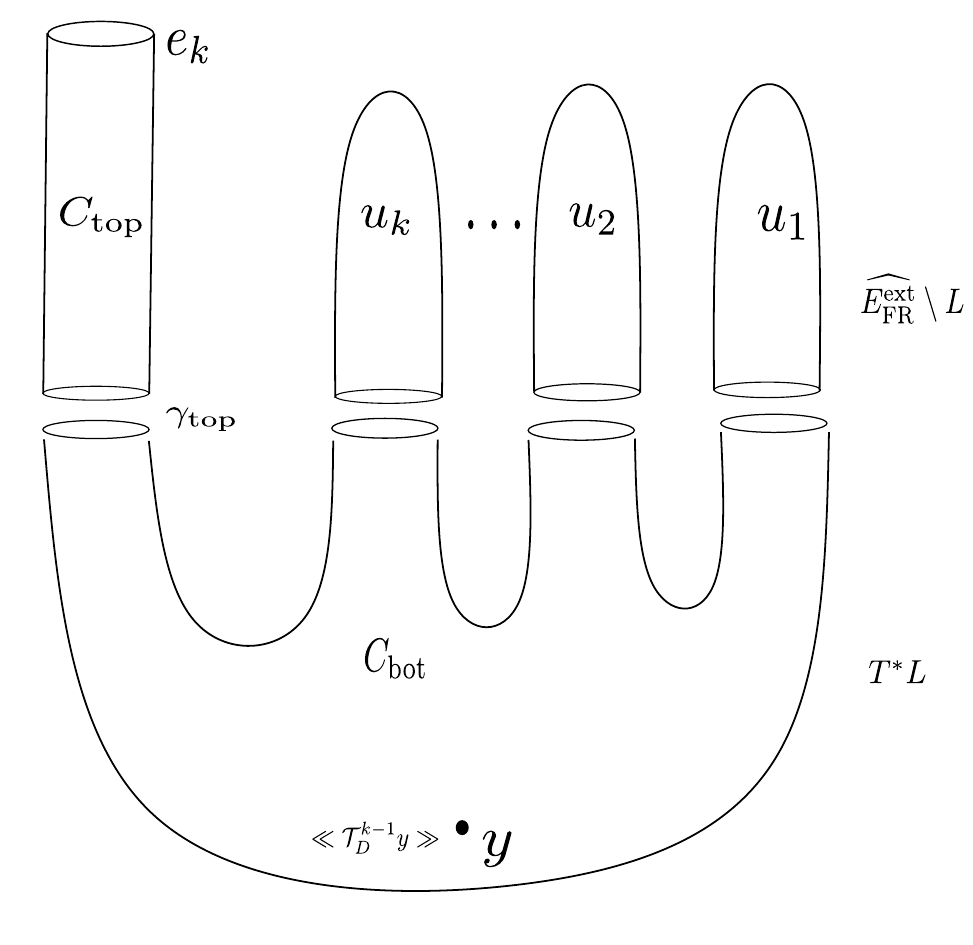}
	\caption{The holomorphic building $\mathbb{H}$.}\label{holo2}
\end{figure}
\begin{claim}\label{claim2}
	If $\epsilon>0$ is sufficiently small,  then:
	\begin{itemize}
		\item  The symplectization levels $\textbf{u}^1, \dots, \textbf{u}^{N_+-l-1}, \textbf{u}^{N_+-l+1}, \dots,\textbf{u}^{N_+}$ are all empty;
		\item The bottom level $\textbf{u}^0$ that sits in $(T^*L, d\lambda_{\mathrm{can}}, J_{\mathrm{bot}})$ consists of a single smooth connected $J_{\mathrm{bot}}$-holomorphic asymptotically cylindrical sphere $C_{\mathrm{bot}}$ with exactly $k+1$ positive punctures. It inherits the tangency constraint  $\ll \mathcal{T}_D^{k-1}y\gg$ at $y$. Moreover, it has zero Fredholm index for generic $J_{\mathrm{bot}}$;
		\item  The  level $\textbf{u}^{N_+-l}$ that sits in $(\widehat{E^\mathrm{ext}_{\mathrm{FR}}\setminus X}, \widehat{\omega}_{\mathrm{std}},\tilde{\omega}_{\mathrm{std}}, J_{\mathrm{top}})$ consists of a single smooth connected asymptotically cylindrical $J_{\mathrm{top}}$-holomorphic  cylinder $C_{\mathrm{top}}$, and $k$ asymptotically cylindrical $J_{\mathrm{top}}$-holomorphic planes $u_1,u_2,\dots, u_k$ with negative ends asymptotic to Reeb orbits on $\partial X$.  The Fredholm indices of $u_1,u_2,\dots, u_k,$ and $C_{\mathrm{top}}$ are zero for generic $J_{\mathrm{top}}$;
		\item The $J_{\mathrm{top}}$-holomorphic cylinder $C_{\mathrm{top}}$ has positive end on $e_k$ and negative end a Reeb orbit, denoted by $\gamma_{\mathrm{top}}$, on $\partial X$. The orbit $\gamma_{\mathrm{top}}$ is an elliptic Reeb \footnote{ A closed Reeb orbit $\gamma$ is called elliptic if there exists a trivialization  $\tau$ such that $\operatorname{CZ}^\tau(\gamma)$ has odd parity. This notion is well-defined because the parity of $\operatorname{CZ}^\tau(\gamma)$ does not depend on the trivialization  $\tau$.} orbit and therefore projects to a closed geodesic on $L$ passing through $p\in L$, a point chosen in the intersection of $L$ and the interior of the domain $X^4_\Omega$. Moreover, $C_{\mathrm{top}}$ is rigid for generic  $J_{\mathrm{top}}$. For an illustration of the building $\mathbb{H}$, see Figure \ref{holo2}. 
	\end{itemize}
\end{claim}
We suppose Claim \ref{claim2} holds and continue with the proof Theorem \ref{extremal-lag-toric}. A proof of Claim \ref{claim2} occupies Subsection \ref{proofofclaim}.

\subsection{Estimating the symplectic area}
Suppose Claim \ref{claim2} holds. Consider the moduli space
\begin{equation*}
	\mathcal{M}^{J_{\mathrm{top}}}_{E^\mathrm{ext}_{\mathrm{FR}}\setminus X,C_{\mathrm{top}} }(e_k,\gamma_{\mathrm{top}}):=\left\{
	\begin{array}{l}
		u:(\mathbb{R}\times \mathbb{R}/\mathbb{Z},i) \to (\widehat{E^\mathrm{ext}_{\mathrm{FR}}\setminus X},J_{\mathrm{top}}),\\
		du\circ i=J_{\mathrm{top}}\circ du  ,\\
		u \text{ is asymptotic to  $e_k$ at $\infty$,} \\
		u \text{ is asymptotic to $\gamma_{\mathrm{top}}$ at $-\infty$,} \\
		\text{and }[u]=[C_{\mathrm{top}}]\in H_2(E^\mathrm{ext}_{\mathrm{FR}}\setminus X, e_k\cup \gamma_{\mathrm{top}}, \mathbb{Z}).
	\end{array}
	\right\}\bigg/\operatorname{Aut}(\mathbb{R}\times \mathbb{R}/\mathbb{Z}).
\end{equation*}
By Remark \ref{somewhere-injective}, the Reeb orbit $e_k$ is simple, so this moduli space consists of somewhere injective curves. By the third bullet point in Claim \ref{claim2}, it is a zero-dimensional manifold for generic $J_{\mathrm{top}}$. 
\begin{lemma}\label{contra}
For every $u\in \mathcal{M}^{J_{\mathrm{top}}}_{E^\mathrm{ext}_{\mathrm{FR}}\setminus X,C_{\mathrm{top}} }(e_k,\gamma_{\mathrm{top}})$, we have 
\[0<\int u ^*\tilde{\omega}^{\mathrm{up}}_{\mathrm{std}}\leq k\epsilon \operatorname{diagonal}(X_{\Omega}^4),\]
where $\tilde{\omega}^{\mathrm{up}}_{\mathrm{std}}$ is the $2$-form defined by
\begin{equation*}
	\tilde{\omega}^{\mathrm{up}}_{\mathrm{std}}:=
	\begin{cases}
		d\lambda_{\mathrm{std}}& \text{on } [0,\infty)\times \partial E^\mathrm{ext}_{\mathrm{FR}},\\
		\omega_{\mathrm{std}} & \text{on } E^\mathrm{ext}_{\mathrm{FR}}\setminus L.
	\end{cases}
\end{equation*}
\end{lemma}
\begin{proof}
By the maximum principle, the  $J_{\mathrm{top}}$-holomorphic planes $u_1,u_2,\dots, u_k$ are contained in $\widehat{E^{\mathrm{ext}}_{\mathrm{FR}}\setminus X}\setminus (0,\infty)\times \partial E^\mathrm{ext}_{\mathrm{FR}}$. These $J_{\mathrm{top}}$-holomorphic planes $u_1,u_2,\dots, u_k$ are asymptotic to Reeb orbits which correspond to the critical point of the Morse function $f$. Therefore, these  planes seen in $\widehat{E^\mathrm{ext}_{\mathrm{FR}}}\setminus L$  can be compactified to  smooth disks in $\widehat{E^\mathrm{ext}_{\mathrm{FR}}}\setminus L$ with boundaries on $L$. Since $L$ is extremal in $X^4_\Omega$, we have 
\begin{equation}\label{est1}
	\sum_{i=1}^k\int u_i^*\tilde{\omega}^{\mathrm{up}}_{\mathrm{std}}=	\sum_{i=1}^k\int u_i^*\omega_{\mathrm{std}}\geq k \operatorname{diagonal}(X_{\Omega}^4).
\end{equation}
Also (\ref{engery2}) implies
\[0<\int C_{\mathrm{top}}
	 ^*\tilde{\omega}^{\mathrm{up}}_{\mathrm{std}}+\sum_{i=1}^k\int u_i^*\tilde{\omega}^{\mathrm{up}}_{\mathrm{std}}\leq E(\mathbb{H})\leq c_{k}^{\mathrm{GH}}(E^{\mathrm{ext}}).\]
This together with the estimate (\ref{est1}) and Equation (\ref{GH-cap}) gives
\[
	0<\int C_{\mathrm{top}}
	^*\tilde{\omega}^{\mathrm{up}}_{\mathrm{std}}\leq k\epsilon \operatorname{diagonal}(X_{\Omega}^4).
\]
The integral on the left side does not depend on the particular representative of the relative homology class $[C_{\mathrm{top}}]\in H_2(E^\mathrm{ext}_{\mathrm{FR}}\setminus X, e_k\cup \gamma_{\mathrm{top}}, \mathbb{Z})$. So for every $u\in \mathcal{M}^{J_{\mathrm{top}}}_{E^\mathrm{ext}_{\mathrm{FR}}\setminus X,C_{\mathrm{top}} }(e_k,\gamma_{\mathrm{top}})$, we have 
\[0<\int u ^*\tilde{\omega}^{\mathrm{up}}_{\mathrm{std}}\leq k\epsilon \operatorname{diagonal}(X_{\Omega}^4).\qedhere\]
\end{proof}
Following the proof of Lemma \ref{counting} we have:
\begin{lemma}\label{compact}
	Suppose $\epsilon<\frac{1}{k}$. For every family $\{J_t\}_{t\in[0,1]}$ of SFT-admissible almost complex structures, the parametric moduli space $\mathcal{M}^{J_t}_{E^\mathrm{ext}_{\mathrm{FR}}\setminus X,C_{\mathrm{top}} }(e_k,\gamma_{\mathrm{top}})$ is compact. In particular, the moduli space $\mathcal{M}^{J_{\mathrm{top}}}_{E^\mathrm{ext}_{\mathrm{FR}}\setminus X,C_{\mathrm{top}} }(e_k,\gamma_{\mathrm{top}})$ is compact.\qed
\end{lemma}
From this lemma it follows that for generic $J_{\mathrm{top}}$,  $\mathcal{M}^{J_{\mathrm{top}}}_{E^\mathrm{ext}_{\mathrm{FR}}\setminus X,C_{\mathrm{top}} }(e_k,\gamma_{\mathrm{top}})$ is discrete set consisting of finitely many points. Each point carries a sign (positive or negative) depending on whether it inherits a positive or negative orientation.
\begin{lemma}\label{non-venishing}
	Suppose $\epsilon<\frac{1}{k}$ and  $J_{\mathrm{top}}$ is generic. Every element in $\mathcal{M}^{J_{\mathrm{top}}}_{E^\mathrm{ext}_{\mathrm{FR}}\setminus X,C_{\mathrm{top}} }(e_k,\gamma_{\mathrm{top}})$ is positively oriented. In particular,  \[\#\mathcal{M}^{J_{\mathrm{top}}}_{E^\mathrm{ext}_{\mathrm{FR}}\setminus X,C_{\mathrm{top}} }(e_k,\gamma_{\mathrm{top}})>0.\] 
	Moreover, this count does not depend on the choice of generic $J_{\mathrm{top}}$.
\end{lemma}
\begin{proof}
By the fourth bullet point in Claim \ref{claim2}, the Reeb orbit $\gamma_{\mathrm{top}}$ is elliptic.  By Remark \ref{elliptic},  $e_k$ is also an elliptic Reeb orbit. So the ends of the cylinder $C_{\mathrm{top}}$ are on elliptic Reeb orbits. By the fourth bullet point in Claim \ref{claim2}, $C_{\mathrm{top}}$ has index zero for generic $J_{\mathrm{top}}$. Somewhere injective curves of index zero are generically immersed: consider a somewhere injective curve $u$ of index zero. Adding a mark point to $u$ increases the index by $2$, but asking the mark point to be a critical point of $u$ drops the index by at least $4$ (cf. {\cite[Proposition A.1]{Wendl2023}}). So, generically, such a configuration will have a negative index, which is impossible.

 By {\cite[Proposition 5.2.2]{McDuff:2021aa}}, the curve $C_{\mathrm{top}}$, and hence every element in $\mathcal{M}^{J_{\mathrm{top}}}_{E^\mathrm{ext}_{\mathrm{FR}}\setminus X,C_{\mathrm{top}} }(e_k,\gamma_{\mathrm{top}})$, carries a positive orientation.

That the count $\#\mathcal{M}^{J_{\mathrm{top}}}_{E^\mathrm{ext}_{\mathrm{FR}}\setminus X,C_{\mathrm{top}} }(e_k,\gamma_{\mathrm{top}})$ does not depend on the choice of generic $J_{\mathrm{top}}$ follows from a cobordism argument and Lemma \ref{compact}.
\end{proof}
\subsection{Constrained half-cylinders in $T^*T^2$}\label{counthalfcy}
Pick an SFT-admissible almost complex structure $J$ on $(\widehat{X},\widehat{d\lambda}_{\mathrm{can}})=(T^*T^2, d\lambda_{\mathrm{can}})$ and a point $p$ on the zero section $T^2$.  
Define
\begin{equation*}
	\mathcal{M}^{J}_{\widehat{X}}(\gamma_{\mathrm{top}}, T^2,  p):=\left\{(u,t_0):
	\begin{array}{l}
		u:([0,\infty)\times S^1,i) \to (T^*T^2,J),\\
		du\circ i=J\circ du  ,\\
		u \text{ is asymptotic to $\gamma_{\mathrm{top}}$ at $\infty$,} \\
		u(0,t)\in T^2 \text{ for all $t\in S^1$,}\\
		u \text{ passes through $p$ at $(0,t_0)$,}

	\end{array}
	\right\}\bigg/\operatorname{Aut}([0,\infty)\times S^1).
\end{equation*}
Here, the marked point $t_0$ is allowed to move on $\{0\}\times S^1$. The Fredholm index (cf. {\cite[Section 8, Page 33]{Cieliebak:2007aa}}) of this moduli space is zero:
\[\operatorname{ind}(\mathcal{M}^{J}_{\widehat{X}}(\gamma_{\mathrm{top}},T^2, p))=(n-3)(2-2)+\operatorname{CZ}^\tau(\gamma_{\mathrm{top}})-n+1=0.\]
Here, by Theorem \ref{special-trivalization}, we have 
\[\operatorname{CZ}^\tau(\gamma_{\mathrm{top}})=\mu(c)=n-1,\]
where  $c$ denotes the closed geodesic that lifts to $\gamma_{\mathrm{top}}$. We are interested to show that the signed count $\# \mathcal{M}^{J}_{\widehat{X}}(\gamma_{\mathrm{top}},T^2, p)$ does not vanish. 
\begin{lemma}\label{counthalf}
	We have 
	\[\#\mathcal{M}^{J}_{\widehat{X}}(\gamma_{\mathrm{top}},T^2, p)\neq 0.\]
\end{lemma}
\begin{proof}
Consider the flat torus $(T^2,g)$. We think of $(T^*T^2,d\lambda_{\mathrm{can}})$ as the symplectic completion of the unit codisk bundle $D^*T^2$. We define an SFT-admissible almost complex structure on $(T^*T^2,d\lambda_{\mathrm{can}})$ as follows. We write  $T^*T^2=S^1\times S^1\times \mathbb{R}^2$. Let $(\theta_1, \theta_2)$ denote the normal coordinates on $T^2=S^1\times S^1$ and $(p_1, p_2)$ be the dual coordinates on $\mathbb{R}^2$. Choose a smooth function $\rho: [0,\infty)\to \mathbb{R}$ such that 
\begin{itemize}
	\item $\rho(t)>0$ and $\rho'(t)\geq 0$ for all $t\geq 0$,
	\item $\rho(t)=1$ for $t \leq 1$, and
	\item $\rho=t$ for large $t$.
\end{itemize}
Define $J_{\rho}$ by 
\[J_{\rho}(\partial_{\theta_i}):=-\rho(\|(p_1,p_2)\|)\partial_{p_i}. \]
\[J_{\rho}(\partial_{p_i}):=\frac{1}{\rho(\|(p_1,p_2)\|)}\partial_{\theta_i}. \]
The almost complex structure $J_{\rho}$ on $T^*T^2$ is compatible with $d\lambda_{\mathrm{can}}$ and cylindrical for the unit cosphere $S^*T^2$, for the choice of flat metric. Hence, it is SFT-admissible.

Closed Reeb orbits on the unit cosphere bundle $(S^*T^2,\lambda_{\mathrm{can}}|_{S^*T^2})$, for the choice of flat metric, come in Morse--Bott $S^1$-families. Let $\Gamma_{\mathrm{top}}$ be the Morse--Bott family that contains the closed orbit $\gamma_{\mathrm{top}}$. Consider the moduli space
\begin{equation*}
	\mathcal{M}^{J_\rho}_{T^*T^2}(T^2,\Gamma_{\mathrm{top}}, p):=\left\{(u,t_0):
	\begin{array}{l}
			u:([0,\infty)\times S^1,i) \to (T^*T^2,J_\rho),\\
			du\circ i=J_{\rho}\circ du  ,\\
			u \text{ is asymptotic to some $\gamma \in \Gamma_{\mathrm{top}}$ at $\infty$,} \\
			u(0,t)\in T^2 \text{ for all $t\in S^1$,}\\
			u \text{ passes through $p$ at $(0,t_0)$}
		\end{array}
	\right\}\bigg/\operatorname{Aut}([0,\infty)\times S^1).
\end{equation*}
Here, the marked point $t_0$ is allowed to vary on $\{0\}\times S^1$.
The Fredholm index of this moduli space is zero:
\[\operatorname{ind}(\mathcal{M}^{J_\rho}_{T^*T^2}(T^2,\Gamma_{\mathrm{top}}, p)=(n-3)(2-2)+\operatorname{CZ}^\tau(\Gamma_{\mathrm{top}})+1-2+1=0.\]
Here $\tau$ is a trivialization that satisfies the second bullet point of  Theorem \ref{special-trivalization}. For this trivialization, we have
\[\operatorname{CZ}^\tau(\Gamma_{\mathrm{top}})=\mu(c)=0.\]
Here, $\mu(c)=0$ because every closed geodesic for a flat metric has Morse index equal to $0$.

%

Let $c$ be the geodesic that lifts to $\gamma_{\mathrm{top}}$. Then $c$ passes through the point $p$. Let $T>0$ be the period of $\gamma_{\mathrm{top}}$. There is a natural half-cylinder $\tilde{u}^c_\rho:[0,\infty)\to T^*T^2$ over  $\gamma_{\mathrm{top}}$ defined by 
\[\tilde{u}^c_\rho(s,t):=(c(Tt),\psi(s)c'(Tt)),\]
 where $\psi:[0,\infty)\to \mathbb{R}$ is a choice of smooth function. This half-cylinder is $(i,J_\rho)$-holomorphic when $\psi$ solves the differential equation 
\[\psi'(s)=T\rho(s),\, \text{with} \ \psi'(0)=0.\]
Denote by $u^c_\rho$ the unparametrized half-cylinder associated with $\tilde{u}^c_\rho$. Then $u^c_\rho$ belongs to the moduli space $\mathcal{M}^{J_\rho}_{T^*T^2}(T^2,\Gamma_{\mathrm{top}}, p)$ since the geodesic $c$ passes through the constrained point $p$. By {\cite[Proposition 5.1]{Cieliebak2023}} or {\cite[Lemma 7.2]{Cieliebak:2007aa}},   $\mathcal{M}^{J_\rho}_{T^*T^2}(T^2,\Gamma_{\mathrm{top}}, p)$ consists of a unique element, which is the unparametrized half-cylinder $u^c_\rho$.  Moreover, by Wendl's automatic transversality {\cite[Theorem 1]{Wendl2010}},  $u^c_\rho$ is transversally cut out. Therefore $\# \mathcal{M}^{J_\rho}_{T^*T^2}(T^2,\Gamma_{\mathrm{top}}, p)=\pm1$.

The SFT-admissible almost complex structure $J_\rho$ is cylindrical with respect to $(S^*L,\lambda_{\mathrm{can}}|_{S^*L})$ and $\partial X$ is a $C^1$-small perturbation of $S^*L$. Performing neck-stretching along $(\partial X, \lambda_{\mathrm{can}}|_{\partial X})$,  the half-cylinder $u^c_\rho$ breaks into a half cylinder that belongs to  $\mathcal{M}^{J}_{\widehat{X}}(\gamma_{\mathrm{top}},T^2, p)$ and a cylinder $u$ in $\widehat{T^*T^2\setminus X}$ with a positive  and negative end. Let $\mathcal{M}^{J}_{\widehat{T^*T^2\setminus X},u}$ be the connected component of the moduli space containing $u$. Then 
\[\pm 1=\# \mathcal{M}^{J_\rho}_{T^*T^2}(T^2,\Gamma_{\mathrm{top}}, p)=\#\mathcal{M}^{J}_{\widehat{X}}(\gamma_{\mathrm{top}},T^2, p)\cdot \#\mathcal{M}^{J}_{\widehat{T^*T^2\setminus X},u}.\]
In particular, 
\[\#\mathcal{M}^{J}_{\widehat{X}}(\gamma_{\mathrm{top}},T^2, p)\neq 0.\]
\end{proof}

\subsection{Concluding the proof assuming the Claim \ref{claim2} holds}

By the third bullet point in Claim \ref{claim2},  the negative end of $C_{\mathrm{top}}$  is on the closed Reeb orbit $\gamma_{\mathrm{top}}$ in $\partial X$. Using gluing results in the regular setting \cite{Pardon-Contacthomologyandvirtualfundamentalcycles} and Lemma \ref{non-venishing}, we glue a half cylinder given by Lemma \ref{counthalf} to $C_{\mathrm{top}}$ to get a $J_{\mathrm{glu}}$-holomorphic half-cylinder $\bar{C}_{\mathrm{top}}:([0,\infty)\times \mathbb{R}/\mathbb{Z},\{0\}\times \mathbb{R}/\mathbb{Z}) \to (\widehat{E^\mathrm{ext}_{\mathrm{FR}}}, L)$ whose boundary passes through the point $p$ on $L$ and which has a positive puncture asymptotic to $e_k$. The symplectic area of $\bar{C}_{\mathrm{top}}$ does not depend on the choice of representative of the relative class $[\bar{C}_{\mathrm{top}}]\in H_2(\widehat{E^\mathrm{ext}_{\mathrm{FR}}},L\cup e_k)$, therefore by  (\ref{contra}) we have 
\begin{equation}\label{contra1}
	0<\int_{[0,\infty)\times \mathbb{R}/\mathbb{Z}}\bar{C}_{\mathrm{top}}^*\tilde{\omega}^{\mathrm{up}}_{\mathrm{std}}=\int_{[0,\infty)\times \mathbb{R}/\mathbb{Z}}C_{\mathrm{top}}^*\tilde{\omega}^{\mathrm{up}}_{\mathrm{std}}\leq k\epsilon \operatorname{diagonal}(X_{\Omega}^4).
\end{equation}
Here, we have compactified $\bar{C}_{\mathrm{top}}$ to a half cylinder with boundary on $L$.

Let $\mathcal{M}^{J_{\mathrm{glu}}}_{E^\mathrm{ext}_{\mathrm{FR}},\bar{C}_{\mathrm{top}} }(e_k,p)$ be the connected component of the moduli space containing  $\bar{C}_{\mathrm{top}}$. Suppose $\epsilon<1/k$. Recall that the Lagrangian torus $L$ under consideration is extremal in  $(X_{\Omega}^4,\omega_{\mathrm{std}})$. So the half-cylinder $\bar{C}_{\mathrm{top}}$ has minimal\footnote{A non-trivial holomorphic building that can appear as a result of degeneration in the connected component of the moduli space containing $\bar{C}_{\mathrm{top}}$ requires a symplectic area at least $\operatorname{diagonal}(X_{\Omega}^4)$. Note that any such building contains a non-constant disk with boundary on $L$. Such a disk has symplectic area at least $\operatorname{diagonal}(X_{\Omega}^4)$ by the extremality of $L$.} symplectic area by (\ref{contra1}). This means the count of half-cylinders 
 \[\#\mathcal{M}^{J_{\mathrm{glu}}}_{E^\mathrm{ext}_{\mathrm{FR}},\bar{C}_{\mathrm{top}} }(e_k,p)\]  
is well-defined. Moreover, this count is invariant under compactly supported deformations of the almost complex structure $J_{\mathrm{glu}}$. Also by Lemma \ref{non-venishing}, Lemma \ref{counthalf}, and gluing this count does not vanish. More precisely,  we have
\[\#\mathcal{M}^{J_{\mathrm{glu}}}_{E^\mathrm{ext}_{\mathrm{FR}},\bar{C}_{\mathrm{top}} }(e_k,p)=\underset{\neq 0\text{ by Lemma } \ref{counthalf}}{\#\mathcal{M}^{J}_{\widehat{X}}(\gamma_{\mathrm{top}},T^2, p)} \cdot \underset{>0 \text{ by Lemma }\ref{non-venishing}}{\#\mathcal{M}^{J_{\mathrm{top}}}_{E^\mathrm{ext}_{\mathrm{FR}}\setminus X,C_{\mathrm{top}} }(e_k,\gamma_{\mathrm{top}})} \neq 0.\]
  The conclusion is that the half-cylinder $\bar{C}_{\mathrm{top}}$ survives to exist under deformation of the almost complex structure $J_{\mathrm{glu}}$ supported near the Lagrangian $L$. 

 Next, we deform $J_{\mathrm{glu}}$ near $L$ to a new almost complex structure $J_{\mathrm{def}}$  which is $\omega_{\mathrm{std}}$-compatible and agrees with  $\phi_{*}J_{\mathrm{std}}$ on $\phi(D^{2n}(\delta))$,  i.e., $J_{\mathrm{def}}|_{\phi(D^{2n}(\delta))}=\phi_{*}J_{\mathrm{std}}$. Here $J_{\mathrm{std}}$ denotes the standard complex structure of $\mathbb{C}^n$.

 The $J_{\mathrm{def}}$-holomorphic half-cylinder $\bar{C}_{\mathrm{top}}:([0,\infty)\times \mathbb{R}/\mathbb{Z},\{0\}\times \mathbb{R}/\mathbb{Z}) \to (\widehat{E^\mathrm{ext}_{\mathrm{FR}}}, L)$ yields a proper $J_{\mathrm{std}}$-holomorphic map $\phi^{-1}\circ\bar{C}_{\mathrm{top}}|_{\bar{C}_{\mathrm{top}}^{-1}(\phi(D^{2n}(\delta)))}:\bar{C}_{\mathrm{top}}^{-1}(\phi(D^{2n}(\delta)))\mapsto D^{2n}(\delta)$. By  Lemma \ref{monotoncity1} and (\ref{contra1}), we have
\[
k\epsilon \operatorname{diagonal}(X_{\Omega}^4) \geq \int_{[0,\infty)\times \mathbb{R}/\mathbb{Z}}\bar{C}_{\mathrm{top}}^*\tilde{\omega}^{\mathrm{up}}_{\mathrm{std}}\geq
	 \int_{\bar{C}_{\mathrm{top}}^{-1}(\phi(D^{2n}(\delta)))} (\phi^{-1}\circ\bar{C}_{\mathrm{top}})^*\omega_{\mathrm{std}}\geq \frac{1}{2}\pi  \delta^2.\]
This is a contradiction since we can choose $\epsilon$ arbitrarily small and keep $\delta$ and $k$ fixed. This completes the proof.

\subsection{Proof of Claim \ref{claim2}}\label{proofofclaim}
We complete this task in a sequence of lemmas.
\begin{lemma}
	The holomorphic building $\mathbb{H}$ has no node between its non-constant curve components. In particular, the building does not have a node at the constrained marked point $y$.
\end{lemma}
\begin{proof}
	Since $\mathbb{H}$ has genus zero, any node between non-constant components decomposes the building $\mathbb{H}$ into two pieces $A_1$ and $A_2$. Both are represented by non-constant punctured holomorphic curves (possibly in different levels). The building $\mathbb{H}$ has only one (unpaired) end, which is positively asymptotic to the Reeb orbit $e_k$. Therefore, one of these pieces, say $A_1$,  is topologically a plane, and the other is a sphere. Since $[A_2]=0\in H_2(E^\mathrm{ext}_{\mathrm{FR}},\mathbb{Z})$, this is in contradiction with Lemma \ref{nonconstantholclass}.
\end{proof}
Let $C_{\mathrm{bot}}$ be the non-constant smooth connected curve component of $\mathbb{H}$ in the bottom level $T^*L$ that carries the tangency constraint $\ll \mathcal{T}_D^{k-1}y\gg$. By Theorem \ref{count-positive-punctures}, $C_{\mathrm{bot}}$ has at least $k+1$ positive punctures. Suppose that  $C_{\mathrm{bot}}$ has $m$ positive punctures with $m\geq k+1$. 

The underlying graph of the building $\mathbb{H}$ is a tree since the building has genus zero. There are $m$ edges emanating from the vertex $C_{\mathrm{bot}}$ in the underlying graph. We order these edges from $1,2, \dots, k+1,\dots, m$. Let $C_i$ be the subtree emanating from the vertex $C_{\mathrm{bot}}$ along the $i$th edge. Since the building has only one positive puncture at $e_{k}$, precisely one among $C_1,\dots, C_{k+1},\dots, C_{m}$, say $C_{m}$, is topologically a cylinder with positive end on $e_{k}$. The rest $C_1,\dots, C_{k+1},\dots, C_{m-1}$ are topological planes with curve components in different levels. Note that $C_1,\dots, C_{k+1},\dots, C_{m}$ is a partition of the building into disconnected parts: every curve component of the building other than $C_{\mathrm{bot}}$ belongs exactly to one of $C_1,\dots, C_{k+1},\dots, C_{m}$. Since there are no contractible Reeb orbits on $\partial X$, each $C_i$ must have some curve components in the level $\widehat{E^\mathrm{ext}_{\mathrm{FR}}\setminus X}$. Also, each of the topological planes $C_i$ has no components in the symplectization levels $\mathbb{R}\times \partial E^\mathrm{ext}_{\mathrm{FR}}$ by the maximum principle. This means that each of the topological planes $C_i$ must have a component in the level $\widehat{E^\mathrm{ext}_{\mathrm{FR}}\setminus X}$ which is a punctured $J_{\mathrm{top}}$-holomorphic plane $u_i$ with negative end on $\partial X$.  Thus, we have at least $k$ negatively asymptotically cylindrical planes, say $u_1,\dots, u_k$, in the level $\widehat{E^\mathrm{ext}_{\mathrm{FR}}\setminus X}$ asymptotic to Reeb orbits on $\partial X$.  The  $J_{\mathrm{top}}$-holomorphic planes $u_1,\dots, u_k$ in $\widehat{E^\mathrm{ext}_{\mathrm{FR}}\setminus X}$ can be compactified to get $k$ disks in $\widehat{E^\mathrm{ext}_{\mathrm{FR}}}$ with boundaries on $L$.

Let $u_1,\dots, u_k, \dots,u_q$ be all of the negatively asymptotic $J_{\mathrm{top}}$-holomorphic planes in the level $\widehat{E^\mathrm{ext}_{\mathrm{FR}}\setminus X}$ that contribute to the building $\mathbb{H}$. Let $C_{\mathrm{top}}$ denote the collection of all other curve components of the building that belongs to the level $\widehat{E^\mathrm{ext}_{\mathrm{FR}}\setminus X}$, i.e., $C_{\mathrm{top}}$ is the portion of the building in $\widehat{E^\mathrm{ext}_{\mathrm{FR}}\setminus X}$ minus  $u_1,\dots, u_k, \dots,u_q$. The $C_{\mathrm{top}}$ contains at least one curve component (which belongs to topological cylinder $C_m$) because the building has one positive (unpaired) end on $e_k$ that belongs to $\partial E^\mathrm{ext}_{\mathrm{FR}}$. In particular, $C_{\mathrm{top}}$ is not empty.

The following shows that, if $\epsilon>0$ is sufficiently small, then we have precisely $k$ planes in the level $\widehat{E^\mathrm{ext}_{\mathrm{FR}}\setminus X}$, i.e., $q=k$. In particular, this means $C_{\mathrm{bot}}$ has exactly $k+1$ positive ends.
\begin{lemma}\label{countdisk}
	Suppose $\epsilon< 1/k$. There are precisely $k$ negatively asymptotic $J_{\mathrm{top}}$-holomorphic planes $u_1,\dots, u_k$ in the level $\widehat{E^\mathrm{ext}_{\mathrm{FR}}\setminus X}$, and all of them are somewhere injective and have ends on simple closed Reeb orbits. Moreover,
\[0<\int C_{\mathrm{top}} ^*\tilde{\omega}^{\mathrm{up}}_{\mathrm{std}}<k\epsilon \operatorname{diagonal}(X_{\Omega}^{4}). \]
\end{lemma}
\begin{proof}
Let $u_1,\dots, u_k, \dots,u_q$ be all of the negatively asymptotic $J_{\mathrm{top}}$-holomorphic planes in the level $\widehat{E^\mathrm{ext}_{\mathrm{FR}}\setminus X}$ that contribute to the building.  As explained in the proof of Lemma \ref{contra} as well, by the maximum principle, these $J_{\mathrm{top}}$-holomorphic planes are contained in $\widehat{E^{\mathrm{ext}}_{\mathrm{FR}}\setminus X}\setminus (0,\infty)\times \partial E^\mathrm{ext}_{\mathrm{FR}}$.  Moreover, these planes  can be compactified to $q$ smooth disks in $\widehat{E^\mathrm{ext}_{\mathrm{FR}}}\setminus L$ with boundaries on $L$. 
Since $L$ is extremal in $X^4_\Omega$, we have 

\begin{equation}\label{est}
	\sum_{i=1}^q\int u_i^*\tilde{\omega}^{\mathrm{up}}_{\mathrm{std}}=	\sum_{i=1}^q\int u_i^*\omega_{\mathrm{std}}\geq q \operatorname{diagonal}(X_{\Omega}^4).
\end{equation}
	
The energy $E(\mathbb{H})$ of the building $\mathbb{H}$ is  $\operatorname{\mathcal{\tilde{A}}}(e_{k})\leq k (1+\epsilon)\operatorname{diagonal}(X_{\Omega}^{4})$. Therefore, 
\begin{equation*}
		\sum_{i=1}^{q}\int u^*_i\omega_{\mathrm{std}}+\int C_{\mathrm{top}} ^*\tilde{\omega}^{\mathrm{up}}_{\mathrm{std}}\leq \operatorname{\mathcal{\tilde{A}}}(e_k)\leq k (1+\epsilon)\operatorname{diagonal}(X_{\Omega}^{4}). 
	\end{equation*}
So \[0<\int C_{\mathrm{top}} ^*\tilde{\omega}^{\mathrm{up}}_{\mathrm{std}}\leq  \operatorname{diagonal}(X_{\Omega}^{4})(k(1+\epsilon) -q). \]
By assumption $\epsilon< 1/k$, so
	\[0<\int C_{\mathrm{top}} ^*\tilde{\omega}^{\mathrm{up}}_{\mathrm{std}}< \operatorname{diagonal}(X_{\Omega}^{4})(k+1-q). \]
We must have $q=k$: otherwise we have 
\[0<\int C_{\mathrm{top}} ^*\tilde{\omega}^{\mathrm{up}}_{\mathrm{std}}< 0 \]
which is not possible because $C_{\mathrm{top}}$ contains at least one non-constant smooth connected punctured $J_{\mathrm{top}}$-holomorphic curve.
		
By the same argument, it follows that all of the disks $u_1,u_2,\dots,u_k$ are somewhere injective and have asymptotics on simple closed Reeb orbits (cf. Lemma \ref{excontradiction}).
Moreover, we have the estimate
	\[0<\int C_{\mathrm{top}} ^*\tilde{\omega}^{\mathrm{up}}_{\mathrm{std}}<k\epsilon \operatorname{diagonal}(X_{\Omega}^{4}).\qedhere \]
\end{proof}
The following lemma shows that every curve component in $C_{\mathrm{top}}$ has at least one positive puncture and one negative puncture. Note that, because of the exactness of the symplectic form, $C_{\mathrm{top}}$ cannot have a non-constant curve component without punctures.
\begin{lemma}\label{cy}
Define 
	 \[0<S_+:=\min\bigg\{\int_{\gamma}\lambda_{\mathrm{std}}>0: \text{ $\gamma$ is a closed Reeb orbit on $\partial E(a,b)$}\bigg\},\]
and set
\[M=\min\bigg\{\frac{S_+}{k \operatorname{diagonal}(X_{\Omega}^{4})},\frac{1}{k} \bigg\}.\]

	\begin{itemize}
	\item If $\epsilon< M$, then $C_{\mathrm{top}}$ does not contain a curve component that has only negative ends.
	\item If $\epsilon< M$, $C_{\mathrm{top}}$ does not contain a curve component that has only positive ends.
	\item If $\epsilon$ is sufficiently small, then the symplectization levels $\mathbb{R}\times  \partial E^\mathrm{ext}_{\mathrm{FR}}$ and $\mathbb{R}\times \partial X$ consist of trivial cylinders over closed Reeb orbits.
	\end{itemize}
\end{lemma}
\begin{proof}
	For the first bullet point, suppose there is component $u$ that has only negative ends. By the maximum principle, the  $J_{\mathrm{top}}$-holomorphic curve $u$ is contained in $\widehat{E^{\mathrm{ext}}_{\mathrm{FR}}\setminus X}\setminus (0,\infty)\times \partial E^\mathrm{ext}_{\mathrm{FR}}$. Compactify $u$ to a surface in  $\widehat{E^\mathrm{ext}_{\mathrm{FR}}}\setminus L$ with boundary on $L$. 	Since $L$ is extremal in $X^4_\Omega$, we have 
	\[\int u^*\tilde{\omega}^{\mathrm{up}}_{\mathrm{std}}\geq  \operatorname{diagonal}(X_{\Omega}^4).\]
But, by Lemma \ref{countdisk}, we have 
\[0<\int u^*\tilde{\omega}^{\mathrm{up}}_{\mathrm{std}}\leq \int C_{\mathrm{top}} ^*\tilde{\omega}^{\mathrm{up}}_{\mathrm{std}}<k\epsilon \operatorname{diagonal}(X_{\Omega}^{4})<\operatorname{diagonal}(X_{\Omega}^{4}). \]
This is a contradiction.

For the second bullet point,  suppose there is a component $u$ with only positive ends, say asymptotic to $\gamma_1,\dots, \gamma_m$. Then by Stokes' theorem
\[\int u^*\tilde{\omega}^{\mathrm{up}}_{\mathrm{std}}=\sum_{i=1}^{m}\int_{\gamma_i}\lambda_{\mathrm{std}}\geq S_+.\]
But by Lemma \ref{countdisk}, we have 
\[0<\int u^*\tilde{\omega}^{\mathrm{up}}_{\mathrm{std}}\leq \int C_{\mathrm{top}} ^*\tilde{\omega}^{\mathrm{up}}_{\mathrm{std}}<k\epsilon \operatorname{diagonal}(X_{\Omega}^{4})<S_+. \]
This is again a contradiction.

For the last bullet point, recall by (\ref{engery2}) we have
\[	E(\mathbb{H}):=\sum_{N=0}^{N_+}\sum_{i=1}^{k_N}\int_{u_i^N}\tilde{ \Omega}_N\leq \operatorname{\mathcal{\tilde{A}}}(e_{k})= k(1+\epsilon) \operatorname{diagonal}(X_{\Omega}^4).\]
This implies that every Reeb orbit that appears as an asymptote of some curve component in the holomorphic building has action at most $\operatorname{\mathcal{\tilde{A}}}(e_{k})$. Also, this estimate implies
\[\sum_{i=1}^{k}\int u^*_i\omega_{\mathrm{std}}+\int C_{\mathrm{top}} ^*\tilde{\omega}^{\mathrm{up}}_{\mathrm{std}}	+\sum_{N=N_+-l+1}^{N_+}\sum_{i=1}^{k_N}\int_{u_i^N}\tilde{ \Omega}_N\leq k(1+\epsilon) \operatorname{diagonal}(X_{\Omega}^4).\]
This leads to 
\[\sum_{N=N_+-l+1}^{N_+}\sum_{i=1}^{k_N}\int_{u_i^N}d\lambda_{\mathrm{std}}\leq k\epsilon \operatorname{diagonal}(X_{\Omega}^4)-\int C_{\mathrm{top}} ^*\tilde{\omega}^{\mathrm{up}}_{\mathrm{std}}.\]
In particular, for every curve component $u$ of the building in $\mathbb{R}\times  \partial E^\mathrm{ext}_{\mathrm{FR}}$ we have 
\[0\leq  \int_{u}d\lambda_{\mathrm{std}}\leq k\epsilon \operatorname{diagonal}(X_{\Omega}^4)-\int C_{\mathrm{top}} ^*\tilde{\omega}^{\mathrm{up}}_{\mathrm{std}}.\]
There is a positive lower bound for the $d\lambda_{\mathrm{std}}$-energy of punctured holomorphic curves in $\mathbb{R}\times  \partial E^\mathrm{ext}_{\mathrm{FR}}$  which are not trivial cylinders and whose asymptotic Reeb orbits have action at most some fixed positive constant. Therefore, if $\epsilon$ is sufficiently small, then $\int_{u}d\lambda_{\mathrm{std}}=0$. This means $u$ must be a trivial cylinder over some closed Reeb orbit. 
\end{proof}

\begin{lemma}
If $\epsilon$ is sufficiently small, then $C_{\mathrm{top}}$ is a smooth connected simple $J_{\mathrm{top}}$-holomorphic cylinder in $\widehat{E^\mathrm{ext}_{\mathrm{FR}}\setminus X}$ with a positive end asymptotic to $e_k$ and negative end asymptotic to a Reeb orbit on $\partial X$.
\end{lemma}
\begin{proof}
By Lemma \ref{cy}, $C_{\mathrm{top}}$ consists of smooth connected $J_{\mathrm{top}}$-holomorphic cylinders each  with a positive and negative end. Moreover, the symplectization levels $\mathbb{R}\times  \partial E^\mathrm{ext}_{\mathrm{FR}}$  of the building are empty. The building has only one positive end, which is on $e_k$, and has no negative ends. Therefore, $C_{\mathrm{top}}$ must be a single smooth connected $J_{\mathrm{top}}$-holomorphic cylinder with positive end asymptotic to $e_k$. 

The cylinder $C_{\mathrm{top}}$ is somewhere injective because it has positive end on $e_{k}$, which is a simple Reeb orbit by  Remark \ref{somewhere-injective}.
\end{proof}

Putting everything together, we conclude that, for sufficiently small $\epsilon$, the building $\mathbb{H}$ has only two levels. The bottom level $T^*L$ contains a single smooth connected punctured sphere $C_{\mathrm{bot}}$ with exactly $k+1$ positive ends and that carries the tangency constraint $\ll \mathcal{T}_D^{k-1}y\gg$ at $y$. Corresponding to the $k+1$ positive ends of $C_{\mathrm{bot}}$, the top level consists of exactly $k$ negatively asymptotic somewhere injective $J_{\mathrm{top}}$ planes $u_1,u_2,\dots,u_k$ and a somewhere injective $J_{\mathrm{top}}$ cylinder $C_{\mathrm{top}}$ in $\widehat{E^\mathrm{ext}_{\mathrm{FR}}\setminus X}$.  Since the building is connected, it has no components other than $C_{\mathrm{bot}}, u_1,u_2,\dots,u_k$, and $C_{\mathrm{top}}$; see Figure \ref{holo2} for an illustration.
\begin{lemma}\label{indezero}
	For generic $J_{\mathrm{bot}}$ and $J_{\mathrm{top}}$, each of components $C_{\mathrm{bot}}, u_1,u_2,\dots,u_k$, and $C_{\mathrm{top}}$ has Fredholm index zero.
\end{lemma}
\begin{proof}
	The curve components of the building in the top level $u_1,u_2,\dots,u_k, C_{\mathrm{top}}$ are somewhere injective, so each has a non-negative index for generic $J_{\mathrm{top}}$. Also, by {\cite{Cieliebak2018}}, $C_{\mathrm{bot}}$ has non-negative zero for generic  $J_{\mathrm{bot}}$. Moreover, the index of the building is zero, so 
\[\underbrace{\operatorname{ind}(C_{\mathrm{bot}})}_{\geq 0}+\underbrace{\operatorname{ind}(C_{\mathrm{top}})}_{\geq 0}+\sum_{i=1}^{k}\underbrace{\operatorname{ind}(u_i)}_{\geq 0}=0.\]
This means we must have $\operatorname{ind}(C_{\mathrm{bot}})=\operatorname{ind}(C_{\mathrm{top}})=\operatorname{ind}(u_1)=\cdots=\operatorname{ind}(u_k)=0$.
\end{proof}
\begin{lemma}
	If	$\epsilon$ is sufficiently small, then the negative end of the
 cylinder $C_{\mathrm{top}}$  is asymptotic to a closed Reeb orbit on $\partial X$, denoted by $\gamma_{\mathrm{top}}$,  whose projection to $L$ is a closed geodesic passing through the point $p$. 
\end{lemma}
\begin{proof}
Recall from Section \ref{perturbsurface} that every closed Reeb orbit on $\partial X$ that appears in the building is a critical point of the Morse function $f$, and the closed Reeb orbits at which $f$ attains its maximum project to geodesics on $L$ passing through $p$. So it is enough to prove that $\gamma_\mathrm{top}$ as a critical point of $f$ has maximal Morse index, i.e.,
\[1=\operatorname{index}_{\gamma_\mathrm{top}}(f)\in \{0,1\}.\]
We prove that the closed Reeb orbit $\gamma_\mathrm{top}$ corresponds to a closed geodesic $c_\mathrm{top}$ of Morse index, denoted by $\mu(c_\mathrm{top})$, equal to $1$. This is enough because then we can find a trivialization $\tau$ such 
\[1=\mu(c_\mathrm{top})=\operatorname{CZ}^\tau(\gamma_\mathrm{top})=\operatorname{index}_{\gamma_\mathrm{top}}(f)\in \{0,1\}\]
by Equation (\ref{ind-f}). 

Choose a symplectic trivialization  $\tau$ of $C_{\mathrm{bot}}^*TT^*L$. The Fredholm index of $C_{\mathrm{bot}}$ (cf.\cite[Proposition 3.1]{Cieliebak2018}), taking the constraint $\ll \mathcal{T}_D^{k-1}y\gg$ into consideration, is given by
\[\operatorname{ind}(C_{\mathrm{bot}})=(k+1)-2+\sum_{i=1}^{k}\operatorname{CZ^{\tau}}(\gamma_i)+\operatorname{CZ^{\tau}}(\gamma_\mathrm{top})-4+2-2(k-1).\]
After simplifying and using $\operatorname{ind}(C_{\mathrm{bot}})=0$ from Lemma \ref{indezero}  gives
\begin{equation}\label{indexfinal}
	k+1=\sum_{i=1}^{k}\operatorname{CZ^{\tau}}(\gamma_i)+\operatorname{CZ^{\tau}}(\gamma_\mathrm{top}).
\end{equation}   
On the other hand, Theorem \ref{special-trivalization} implies that	
\[\sum_{i=1}^{k}\operatorname{CZ^{\tau}}(\gamma_i)+\operatorname{CZ^{\tau}}(\gamma_\mathrm{top})+\sum_{i=1}^{k}\mu^{\kappa(\tau)}(c_i)+\mu^{\kappa(\tau)}(c_\mathrm{top})=\sum_{i=1}^{k}\mu(c_i)+\mu(c_{\mathrm{top}}),\]
where $\mu(c_i)$ is the Morse index of the geodesic $c_i$ corresponding to $\gamma_i$ and  $\mu^{\kappa(\tau)}(c_i)$ denotes its Maslov index. Since $C_{\mathrm{bot}}$ is a null-homology of the cycles $\gamma_1,\dots,\gamma_{k}, \gamma_{\mathrm{top}}$, we have \[\sum_{i=1}^{k}\mu^{\kappa(\tau)}(c_i)+\mu^{\kappa(\tau)}(c_\mathrm{top})=0.\]
This implies
\[\sum_{i=1}^{k}\operatorname{CZ^{\tau}}(\gamma_i)+\operatorname{CZ^{\tau}}(\gamma_\mathrm{top})=\sum_{i=1}^{k}\mu(c_i)+\mu(c_{\mathrm{top}}).\]
So Equation (\ref{indexfinal1}) implies
\begin{equation}\label{indexfinal1}
	k+1=\sum_{i=1}^{k}\mu(c_i)+\mu(c_{\mathrm{top}}).
\end{equation} 
Since for all $i\in \{1,\dots,k,\text{top}\}$ we have
\[0\leq  \mu(c_i)\leq 1,\]
we must also have
\[\mu(c_i)= 1.\]
for all $i=\infty, 1,\dots, k, \text{top}$.

Again choose trivializations $\tau$ along the ends $\gamma_1,\dots, \gamma_k,\gamma_{\mathrm{top}}$ that satisfy the second bullet point of Theorem \ref{special-trivalization}. In this trivialization we have 
\begin{equation}\label{ind}
\operatorname{CZ}^\tau(\gamma_i)=\mu(c_i)= 1
\end{equation}
for all $i=\infty,1,\dots,k$. In particular,  this means $\gamma_{\mathrm{top}}$ is an elliptic Reeb orbit. By Equation (\ref{ind-f}), we have 
\[1=\operatorname{CZ}^\tau(\gamma_{\mathrm{top}})=\operatorname{index}_{\gamma_{\mathrm{top}}}(f_T).\qedhere\]
\end{proof}

We believe that Theorem \ref{extremal-lag-toric} holds for ellipsoids of all dimensions. Note that the rounding procedure described in Section \ref{roudningfullyconvex} and Theorem \ref{GH} hold for convex toric domains of all dimensions. So it's enough to prove  Theorem \ref{puctured-dsik} for higher-dimensional convex toric domains. 
\begin{conjecture}\label{generalizedconjecture}
For all $0<a_1\leq a_2\leq\dots\leq a_n< \infty$, every extremal Lagrangian torus in $(E^{2n}(a_1,a_2,\dots, a_n),\omega_{\mathrm{std}})$  is entirly contained in the boundary $\partial E^{2n}(a_1,a_2,\dots, a_n)$. 
\end{conjecture}

\AtEndDocument{
\bibliographystyle{alpha}
\bibliography{extremallagrangian}
\Addresses}
\end{document}